\documentclass[12pt]{article}
\usepackage{url,color}

\usepackage[margin=0.75 in]{geometry}
\usepackage{amsfonts,bm}
\usepackage{graphicx}
\usepackage{amsmath}
\usepackage{setspace}
\usepackage{float}
\usepackage{subfigure,color,epsfig,multirow,bbm,graphicx,graphics, sidecap}
\usepackage{tabularx,ulem}
\usepackage{algorithm} 
\usepackage{algpseudocode}
\usepackage{xcolor}
\usepackage[numbers]{natbib}
\usepackage[T1]{fontenc}
\usepackage{authblk}
\usepackage{amsthm}

\fontencoding{OT1}
\fontfamily{ptm}  
\fontseries{m}    
\fontshape{n}     
\fontsize{11}{11}
\selectfont

\newcommand{\x}{ {\mathbf x}}

\newcommand{\rr}{ {\mathbf r}}
\newcommand{\trr}{ {\mathbf {\tilde r}}}
\newcommand{\tx}{ {\mathbf {\tilde  x}}}

\newcommand{\F}{\mathbf F}

\newcommand{\mylist}{\begin{list}{$\bullet$}{ \setlength{\topsep}{0cm}
\setlength{\itemsep}{0cm} \setlength{\parsep}{0cm}}}
\newcommand{\mylistend}{\end{list}}

\newcounter{task}
\newcounter{subtask}
\newcounter{subsubtask}

\newcommand{\bu}{\bf{u}}

\newcommand{\xti}{{\bf x}_{t_i}}
\newcommand{\JF}{{\bf J}_{{\bf F}}}
\newcommand{\JTF}{{\bf J}_{{\tilde {\bf F}}}}
\newcommand{\txti}{{\tilde {\bf x}}_{t_i}}

\newtheorem{theorem}{Theorem}[section]

\newtheorem{lemma}[theorem]{Lemma}

\makeatletter
\newcommand{\mathleft}{\@fleqntrue\@mathmargin\parindent}
\newcommand{\mathcenter}{\@fleqnfalse}
\makeatother

\renewcommand\bibname{References}

\begin{document}

\thispagestyle{empty}
\setcounter{page}{0}

\begin{Huge}
\begin{center}
Computer Science Technical Report CSTR-{19/2015} \\
\today
\end{center}
\end{Huge}
\vfil
\begin{huge}
\begin{center}
{\tt R. \c{S}tef\u{a}nescu, A. Sandu}
\end{center}
\end{huge}

\vfil
\begin{huge}
\begin{it}
\begin{center}
``{\tt Efficient approximation of sparse Jacobians for time-implicit reduced order models}''
\end{center}
\end{it}
\end{huge}
\vfil

\begin{large}
\begin{center}
Computational Science Laboratory \\
Computer Science Department \\
Virginia Polytechnic Institute and State University \\
Blacksburg, VA 24060 \\
Phone: (540)-231-2193 \\
Fax: (540)-231-6075 \\
Email: \url{sandu@cs.vt.edu} \\
Web: \url{http://csl.cs.vt.edu}
\end{center}
\end{large}

\vspace*{1cm}

\begin{tabular}{ccc}
\includegraphics[width=2.5in]{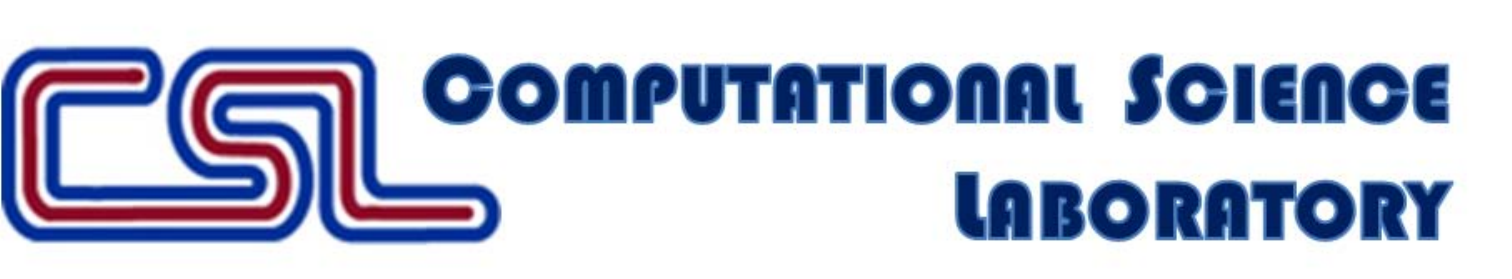}
&\hspace{2.5in}&
\includegraphics[width=2.5in]{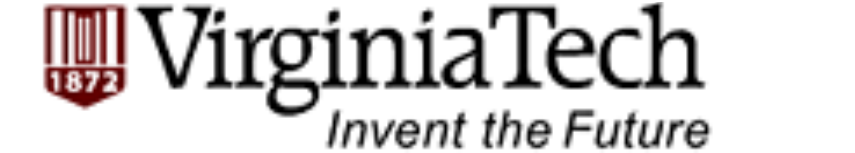} \\
{\bf\em Innovative Computational Solutions} &&\\
\end{tabular}

\newpage

\title{Efficient approximation of sparse Jacobians \\ for time-implicit reduced order models}
\author[1]{R\u{a}zvan \c{S}tef\u{a}nescu \thanks{rstefane@vt.edu}}
\author[1]{Adrian Sandu \thanks{sandu@cs.vt.edu}}
\affil[1]{Computational Science Laboratory, Department of Computer Science, Virginia Polytechnic Institute and State University, Blacksburg, Virginia, USA, 24060}

\date{}
\maketitle

\begin{abstract}
This paper introduces a sparse matrix discrete interpolation method to effectively compute matrix approximations in the reduced order modeling framework. 
The sparse algorithm developed herein relies on the discrete empirical interpolation method and uses only samples of the nonzero entries of the matrix series.
The proposed approach can approximate very large matrices, unlike the current matrix discrete empirical interpolation method which is limited by its 
large computational memory requirements.  The empirical interpolation indexes obtained by the sparse algorithm slightly differ from the ones computed by
the matrix discrete empirical interpolation method as a consequence of the singular vectors round-off errors introduced by the economy or full singular value 
decomposition (SVD) algorithms when applied to the full matrix snapshots. When appropriately padded with
zeros the economy SVD factorization of the nonzero elements of the snapshots matrix is a valid economy SVD for the full snapshots matrix. Numerical experiments are performed with the 1D Burgers and 2D Shallow Water Equations test problems where the quadratic reduced nonlinearities are computed via tensorial calculus.  The sparse matrix approximation strategy is compared against five existing methods for computing reduced Jacobians: a) matrix discrete empirical interpolation method, b) discrete empirical interpolation method, c) tensorial calculus, d) full Jacobian projection onto the reduced basis subspace, and e) directional derivatives of the model along the reduced basis functions.
The sparse matrix method outperforms all other algorithms. The use of traditional matrix discrete empirical interpolation method is not possible for very large instances due to its excessive memory requirements.
\end{abstract}


\section{Introduction}\label{Sec::Introduction}

Modeling and simulation of multi-scale complex physical phenomena has become an indispensable tool across a wide range of disciplines. This usually translates into large-scale systems of coupled partial differential equations, ordinary differential equations or differential algebraic equations which often bear an extremely large computational cost and demand excessive storage resources due to the large-scale, nonlinear nature of high-fidelity models. Since many problems arising in practice are stiff an implicit time integrator is applied to advance the solution and to keep the errors in the results bounded. A sequence of high dimensional linear systems are solved iteratively at each time step when a Newton-like method is employed to obtain the solution of the corresponding system of nonlinear algebraic equations. As a result, the computational complexity of individual simulations can become prohibitive even when high-performance computing resources are available.

Not surprisingly, a lot of attention has been paid to reducing the costs of the complex system solutions by
retaining only those state variables that are consistent with a particular phenomena of interest. Reduced order modeling refers to the development of low-dimensional models that represent important characteristics of a high-dimensional or infinite dimensional dynamical system.

Balanced truncation \citep{Moore1981,Antoulas2009,Sorensen_Antoulas2002, Mullis_Roberts1976} and moment matching \citep{Freund2003,Feldmann_Freund1995,Grimme1997} have been proving successful in developing reduced order models in the case of linear models. Unfortunately balanced truncation does not extend easily for high-order systems, and several grammians approximations were developed leading to methods such as approximate subspace iteration \citep{Bakeretal1996},
least squares approximation \citep{Hodel1991}, Krylov subspace methods \citep{Jaimoukha_Kasenally1994,Gudmundsson_Laub1994} and balanced Proper Orthogonal Decomposition \citep{Willcox02balancedmodel}. Among moment matching  methods we mention partial realization \citep{Gragg_Lindquist1983,Benner_Sokolov2006},
Pad\'{e} approximation \citep{Gragg1972,Gallivan94,Gutknecht1994,VanDooren_1995}
and rational approximation \citep{Bultheel_Moor2000}.

Currently input-independent highly accurate reduced models can be utilized to successfully reproduce the solutions of high-fidelity linear models. However in the case of general nonlinear systems, the transfer function approach is not yet applicable and input-specified semi-empirical methods are usually employed. Recently some encouraging research results using  generalized transfer functions and generalized moment matching have been obtained in \citep{MPIMD12-12} for nonlinear model order reduction but future investigations are required.

Proper Orthogonal Decomposition (POD) \citep{karhunen1946zss,loeve1955pt,hotelling1939acs,lorenz1956eof} is the most prevalent basis selection method for nonlinear
problems. Construction of low relevant manifolds is also the scope of reduced basis method \citep{BMN2004,grepl2005posteriori,patera2007reduced,rozza2008reduced,
Dihlmann_2013,Lieberman_et_al_2010} and dynamic mode decomposition \citep{Rowley2009,Schmid2010,Tissot_et_al_2014,Bistrian_Navon_2014}. Data analysis
using POD and method of snapshots \citep{Sir87a, Sir87b, Sir87c} is conducted to extract basis functions, from experimental data or detailed simulations
of high-dimensional systems, for subsequent use in Galerkin projections that yield low dimensional dynamical models. Unfortunately the POD Galerkin approach
has a major efficiency bottleneck since its nonlinear reduced terms still have to be evaluated on the original state space making the simulation of the
reduced-order system too expensive. There exist several ways to avoid this problem such as the empirical interpolation method
(EIM) \citep{BMN2004} and its discrete variant DEIM \citep{Cha2008,ChaSor2012,ChaSor2010} and best points interpolation method \citep{NPP2008}. Recently the
interpolation selection procedure in DEIM is formulated using a QR factorization with column pivoting \citep{Drmac_Gugercin2015}.
Missing point estimation \citep{Astrid_2008} and Gauss-Newton with approximated tensors \citep{Carlberg2_2011,Carlberg_2012} methods are relying upon the gappy POD technique \citep{Everson_1995} and were developed for
the same reason.

In the case of implicit POD Galerkin reduced order models solved via Newton based methods, during the on-line stage the Jacobian of the nonlinear term has also
a computational cost that depends on the full-order dimension. More precisely, at each iteration, the full Jacobian is evaluated using the reduced order solution and then projected onto the POD manifold to obtain the reduced Jacobian required by Newton solver. One can slightly decrease the computational load of the POD Galerkin method by approximating the reduced Jacobians using the directional derivatives of Newton residuals in the directions of POD basis functions \citep{Vermeulen2006}. For polynomial nonlinearities of order $p$, tensorial calculus transfers several calculations from on-line
to off-line stage and proposes reduced Jacobians computations with a complexity of order of ${\mathcal O}(k^{p+1})$, where $k$ is the dimension of reduced manifold. Such method was applied to obtain implicit reduced order Shallow Water Equations models not only for forward simulation purposes
\citep{Stefanescu_etal_forwardPOD_2014} but also for deriving a reduced order optimization framework \citep{Stefanescu_etal_ROM_DA_2014}.

Recently the use of interpolation methods relying on greedy algorithms became attractive for calculus of reduced order nonlinear terms derivatives.
Chaturantabut \citep{ChaSor2010} proposed a sampling strategy centered on the trajectory of the nonlinear terms in order to approximate the reduced Jacobians.
An extension for nonlinear problems that do not have componentwise dependence on the state has been introduced in \citep{Zhou_2012}. More accurate methods 
directly sample entries of the discrete Jacobians in addition to the nonlinear function. For example, based on EIM, Tonn (2011) \cite{Tonn_2011} developed
Multi-Component Empirical Interpolation Method for deriving affine approximations for continuous vector valued functions and \citet{Wirtz2012}
introduced matrix DEIM (MDEIM) approach to approximate the Jacobian of a nonlinear function to obtain aposteriori error estimates of DEIM reduced 
nonlinear dynamical system. In the context of the finite element method, an unassembled variant of DEIM was developed \citep{antil2014application}, 
\citep{dedden2012model}, \citep{tiso2013discrete} and it can be used for approximation of sparse
Jacobian matrices arising from element-wise assembly.

This paper introduces the sparse matrix discrete empirical interpolation method (SMDEIM) to construct fast and accurate approximation for sparse parametric matrices such as time dependent Jacobians. SMDEIM is a sparse variant of MDEIM approximation method  and relies on the greedy algorithm introduced in
\citep{Cha2008} for computing approximations of the
nonlinear functions. The proposed sparse algorithm utilizes samples of the nonzero entries of the matrix series and the output discrete interpolation indexes
slightly differ from the ones obtained using full matrix snapshots.The differences are a consequence of the singular vectors round-off errors introduced
by the economy or full SVD algorithms during the factorization of the full matrix snapshots. We proved that the economy SVD of the SMDEIM snapshots matrix when
appropriately padded with zeros is a valid thin SVD for the MDEIM snapshots matrix. Now in contrast with MDEIM method, we apply the DEIM algorithm directly
to the dense singular vectors and not their extended variant padded with zeros. Since DEIM algorithm selects only interpolation indexes corresponding to non empty
rows, the output is similar but less computational demanding. The computational complexity of SMDEIM depends on the number of nonzero
elements of the parametric matrices which is typically $\mathcal O(n)$, in contrast to MDEIM where the snapshots contain $n^2$ elements. As an application we
integrate the SMDEIM approach with the reduced order modeling framework to deliver fast time implicit surrogate models.

The corresponding reduced order model is compared against the ROM versions obtained via MDEIM and other four type of approaches already existing
in the literature for 1D Burgers and 2D Swallow Water Equations models. All the surrogate models employ tensorial calculus discussed in \citep{Stefanescu_etal_forwardPOD_2014}
to approximate the reduced nonlinearities so the proposed ROMs differ only in the way they compute the reduced derivatives. The on-line stages of MDEIM and sparse
version reduced order models have the same computational complexities. For the
off-line stage, the sparse version masively decreases the computational cost required by MDEIM reduced order model thus making it practically attractive.
For a small number of DEIM indexes, the MDEIM reduced Jacobians and their sparse variant are as accurate as ones obtained by Galerkin projections and tensorial calculus.

The paper is organized as follows. Section \ref{Sec::Greedy_alg_Jacobians} describes the greedy algorithm introduced in \citep{Cha2008,ChaSor2010} for
nonlinear functions approximations and extensions for their Jacobians computations. Section \ref{sec::SMDEIM} introduces the novel
SMDEIM methodology. Section \ref{sec:ROM} presents the reduced order modeling framework focusing on proper orthogonal decomposition method while integrating
both MDEIM and SMDEIM approaches for reduced Jacobian computations purposes. Other strategies for computing reduced order derivatives are also described.
Section \ref{sec:Numerical_resuls} discusses the discrete 1D Burgers and 2D Swallow Water Equations models as well as the results of numerical experiments of the aforementioned
reduced order models. Conclusions are drawn in Section \ref{sec:Conclusions}.

\section{Greedy algorithms for Jacobian approximations}\label{Sec::Greedy_alg_Jacobians}

An important question in many applications is the following: given a matrix $A$ find an approximation that satisfies certain properties. For example, one may be interested in finding a reliable approximation of $A$ by a matrix of lower rank. The singular value decomposition (SVD) is known to provide the best such approximation for any given fixed rank. In our case we are particular concerned for specific approximation matrices whose structures can be exploit in the framework of reduced order modeling for efficient on-line computation of reduced Jacobian matrices. Unfortunately a SVD matrix approximation only is not able to provide the computational complexity reduction expected by an implicit reduced order model for reduced derivatives calculations. However in combination with a greedy technique the desired matrix structure is obtained and bellow we describe two methodologies already existing in the literature. The new sparse matrix DEIM approach developed in this work seeks to overcome the deficiencies of the currently available approximations.

\subsection{Discrete Empirical Interpolation Method for approximation of nonlinear functions}\label{sub::DEIM_nf}

DEIM is a discrete variation of the Empirical Interpolation method proposed by \citet{BMN2004} which provides an efficient way to approximate
nonlinear vector valued functions. The application was suggested and analyzed by Chaturantabut and Sorensen in \citep{Cha2008,ChaSor2010,ChaSor2012}.

Let  ${\bf F}:D\rightarrow \mathbb{R}^n,~D\subset \mathbb{R}^n$ be a nonlinear function. If $V = \{{\bf v}_l\}_{l=1}^m$, ${\bf v}_l \in \mathbb{R}^n$, is a linearly independent set, for $ m\leq n$, then for $\tau\in D$, the DEIM approximation of order $m$ for ${\bf F}(\tau)$ in the space spanned by $\{{\bf v}_l\}_{l=1}^m$ is given by
\begin{equation}\label{eq8}
{\bf F}(\tau)\approx V{\bf c}(\tau),~V\in\mathbb{R}^{n\times m},~{\bf c}(\tau) \in \mathbb{R}^m.
\end{equation}
The basis $V$ can be constructed effectively by applying the POD method on the nonlinear snapshots ${\bf F}(\tau^{t_i}),~\tau^{t_i}\in D~(~\tau\textrm{ may be a function defined from } [0, T] \rightarrow D,\textrm{ and }\tau^{t_i} \textrm{ is the }$ value of  $\tau \textrm{ evaluated at }t_i) ,~i=1,..,n_s,~n_s>0$. Next, interpolation is used to determine the coefficient vector ${\bf c}(\tau)$ by selecting $m$ rows $\rho_1,..,\rho_m,~\rho_i\in \mathbb{N}^{*}$,  of the overdetermined linear system (\ref{eq8}) to form a $m-$by$-m$ linear system
$$P^T\,V\,{\bf c}(\tau)=P^T\,{\bf F}(\tau),$$
where $P=[{\bf e}_{\rho_1},..,{\bf e}_{\rho_m}]\in \mathbb{R}^{n\times m}$, ${\bf e}_{\rho_i}=[0,..0,\underbrace{1}_{\rho_i},0,..,0]^T\in\mathbb{R}^n.$
The DEIM approximation of ${\bf F}(\tau)\in \mathbb{R}^n$ becomes
\begin{equation}
 {\bf F}(\tau)\approx V\, (P^TV)^{-1}\, P^T\, {\bf F}(\tau).
 \label{DEIM_on_scalar_function}
\end{equation}

Now the only unknowns that need to be specified are the indexes $\rho_1,\rho_2,...,\rho_m$ or the matrix $P$  whose dimensions are $n \times m$. These are determined by Algorithm \ref{alg::DEIM}.

\begin{algorithm}
{\bf INPUT}: $\{{\bf v}_l\}_{l=1}^m\subset\mathbb{R}^n$ (linearly independent): \newline
{\bf OUTPUT}: $\vec\rho=[\rho_1,..,\rho_m]\in\mathbb{N}^m$
 \begin{algorithmic}[1]
 \State $[|\psi|~~ \rho_1]=\max|{\bf v}_1|,{\psi}\in \mathbb{R}$ and $\rho_1$ is the component position of the largest absolute value of $v_1$, with the smallest index taken in case of a tie.
 \State $V=[{\bf v}_1]\in \mathbb{R}^n,~P=[{\bf e}_{\rho_1}]\in \mathbb{R}^n,~\vec\rho=[\rho_1]\in \mathbb{N}.$
 \For{$\ell=2,..,m$}

      \State Solve $(P^TV){ \bf c}=P^T{\bf v}_\ell \textrm{  for } {\bf c}\in \mathbb{R}^{\ell-1};~V,P\in\mathbb{R}^{n\times(\ell-1)}.$
      \State ${\bf r}={\bf v}_\ell-V{{\bf c}},~{\bf r}\in \mathbb{R}^n.$
      \State $[|\psi|~~ \rho_\ell]=\max\{|{\bf r}|\}.$
      \State $V \leftarrow [U~~{\bf u}_\ell],~P \leftarrow [P~~{\bf e}_{\rho_\ell}],~\vec\rho \leftarrow\left[\begin{array}{cc}
           \vec\rho\\
           \rho_\ell\\
           \end{array}\right].$
   \EndFor
 \end{algorithmic}
 \caption{Computation of DEIM Interpolation Indexes}
 \label{alg::DEIM}
\end{algorithm}

The DEIM procedure inductively constructs a set of indexes from the input POD basis $\{{\bf v}_l\}_{l=1}^m\subset\mathbb{R}^n$. Initially the algorithm searches
for the largest value of the first POD basis $|{\bf v}_1|$ and the corresponding index represents the first DEIM interpolation index $\rho_1\in\{1,2,..,n\}$.
The remaining interpolation indexes $\rho_l,~l=2,3..,m$ are selected so that each of them corresponds to the entry of the largest magnitude of $|{\bf r}|$.
The vector ${\bf r}$ can be viewed as the residual or the error between the input basis ${\bf v}_l,~l=2,3..,m$ and its approximation $V{\bf c}$ from interpolating the basis
$\{{\bf v}_1,{\bf v}_2,..,{\bf v}_{l-1}\}$ at the indexes ${\rho_1},{\rho_2},..,{\rho_{l-1}}$. The linear independence of the input basis $\{{\bf v}_l\}_{l=1}^m$ guarantees that,
in each iteration, ${\bf r}$ is a nonzero vector and the output indexes $\{\rho_i\}_{i=1}^m$ are not repeating \cite{ChaSor2010}.

An error bound for the DEIM approximation is provided in Chaturantabut and Sorensen \citep{Cha2008,ChaSor2012}. An example of DEIM approximation of a highly nonlinear function defined
on a discrete 1D spatial domain can be found in \cite{ChaSor2010}, underlying the DEIM efficiency.

Based on the greedy algorithm detailed above we will describe three approaches for approximating the Jacobian of ${\bf F}$ denoted by ${\bf J}_{\bf F}(\tau) \in \mathbb{R}^{n \times n}$. The first two techniques DEIM and matrix DEIM were introduced
in \cite{ChaSor2010} and \cite{Wirtz2012} while the sparse matrix DEIM algorithm is introduced here for the first time. While DEIM utilizes function samples,
the matrix DEIM and sparse matrix DEIM
are directly sampling entries of the discrete operator, i.e. the Jacobian of ${\bf F}$.

\subsection{Discrete Empirical Interpolation Method for approximating Jacobians of nonlinear functions}\label{sub::DEIM_jnf}

This method was suggested in \cite{ChaSor2010} for computing Jacobians of nonlinear functions
in the framework of reduced order modeling. It proposes a sampling strategy centered on the trajectory of the nonlinear functions and makes use of DEIM approximation formula \eqref{DEIM_on_scalar_function}. Extensions for nonlinear problems that do not have componentwise dependence on the state have also investigated in \cite{ChaSor2010,Zhou_2012}. The DEIM Jacobian approximation is given by
\begin{equation}
{\bf J}_{\bf F}(\tau) \approx V\, (P^TV)^{-1}\,P^T\,{\bf J}_{\bf F}(\tau),\quad \tau =\{\tau^1,..,\tau^{n_s}\},\quad {\bf J}_{\bf F}(\tau^i) \in \mathbb{R}^{n \times n}, i=1,..,n_s,
\label{eq::Jacobian_DEIM_approx}
\end{equation}
where $V$ is constructed by extracting the left singular vectors of nonlinear snapshots matrix ${\bf F}(\tau^i),i=1,..,n_s,~n_s > 0$, while matrix $P$ is
the output of Algorithm \ref{alg::DEIM}.

\subsection{Matrix Discrete Empirical Interpolation Method} \label{sub::MDEIM}

This method applies the greedy technique described in subsection \ref{sub::DEIM_nf} in a different manner than DEIM approach discussed in the previous subsection. Thus instead of using snapshots of the nonlinear function ${\bf F}$, Jacobian snapshots written as vectors feed the Algorithm
\ref{alg::DEIM} providing a direct approximation of ${\bf J}_{\bf F}(\tau)$. It was introduced by \citet{Wirtz2012} to develop an efficient
a-posteriori error estimation for POD-DEIM reduced nonlinear dynamical systems. In particular the matrix DEIM approach was employed for an efficient
off-line/on-line approximation of logarithmic Lipschitz constants of linear functions. A similar idea named ''Multi-Component EIM`` has been formulated
in \citet{Tonn_2011} to derive affine approximations for continuous vector valued functions.

First, we define the transformation ${A} \rightarrow T[{A}]$ which maps the entries of a matrix ${A} \in \mathbb{R}^{n \times n}$ column-wise into the vector $T[{A}] \in \mathbb{R}^{n^2 \times 1}$. Next, we compute the economy or thin SVD \cite{Trefethen_Bau_1997}  of the matrix of snapshots
$\big[T[{\bf J}_{\bf F}(\tau^{i})]\big]_{i=1,..,n_s} \in \mathbb{R}^{n^2 \times n_s},~\tau^{i}\in D,~i=1,..,n_s,~n_s>0$ and the left singular vectors $V_{J} \in {\mathbb{R}^{n^2\times n_s}}$ are given by
\begin{equation}\label{eq::SVD_MDEIM}
  \big[ \, T[{\bf J}_{\bf F}(\tau^{1})], \dots, T[{\bf J}_{\bf F}(\tau^{n_s})]\, \big] = V_{J}~\Sigma_{J}~W_{J}^T,
 \end{equation}
where $\Sigma_{J} \in {\mathbb{R}^{n_s\times n_s}}$ is a diagonal matrix containing the singular values of $ \big[T[{\bf J}_{\bf F}(\tau^{i})]\big]_{i=1,..,n_s}$
and $W_{J} \in {\mathbb{R}^{n_s\times n_s}}$ gathers the right singular vectors of the matrix of snapshots.

Let $m\leq n_s$ be the number of DEIM index points and $P_{J} \in {\mathbb{R}^{n^2\times n_s}}$ represents
the index points transformation matrix obtained by application of DEIM Algorithm \ref{alg::DEIM} to the left singular vector matrix $V_{J}$. Then the $m^{\rm th}$
order matrix DEIM approximation of ${\bf J}_{\bf F}(\tau),~\tau =\{\tau^1,..,\tau^{n_s}\}$ is
\begin{equation}\label{eq::MDEIM}
  {\bf J}_{\bf F}(\tau) \approx T^{-1}\left[\,V_m\, \left(P_m^TV_m\right)^{-1}\, P_m^T\, T[{\bf J}_{\bf F}(\tau)]\,\right],
 \end{equation}
where $V_m = V_{J}(:,1:m)$, $P_m = P_{J}(:,1:m)$.

However for large values of $n$ the SVD factorization calculation demands increased computational resources and the memory required to store the left singular vectors $V_J$ \ increases substantially thus limiting the algorithm application. This is the case even if a sparse SVD is employed since the output singular vectors are generally not sparse.

\section{Sparse Matrix Discrete Empirical Interpolation Method} \label{sec::SMDEIM}

The most general method to compute the SVD factorization uses two phases. Initially the MDEIM snapshots matrix $\big[T[{\bf J}_{\bf F}(\tau^{i})]\big]_{i=1,..,n_s} \in \mathbb{R}^{n^2 \times n_s}$ is brought
into a bidiagonal form and then the bidiagonal  matrix is diagonalized \cite{Golub_van_Loan_1996,Trefethen_Bau_1997}. The first stage usually employs Golub-Kahan or
Lawson-Hanson-Chan bidiagonalizations while in the second stage a variant of the QR or divide-and-conquer algorithms are applied to generate the diagonal
form. The cumulated computational cost of the SVD applied to the MDEIM snapshots matrix in case $n_s\ll n^2$ is $O(n^2\cdot n_s^2)$ and
singular vectors of $n^2$ size need to be stored in the memory.

Since usually the MDEIM Jacobian snapshots matrix  is sparse we can apply a fill reducing ordering \cite{Davis_et_al_2004a,Davis_et_al_2004b}
before the QR factorization to minimize the number of non-zeros in R or use a profile reduction ordering of R \cite{Hager2002}. Other sparse singular value
decomposition approaches relying on blocked algorithms have been proposed in \cite{Rajamanickam_2009}. The Lanczos subspace iteration
based algorithms implemented in SVDPACK [Berry 1992] and PROPACK [Larsen 1998] are probably the most succesfully aproaches for finding the sparse SVD.
While significantly decreasing the factorization computational cost, the sparse SVD methodologies still require to store singular vectors of size $n^2$. Moreover in case a thin SVD is applied round-off errors
usually spoil the sparsity structure of the singular vectors.

To mitigate this drawback we propose a sparse version of matrix DEIM algorithm which relies on the fact that typically the Jacobians of
large-scale time dependent problems have few nonzero entries and preserve their structure in time. We will build the snapshots matrix containing only the nonzero
elements of the Jacobian matrices thus significantly decreasing the computation cost of the thin SVD factorization and memory requirement for saving the corresponding
singular vectors. Owning to the Jacobian structure we will prove that the thin SVD of the SMDEIM snapshots when appropriately padded with zeros is a valid thin SVD
for the MDEIM snapshots $\big[T[{\bf J}_{\bf F}(\tau^{i})]\big]_{i=1,..,n_s} \in \mathbb{R}^{n^2 \times n_s}$. The algorithm can be also adapted to accommodate Jacobian matrices with structures that vary in time.

For the moment we assume that Jacobian snapshots ${\bf J}_{\bf F}(\tau^{t_i})$ have the same sparsity pattern for all
$\tau^{t_i} \in D,$ $i=1,..,n_s$ as it is the case for the majority of time dependent problems.  Moreover we consider that only $r$ entries of
${\bf J}_{\bf F}(\tau^{t_i}),\quad i=1,..,n_s,$  out of $n^2$ are different from zero. This suggests that MDEIM snapshots matrix
$\big[T[{\bf J}_{\bf F}(\tau^{i})]\big]_{i=1,..,n_s}$ has a rank less than or equal to $\min(r,n_s)$ and contains $n^2 -r$ rows with zero elements only. Usually the
number of nonzero entries $r$ is larger than the number of snapshots $n_s$. Now we can appropriately select the $r$ nonzero rows of MDEIM snapshots matrix by using a truncated identity matrix ${\bar P}\in\mathbb{R}^{r\times n^2}$
\begin{equation}\label{sparse mdeim matrix}
{\bar P}\, \big[T[{\bf J}_{\bf F}(\tau^{i})]\big]_{i=1,..,n_s} \in \mathbb{R}^{r \times n_s}.
\end{equation}

Consequently ${\bar P}={\bar p}_{jl},j=1,..,r;l=1,..,n^2$ represents a linear transformation ${\bar P}: \mathbb{R}^{n^2} \to \mathbb{R}^{r}$ that contains only one element other than zero per each line, i.e. ${\bar p}_{jo_j} = 1,~j=1,..r$, where $o_j\in \{1,2,..,n^2\}$ are the $r$ locations corresponding to the nonzero
elements of $T[{\bf J}_{\bf F}(\tau^{i})] \in \mathbb{R}^{n^2}$. Here and in the subsequent lemma we abuse of notations of ${\bar P}$ and ${\bar P}^T$ to denote both the linear
transformations and their matrices representations, respectively.
\begin{lemma}\label{Lemma1}
 The transpose of matrix ${\bar P}$, ${\bar P}^T: \mathbb{R}^r \to \mathbb{R}^{n^2}$ represents the inverse transformation of ${\bar P}$.

\end{lemma}
\begin{proof} First let us denote by $G = {\bar P}^T {\bar P}$, $G = g_{jl},~j,l=1,..,n^2 \in \mathbb{R}^{n^2 \times n^2}$,
\begin{equation*}
g_{jl} = \sum_{k=1}^r {\bar p}_{kj}{\bar p}_{kl},\quad j,l=1,..,n^2.
\end{equation*}
Since ${\bar p}_{jo_j} = 1,~j=1,..r$, then the only components of $G$ other then zero are $g_{o_jo_j}=1,~j=1,..,r$. Now the elements of
$T[{\bf J}_{\bf F}(\tau^{i})]$
different then zero are located at $o_j,~j=1,..,r$ positions and thus
\begin{equation*}
{\bar P}^TPT[{\bf J}_{\bf F}(\tau^{i})] = T[{\bf J}_{\bf F}(\tau^{i})],\forall i=1,..,n_s.
\end{equation*}
It is easy to show that ${\bar P} \cdot {\bar P}^T \in \mathbb{R}^{r \times r}$ is the identify matrix which completes the proof.
\end{proof}

Next we compute the thin singular value decomposition of the dense snapshots matrix \eqref{sparse mdeim matrix} denoted by SMDEIM snapshots matrix and the left singular vectors
$V_{J_{nz}} \in {\mathbb{R}^{r\times n_s}}$ are given by
\begin{equation}\label{eq::SVD_SMDEIM}
  {\bar P}\big[T[{\bf J}_{\bf F}(\tau^{i})]\big]_{i=1,..,n_s} = V_{J_{nz}}~\Sigma_{J_{nz}}~W_{J_{nz}}^T,
 \end{equation}
where $\Sigma_{J_{nz}} \in {\mathbb{R}^{n_s\times n_s}}$ is the singular values diagonal matrix and $W_{J_{nz}} \in {\mathbb{R}^{n_s\times n_s}}$
collects the right singular vectors of the dense matrix.
\begin{lemma}\label{Lemma2}
 ${\bar P}^TV_{J_{nz}}~\Sigma_{J_{nz}}~W_{J_{nz}}^T$ is a thin SVD representation of the MDEIM snaspshots matrix $\big[T[{\bf J}_{\bf F}(\tau^{i})]\big]_{i=1,..,n_s}$.
\end{lemma}
\begin{proof}
 If we multiply the left hand side of the equation \eqref{eq::SVD_SMDEIM} by ${\bar P}^T$ and make use of the Lemma \ref{Lemma1} we immediatly obtain

\begin{equation}
  \big[T[{\bf J}_{\bf F}(\tau^{i})]\big]_{i=1,..,n_s} = {\bar P}^TV_{J_{nz}}~\Sigma_{J_{nz}}~W_{J_{nz}}^T.
 \end{equation}

Now $\Sigma_{J_{nz}}$ is a diagonal matrix with positive real entries and $W_{J_{nz}}$ describes an orthogonal matrix since both of the matrices are obtained from the
singular value decomposition in \eqref{eq::SVD_SMDEIM}. According to \cite{Trefethen_Bau_1997} we only need to demostrate that ${\bar P}^TV_{J_{nz}}$ has orthonormal
columns in order to complete the proof. This is obvious since for all $j,l=1,..,n_s$
\begin{equation*}
 \langle v_j,v_l\rangle_2 = \langle {\bar v}_j,{\bar v}_l \rangle_2 = \left \{
  \begin{array}{ll}
  1,\quad j=l \\
  0,\quad j \neq l, \\
  \end{array}
\right.
\end{equation*}
where $\langle \cdot,\cdot \rangle_2$ is the Euclidian product and $v_j$ and ${\bar v}_j,j=1,..,n_s$ are the columns of matrices ${\bar P}^TV_{J_{nz}}$ and $V_{J_{nz}}$,
respectively.
\end{proof}

Now instead of applying the DEIM algorithm \ref{alg::DEIM} to the left singular vectors ${\bar P}^TV_{J_{nz}}$ of $\big[T[{\bf J}_{\bf F}(\tau^{i})]\big]_{i=1,..,n_s}$
as in the matrix DEIM approach we propose to use the left singular vectors $V_{J_{nz}}$ of ${\bar P}\big[T[{\bf J}_{\bf F}(\tau^{i})]\big]_{i=1,..,n_s}$ as input for
the DEIM algorithm \ref{alg::DEIM}. The output DEIM indexes match perfectly in both situations, however in the latest case the computational load is decreased leading to the DEIM index points matrix $P_{J_{nz}} \in \mathbb{R}^{r \times n_s}$. By leting $m\leq n_s$ be the number of DEIM index points, we obtain the following $m^{\textrm{th}}$ DEIM approximation
\begin{equation*}
  {\bar P}\big[T[{\bf J}_{\bf F}(\tau^{i})]\big] \approx V_{m_{nz}}\Big(P_{m_{nz}}^T\,V_{m_{nz}}\Big)^{-1}\, P_{m_{nz}}^T\,{\bar P}\big[T\big[{\bf J}_{\bf F}(\tau^i)\big]\big],\forall i=1,..,n_s,
 \end{equation*}
where $V_{m_{nz}} = V_{J_{{nz}}}(:,1:m),~P_{m_{nz}} = P_{J_{nz}}(:,1:m)$.

Now multiplying the left hand side of the above equation with ${\bar P}^T$ and using Lemma \ref{Lemma1} and inverse transformation $T^{-1}$ we obtain the $m^{\rm th}$ order sparse
matrix DEIM approximation of ${\bf J}_{\bf F}(\tau),~\forall \tau=\{\tau^1,..,\tau^{n_s}\}$
\begin{equation}\label{eq::sparseMDEIM}
  {\bf J}_{\bf F}(\tau) \approx T^{-1}\Bigg[{\bar P}^T\bigg[ V_{m_{nz}}\Big(P_{m_{nz}}^TV_{m_{nz}}\Big)^{-1}\, P_{m_{nz}}^T\, {\bar P} \Big[T\big[{\bf J}_{\bf F}(\tau)\big]\Big]\bigg]\Bigg],
 \end{equation}

The following result based on \cite[Lemma 3.2]{ChaSor2010} provides an error bound for the SMDEIM approximation.

\begin{lemma}\label{Lemma3}
 Let ${\bf F} \in \mathbb{R}^n$ be a sparse column-wise vector representation of a matrix with $r>0$ nonzero entries located at $o_j,~j=1,2,..,r,~o_j\in\{1,2,..,n\}$ . Let ${\bar P} \in \mathbb{R}^{r
 \times n}$ be a truncated identity matrix with the nonzero elements ${\bar p}_{j,o_j} = 1,~j=1,..,r$ such that ${\bar P}\,{\bf F} \in \mathbb{R}^r$ comprises only
 the non-zero elements of ${\bf F}$. Let $V_{m_{nz}} = \{{{\bar v}_l}\}_{l=1}^m \in \mathbb{R}^{r\times m},~m>0$ be a collection of orthonormal vectors and
 \begin{equation}\label{eq::sparseMDEIM_approximation}
  {\hat {\bf F}} = {\bar P}^TV_{m_{nz}}\, \Big(P_{m_{nz}}^T\,V_{m_{nz}}\Big)^{-1}\, P_{m_{nz}}^T\,{\bar P}\,{\bf F},
 \end{equation}
 be the sparse matrix DEIM approximation of order $m\ll n$ for ${\bf F}$
with $P_{m_{nz}}= [{\bf e}_{\rho_1},..,{\bf e}_{\rho_m}]\in \mathbb{R}^{r\times m}$ being the output
of DEIM Algorithm \ref{alg::DEIM} having as input the basis $V_{m_{nz}}$. Then the following result holds
 \begin{equation}\label{eq::sparseMDEIM_error_bound}
\| {\bf F} - {\hat {\bf F}} \|_2 \leq  \|\Big(P_{m_{nz}}^T\,V_{m_{nz}}\Big)^{-1} \|_2 \, \|\Big( I - V_{m_{nz}}\,V_{m_{nz}}^T\Big)\, {\bar P}\,{\bf F}\|_2,
 \end{equation}
 where $\|\cdot \|_2$ is the appropriate vector or matrix $2-$ norm.
 \end{lemma}

\begin{proof} Since $P_{m_{nz}} \in \mathbb{R}^{r \times m}$ is the output of the DEIM algorithm \ref{alg::DEIM} with the input basis $V_{m_{nz}}$, then DEIM approximation
of ${\bar P}\,{\bf F}$ in the space spanned by $\{{{\bar v}_l}\}_{l=1}^m$ is:

 \begin{equation}\label{eq::dense_DEIM_error_bound}
\widehat{{\bar P}\,{\bf F}} = V_{m_{nz}}\,\Big(P_{m_{nz}}^T\,V_{m_{nz}}\Big)^{-1}\, P_{m_{nz}}^T\,{\bar P}\,{\bf F}.
 \end{equation}

According to \cite[Lemma 3.2]{ChaSor2010} we have the following bound
\begin{equation}
 \| {\bar P}\,{\bf F}-\widehat{{\bar P}\,{\bf F}}\| _2 \leq \| \Big(P_{m_{nz}}^TV_{m_{nz}}\Big)^{-1} \| _2 \, \| \Big( I - V_{m_{nz}}V_{m_{nz}}^T\Big){\bar P}\,{\bf F}\| _2
\end{equation}

From lemma \ref{Lemma1}  we obtain that
\begin{equation}
 {\bf F} = {\bar P}^T{\bar P}\,{\bf F},
\end{equation}
and by applying equations \eqref{eq::sparseMDEIM_approximation},\eqref{eq::dense_DEIM_error_bound} we get
\begin{equation}
 \| {\bf F} - {\hat {\bf F}} \|_2 = \|  {\bar P}^T {\bar P}\,{\bf F} - {\bar P}^T \widehat{{\bar P}\,{\bf F}}\| _2 \leq \| {\bar P}^T\| _2\, \| {\bar P}\,{\bf F}-\widehat{{\bar P}\,{\bf F}}\| _2.
\end{equation}
From the definition of matrix ${\bar P}$ we have $\| {\bar P}^T\|_2 = 1$ which completes the proof.
\end{proof}

The above approximation can be applied to any algebraic structure that can be reduced to a vector. The danger of directly applying the full or thin SVD to factorize the MDEIM
snapshots matrix consists in generating dense singular vectors due to the round-off errors which subsequently may lead to DEIM indexes pointing to zero element rows
of the Jacobian. The newly proposed version of MDEIM avoids interpolating the Jacobian zeros even for large number of DEIM indexes providing a fast
and efficient approximation of the Jacobian matrix with constant sparse structure.

Formula \eqref{eq::sparseMDEIM} can be adapted to approximate a Jacobian matrix with time variable sparse structure too. One possibility would be to identify
distinctive paterns of the Jacobians and form separate snapshots matrices that lead to different transformations ${\bar P}$ and local in time singular vectors $V_{m_{nz}}$
and DEIM indexes $P_{m_{nz}}$.  For large number of snapshots $n_s$, unsupervised learning tools such as biclustering methods \cite{Madeira_et_al_2004} can be employed to identify the distinctive
structures in data matrices. The other approach consists in selecting $r$ as the largest  number of nonzero locations available at one time step over the entire
time interval and the corresponding sparsity structure will define a new transformation ${\bar P}$.

Next the proposed Jacobian approximation will be tested in the framework of reduced order modeling. Traditionally the reduced Galerkin nonlinearities and
their derivatives computations are considered time consuming since they still depend on the dimension of the full space. The sparse matrix DEIM technique
proposes an efficient off-line stage in comparison with the traditional MDEIM method while maintaining the same computational complexity in the on-line stage.

\section{Reduced order modeling}\label{sec:ROM}

Our plan is to integrate the proposed Jacobian matrix approximations in the reduced order
modeling framework and generate faster off-line/on-line implicit reduced order models for large spatial configurations. We will consider the Proper Orthogonal Decomposition
technique combined with tensorial calculus \cite{Stefanescu_etal_forwardPOD_2014} as the main strategy for deriving the surrogates models and their reduced nonlinearities. For reduced Jacobian computations we propose six different techniques including MDEIM and its sparse variant, DEIM, tensorial, direct projection and directional derivative methods.

\subsection{Proper Orthogonal Decomposition}

Proper Orthogonal Decompositions has been used successfully in numerous applications such as compressible flow \citep{Rowley2004}, computational
fluid dynamics \citep{Kunisch_Volkwein_POD2002,Rowley2005,Willcox02balancedmodel}, and aerodynamics  \citep{Bui-thanh04aerodynamicdata}. It can be thought of as a Galerkin approximation in the spatial variable built from functions corresponding to the solution of the physical system at specified time instances. \citet{Noack2010} proposed a system reduction strategy for Galerkin models of fluid flows leading to dynamic models of lower order based on a partition in slow, dominant and fast modes. \citet{San_Iliescu2013} investigate several closure models for POD reduced order modeling of fluids flows and benchmarked against the fine resolution numerical simulation.

In what follows, we will only work with discrete inner products (Euclidian dot product) though continuous products may be employed too.  Generally, an atmospheric or oceanic model is usually governed by the following discrete dynamical system written in the residual form 
\begin{equation}
\rr^i(\x_{t_i})= \x_{t_i} - \x_{t_{i-1}} - \Delta t \, \F(\x_{t_i}) =0,\quad \text{for } i=1,..,N_t,~~~N_t\in \mathbb{N},
\label{eq:full_residual}
\end{equation}
where $\x_{t_i} \in \mathbb{R}^n$ is the state and $\rr^i : \mathbb{R}^n \to \mathbb{R}^n$ denotes the residual operator at time step $t_i$. Usually a Newton based approach is employed to solve \eqref{eq:full_residual}. Once the discrete solution is obtained, we define
the centering trajectory, shift mode, or mean field correction \citep{NAMTT03} ${\bf \bar x} =\frac{1}{N_t} \sum_{i=1}^{N_t} {\bf x}_{t_i}$. The method of POD consists in choosing a
complete orthonormal basis $U=\{{\bf u}_{i}\},~i=1,..,k;~k>0;~u_i \in \mathbb{R}^n;~U \in \mathbb{R}^{n\times k}$ such that the mean square error between ${\bf x}_{t_i}$ and
POD expansion 
\begin{equation} \label{eq::POD_expansion}
{\bf x}^{POD}_{t_i} = {\bf \bar x} + U{\bf \tilde x}_{t_i},~{\bf \tilde x}_{t_i} \in \mathbb{R}^k 
\end{equation}

is minimized on average. The POD dimension $k \ll n$ is appropriately chosen to capture the dynamics of the flow as described by Algorithm \ref{euclid}.

\begin{algorithm}
 \begin{algorithmic}[1]
 \State Calculate the mean ${\bf \bar x} =\frac{1}{N_t} \sum_{i=1}^{N_t} {\bf x}_{t_i}$.
 \State Set up the correlation matrix $K=[k_{ij}]_{i,j=1,..,N_t}$ where $k_{ij} = \langle{\bf x}_{t_i} - {\bf \bar x},{\bf x}_{t_j} - {\bf \bar x}\rangle_2$.
 \State Compute the eigenvalues $\lambda_1\geq \lambda_2\geq ...\lambda_{N_t}\geq 0$ and the corresponding orthogonal eigenvectors ${\bf v}^1,{\bf v}^2,..,{\bf v}^{N_t} \in \mathbb{R}^{N_t}$ of $K$.
 \State Set ${\bf u}_{i} = \sum_{j=1}^{N_t}{\bf v}^i_j({\bf x}_{t_i} - {\bf \bar x})$, $i=1,..,N_t.$ Then, ${\bf u}_{i}\in \mathbb{R}^n,~i=1,..,N_t$ are normalized to obtain an orthonormal basis.
\State Define $I(m)=\left( \sum_{i=1}^m \lambda_i\right) / \left( \sum_{i=1}^{N_t} \lambda_i\right)$ and choose $k$ such that $ k=\min \{I(m):I(m)\geq \gamma\}$ where $0 \leq \gamma \leq 1$ is the percentage of
total informations captured by the reduced space $\textrm{span}\{{\bf u}_1,{\bf u}_2,...,{\bf u}_k\}.$ Usually $\gamma$ is taken $0.99$.

 \end{algorithmic}
 \caption{POD basis construction}
 \label{euclid}
\end{algorithm}

 Singular value decomposition is another choice for POD basis construction and is less affected by numerical errors than the eigenvalue decomposition. Moreover, the SVD-based POD basis construction is more computational efficient since it decomposes the snapshots matrix whose condition number is the square root of the correlation matrix $K$ used in Algorithm \ref{euclid}. The snapshots matrix should also contain the difference quotients of the state variables in order to achieve optimal pointwise in time rates of convergence with respect to the number of POD basis functions \citep{Kunisch_Volkwein_2001,Iliescu_Wang_2013}.

The Galerkin projection of the full model equations onto the space spanned by the POD basis elements leads to the reduced order model
\begin{equation}
\trr^i(\tx_{t_i}) = \tx_{t_i} - \tx_{t_{i-1}} -\Delta t\, {\tilde{ \bf F}}(\tx_{t_i}),~i=1,..,N_t,
\end{equation}
where $\tx_{t_i}\in \mathbb{R}^{k}$ and ${\tilde{ \bf F}}: \mathbb{R}^k \to  \mathbb{R}^k, {\tilde{ \bf F}}(\tx_{t_i}) = U^T\F({\bf \bar x} + U\tx_{t_i})$ are the reduced state and nonlinear function, respectively and $\trr^i:\mathbb{R}^k \to \mathbb{R}^k$ denotes the reduced residual operator.

The majority of the current reduced discrete schemes available in the literature are usually explicit or at most semi-implicit in time thus avoiding
computing reduced Jacobians. Consequently most of the attempts to increase the efficiency of reduced order models were focused on providing efficient
off-line/on-line decoupled approximations for the nonlinear terms only. In this research we shift the attention toward generating efficient off-line/on-line
approximations of the reduced Jacobians.

\subsection{Reduced Jacobian computations}\label{Sub::Red_Jacobian_comp}

In this subsection we integrate the Jacobian approximations MDEIM and its sparse version discussed in Sections \ref{sub::MDEIM} and \ref{sec::SMDEIM} into reduced
order modeling framework to enable fast and accurate estimations of the reduced Jacobians of the nonlinear terms. Here is the first time when these strategies
are employed for construction of implicit reduced order schemes. \citet{Wirtz2012} introduced MDEIM  to develop an efficient off-line/on-line approximation of
logarithmic Lipschitz constants of linear functions and delivered a-posteriori error estimates of DEIM reduced nonlinear dynamical system. Along with the dense
and sparse matrix DEIM approximation methods we describe the current available techniques used to compute the reduced Jacobians. We will begin with the exact formulations
and then continue with the approximation techniques including the novel MDEIM expressions.

\paragraph{Direct projection method} The simplest approach for calculating the reduced Jacobian ${\bf J}_{\tilde{\bf F}}(\tx_{t_i})$ is to follow the analytical
route. The derivatives of the function ${\tilde{ \bf F}}(\tx_{t_i})$ with respect to the $\tx_{t_i}$ are computed using the chain rule and we get
\begin{equation}
\label{eq::red_jac_nonlin_func}
{\bf J}_{\tilde{\bf F}}(\tx_{t_i}) = U^T\, {\bf J}_{\bf F}({\bf \bar x} + U{\bf \tilde x}_{t_i})\, U,\qquad {\bf J}_{\tilde{\bf F}}(\tx_{t_i}) \in \mathbb{R}^{k \times k},
\end{equation}
where
\begin{equation*}
 {\bf J}_{\tilde{\bf F}}(\tx_{t_i}) = \frac{\partial \tilde{\bf F}}{\partial \tx_{t_i}}(\tx_{t_i}), \qquad  {\bf J}_{{\bf F}}({\bf \bar x} + U{\bf \tilde x}_{t_i}) = \frac{\partial{\bf F}}{\partial \x_{t_i}}({\bf \bar x} + U{\bf \tilde x}_{t_i})
\end{equation*}

There is no off-line cost for this strategy since all the computations are performed on-line. At every time step the full Jacobian $ {\bf J}_{{\bf F}}({\bf \bar x} + U{\bf \tilde x}_{t_i})$ is evaluated using the reduced solution and then projected to the reduced space.
Suppose the complexity for evaluating the $r$ nonzero elements of the full Jacobian is ${\mathcal O}(\alpha(r)),$ where $\alpha$ is some function of $r$,
then the on-line computational complexity of ${\bf J}_{\tilde{\bf F}}(\tx_{t_i})$ is of order of ${\mathcal O}(\alpha(r)) + nk + rk + nk^2 )$ in case the sparse
structure of the Jacobian is expoited. Unfortunately this approach is extremely costly and for large number of mesh points, it leads to slower
reduced order models in comparison with the high fidelity versions.

\paragraph{Tensorial method}
 For $\F$ containing only polynomial nonlinearities, tensorial calculus can be applied to compute ${\bf J}_{{\tilde \F}}(\tx_{t_i})$ and most of the
 required calculations can be translated to the off-line stage, making the on-line phase independent of $n$. To emphasize the tensorial procedure we assume
 that $\F$ presents only a quadratic nonlinearity, thus, at time $t_i$, $\F = \x_{t_i}\odot\x_{t_i},~\x_{t_i} \in \mathbb{R}^n$ and $\odot$ is the componentwise
 operator. The Jacobian of $\F$ calculated at $\x_{t_i}$  is a diagonal matrix

  \begin{equation*}
  \JF(\xti) = \textnormal{diag}\left\{2\x_{t_i}^1,2\x_{t_i}^2,...,2\x_{t_i}^n\right\} \in \mathbb{R}^{n \times n},
  \end{equation*}
where $\xti^k$ represents the $k$ component of the vector $\xti$. Then according to \eqref{eq::red_jac_nonlin_func} we get
\begin{equation*}
\JTF(\txti) = 2\,U^T\, \textnormal{diag}\left\{{\bar {\bf x}} + U(1,:)\txti,{\bar {\bf x}} + U(2,:)\txti,...,{\bar {\bf x}} + U(n,:)\txti\right\}\, U,
\end{equation*}
and subsequently

\begin{equation}
{\bf J}_{{\tilde \F}}(\tx_{t_i}) = {T}_1 + {T}_2^i,
\end{equation}
where $T_1 = 2U^T\big(\underbrace{[{\bf \bar x}\quad {\bf \bar x} \quad \cdot \cdot \quad {\bf \bar x}]}_{k \text{ times}} \odot U \big) \in \mathbb{R}^{k \times k}$
and $T_2^i\in \mathbb{R}^{k \times k}$, $T_2^i(j,l) = \sum_{p=1}^k\tx_{t_i}^p \cdot\big(g^j_{lp} + g^j_{pl}\big),~j,l=1,..,k.$ Tensor $G \in \mathbb{R}^
{k \times k \times k}$ is defined by 
\begin{equation}\label{eq::tensor_g}
g^j_{lp} = U(s,j)\cdot U(s,l) \cdot U(s,p),~j,l,p=1,..,k 
\end{equation}
and $\txti^p$ is the $p$ component of $\txti$. Now $T_1$ and $G$
are computed off-line and the computational complexity is of order ${\mathcal O}(k^3n)$. For the on-line stage ${T}_2^i$ is required and its computational complexity
is ${\mathcal O}(k^3)$. In the case of a polynomial nonlinearity of order $p$ the computational complexity for calculating the reduced Jacobian using tensorial calculus in the
on-line stage is ${\mathcal O}(k^{p+1})$ while the off-line components require ${\mathcal O}(k^{p+1})$ flops. We already applied this strategy to generate
implicit reduced SWE models in \cite{Stefanescu2013,Stefanescu_etal_forwardPOD_2014}.

\paragraph{Directional derivatives method} One can decrease the computational load of the direct projection method by approximating
${\bf J}_{\bf F}({\bf \bar x} + U{\bf \tilde x}_{t_i})$ using the directional derivatives of $\F$ in the directions of POD basis functions $\bu_j = U(:,j)$, for
$j=1,2..,k$
\begin{equation}
\frac{\partial \F}{\partial \x_{t_i}}({\bf \bar x} + U{\bf \tilde x}_{t_i})\bu_j = \left(\begin{array}{c}
           \nabla \F_1 ({\bf \bar x} + U{\bf \tilde x}_{t_i})\\
           \nabla \F_2 ({\bf \bar x} + U{\bf \tilde x}_{t_i})\\
           \cdot \\
           \nabla \F_n ({\bf \bar x} + U{\bf \tilde x}_{t_i})\end{array}\right) \bu_j =
           \left(\begin{array}{c}
           \nabla_{\bu_j} \F_1 ({\bf \bar x} + U{\bf \tilde x}_{t_i})\\
           \nabla_{\bu_j} \F_2 ({\bf \bar x} + U{\bf \tilde x}_{t_i})\\
           \cdot \\
           \nabla_{\bu_j} \F_n ({\bf \bar x} + U{\bf \tilde x}_{t_i})\end{array}\right).
\end{equation}
The vector valued function is written using its scalar components
$\F = ( \F_1, \F_2, \cdot \cdot, \F_n)^T,$
$\F_l:\mathbb{R}^n \to \mathbb{R},~l=1,..n$ and their gradients $\nabla \F_l$ belong to $\mathbb{R}^{1\times n}$. By $\nabla_{\bu_j} \F_l$ we denote
the directional derivative of $\F_l$ in the direction of $\bu_j$.
\begin{equation}
\nabla_{\bu_j}\F_l \approx \frac{\F_l({\bf \bar x} + U{\bf \tilde x}_{t_i}+h\bu_j)-\F_l({\bf \bar x} + U{\bf \tilde x}_{t_i})}{h}. \label{eq::direct_deriv}
\end{equation}
If the cost for evaluating the $n$ scalar components of $F$ is ${\mathcal O}(\alpha(n))$, the computational complexity of this method includes effort in the on-line
stage only and is of order of ${\mathcal O}(k\alpha(n) + nk^2).$ The accuracy level depends on the values of $h$. \citet{Vermeulen2006}
linearized a high-order nonlinear model and their reduced model was obtained using $h=0.01$.
This strategy is a non-intrusive approach allowing for the reduced Jacobian computation by making use of only the high-fidelity function.

\paragraph{DEIM method} \citet{Cha2008} noticed that the Jacobian of a vector valued function can be approximated using the POD/DEIM approximation of the
function itself \eqref{eq::Jacobian_DEIM_approx}, i.e.
\begin{equation}
 {\bf J}_{{\tilde \F}}(\tx_{t_i}) \approx \underbrace{U^TV(P^TV)^{-1}}_{\text{precomputed}:k\times m}\, \underbrace{P^T\JF({\bf \bar x} + U{\bf \tilde x}_{t_i})}_{m \times n}\, \underbrace{U}_{n\times k}.
\label{eq::reduced_Jacob_DEIM}
\end{equation}
Here $m$ is the number of DEIM points and their locations given by $P^T$ are obtained by applying the DEIM algorithm \ref{alg::DEIM} with nonlinear term basis
$V$ as input. Typically the Jacobians of large-scale problems are sparse, and then the approximation \eqref{eq::reduced_Jacob_DEIM} will
be very efficient. Assuming an average of $\mu$ nonzero Jacobian elements per each row, the reduced derivatives calculations during the on-line stage require a computational
complexity of order ${\mathcal O}(\mu k^2m + mk^2 + \alpha(\mu m))$, where ${\mathcal O}(\alpha(\mu m))$ stands for the cost of evaluating the $\mu m$ full Jacobian
entries. More details on the sparse procedure are available in \cite{ChaSor2010}.

The off-line cost of computing DEIM reduced Jacobian arises from the singular value decomposition of the nonlinear term snapshots $\Big({\mathcal O}(n\cdot
n_s^2)\Big)$ , DEIM algorithm for selecting the
interpolation points $\Big({\mathcal O}(m^2\cdot n +
m^3)\Big)$ \citep{Drmac_Gugercin2015} and matrix operations in \eqref{eq::reduced_Jacob_DEIM} $\Big({\mathcal O}(m^3 + n \cdot m^2 + k^2 \cdot m)\Big)$. No additional effort is needed
in this stage in the case the DEIM method is employed to approximate the reduced nonlinear term too.

\paragraph{Matrix DEIM method}

By using MDEIM Jacobian approximation \eqref{eq::MDEIM} inside of \eqref{eq::red_jac_nonlin_func} we obtain the reduced Jacobian approximation
\begin{equation}
{\bf J}_{\tilde{\bf F}}(\tx_{t_i}) \approx U^T\, T^{-1}\left[V_m\, \left(P_m^TV_m\right)^{-1}\, P_m^T\, T[{\bf J}_{\bf F}({\bf \bar x} + U{\bf \tilde x}_{t_i})]\right]\,U.
\end{equation}
From here, one can easily prove that
\begin{equation}
{\bf J}_{\tilde{\bf F}}(\tx_{t_i}) \approx {\tilde T}^{-1}\left[\underbrace{C}_{k^2\times n^2}\cdot \underbrace{V_m\, \left(P_m^TV_m\right)^{-1}}_{n^2\times m}\,\cdot \underbrace{P_m^T\, T[{\bf J}_{\bf F}({\bf \bar x} + U{\bf \tilde x}_{t_i})]}_{m \times 1}\right], \label{eq::red_MDEIM}
\end{equation}
where transformation ${B} \rightarrow {\tilde T}[{B}]$ maps the entries of a matrix ${B} \in \mathbb{R}^{k \times k }$ column-wise into a vector of the size $\mathbb{R}^{k^2}$ and ${\tilde T}^{-1}$ is its inverse. Matrix $C$ is defined bellow
\begin{equation}
C(i,:) = T[{\bf u}_j{\bf u}_l^T]^T,\quad i=1,..,k^2,
\end{equation}
with each $i$ corresponding to a pair of indexes $(j,l),~j,l=1,..,k$. The transformation $T$ is defined in Section \ref{sub::MDEIM}.

Now the complexity of the off-line stage of MDEIM is dominated by the computation of  the matrix $C$ and its product with $V_m\, \left(P_m^TV_m\right)^{-1}$ which requires ${\mathcal O}(n^4\cdot k^2 + n^2 \cdot m \dot k^2 + n^2 \cdot m^2 + m^3)$. Other costs arise from SVD calculation of the nonlinear term snapshots $\Big({\mathcal O}(n^2\cdot
n_s^2)\Big)$ and DEIM algorithm for selecting the
interpolation indexes $\Big({\mathcal O}(m^2\cdot n^2 +
m^3)\Big)$ . The on-line computational complexity of ${\bf J}_{\tilde{\bf F}}(\tx_{t_i})$ is of ${\mathcal O}(k^2\cdot m)$ plus the cost of evaluating $m$ entries of $\JF ({\bf \bar x} + U\tx_{t_i})$ that needs ${\mathcal O}(\alpha(m))$ flops.

\paragraph{Sparse Matrix DEIM method}

The sparse version of the MDEIM method was derived to alleviate the memory requirement of storing jacobian snapshots ${\bf J}_{\F}(\x_{t_i}),~i=1,..,n_s$ of $\mathbb{R}^{n^2}$ size. By applying SMDEIM approximation \eqref{eq::sparseMDEIM} inside equation  \eqref{eq::red_jac_nonlin_func}, the computational complexity for calculating the reduced Jacobians off-line components will depend only on the number of nonzero entries of the high-fidelity Jacobian, dimension of POD basis and number of DEIM indexes,

\begin{equation}
{\bf J}_{\tilde{\bf F}}(\tx_{t_i}) \approx U^T T^{-1}\Bigg[{\bar P}^T\bigg[ V_{m_{nz}}\Big(P_{m_{nz}}^TV_{m_{nz}}\Big)^{-1}\, P_{m_{nz}}^T\, {\bar P} \Big[T\big[{\bf J}_{\bf F}({\bf \bar x} + U{\bf \tilde x}_{t_i})\big]\Big]\bigg]\Bigg]U.
\end{equation}
Next, it follows that
\begin{equation}
{\bf J}_{\tilde{\bf F}}(\tx_{t_i}) \approx {\tilde T}^{-1}\bigg[\underbrace{\tilde C}_{k^2 \times r} \underbrace{V_{m_{nz}}\Big(P_{m_{nz}}^TV_{m_{nz}}\Big)^{-1}}_{r \times m}\, \underbrace{P_{m_{nz}}^T\, {\bar P} \Big[T\big[{\bf J}_{\bf F}({\bf \bar x} + U{\bf \tilde x}_{t_i})\big]}_{m \times 1}\Big]\bigg], \label{eq::red_SMDEIM}
\end{equation}
where the rows of matrix ${\tilde C}$ are computed with the following formula
\begin{equation}
{\tilde C}(i,:) = {\bf u}_j({\bf coef1})\odot{\bf u}_l({\bf coef2}),\quad i=1,..,k^2.
\end{equation}
Each $i$ corresponds to a pair of indexes $(j,l),~j,l=1,..,k$ and vectors ${\bf coef1},~{\bf coef2}\in \mathbb{R}^r$ store the ${\bf J}_{\F}$ matrix column and row indexes where nonzero entries are found.

Now the cost for assemblying the matrix $\tilde C$ and its product with  $ V_{m_{nz}}\Big(P_{m_{nz}}^TV_{m_{nz}}\Big)^{-1}$ is of order of ${\mathcal O}(r\cdot k^2 + r\cdot m \cdot k^2 + r\cdot m^2 + m^3)$. In addition the off-line stage cost includes the computation of singular value decomposition of the  dense nonlinear snapshots ${\mathcal O}(
r\cdot n_s^2)$ and DEIM indexes via Algorithm\ref{alg::DEIM} -- ${\mathcal O}(m^2\cdot r + m^3)$. The on-line cost of ${\bf J}_{\tilde{\bf F}}(\tx_{t_i})$ is the same as in the case of  MDEIM approximation and counts ${\mathcal O}(k^2\cdot m+\alpha(m))$ flops. Tables \ref{table_complexityI},\ref{table_complexityII} resume the findings of this section.

\begingroup
\begin{table}[h]
\begin{center}
\scalebox{0.93}{
\begin{tabular}[h]{|l|l|l|l|l|}
\hline
& & \textrm{MDEIM} & \textrm{SMDEIM} & \textrm{DEIM} \\ \hline
& \textrm{SVD}&$\mathcal{O}(n^2\cdot n_s^2)$ & $\mathcal{O}(r \cdot n_s^2)$ & $\mathcal{O}(n \cdot n_s^2)$ \\ \textrm{Off-line}& \textrm{DEIM indexes} &  $\mathcal{O}(m^2\cdot n^2 + m^3)$ & $\mathcal{O}(m^2\cdot r + m^3)$ & $\mathcal{O}(m^2\cdot n + m^3)$ \\
& \textrm{other} & $\mathcal{O}(n^4\cdot k^2 + n^2 \cdot m \cdot k^2 +$ & ${\mathcal O}(r\cdot k^2 + r\cdot m \cdot k^2 + $ & ${\mathcal O}(m^3 + n \cdot m^2 + $ \\ 
& & $~~~~~n^2 \cdot m^2 + m^3)$ & $~~~~~r\cdot m^2 + m^3)$ & $~~~~~~k^2 \cdot m)$ \\ \hline
\textrm{On-line} & & $\mathcal{O}(k^2\cdot m + \alpha(m))$ & $\mathcal{O}(k^2\cdot m + \alpha(m))$ & $\mathcal{O}(\mu k^2m + mk^2 + \alpha(\mu m))$ \\ \hline
\hline
\end{tabular}}
\end{center}
\caption{Computational complexities of the reduced Jacobians.  $n,~n_s,~r,~\mu,~m$ and $k$ denote the
numbers of independent variables, snapshots of the Jacobian, nonzero entries of a Jacobian snapshot, average nonzeros entries of a Jacobian snapshot per row, DEIM indexes and size of the POD basis.  By $\alpha(p)$ we mean the cost of evaluating p entries of the high-fidelity Jacobian linearized at ${\bf \bar x} + U{\bf \tilde x}_{t_i}.$}
\label{table_complexityI}
\end{table} %
\endgroup%

\begingroup
\begin{table}[h]
\begin{center}
\begin{tabular}[h]{|c|c|c|c|}
\hline
& \textrm{tensorial} & \textrm{Direct proj.} & \textrm{Directional deriv.} \\ \hline 
\textrm{Off-line} &  $\mathcal{O}(k^{p+1}\cdot n)$ & $-$ & $-$ \\ \hline
\textrm{On-line} & $\mathcal{O}(k^{p+1})$ & ${\mathcal O}(\alpha(r)) + nk + rk + nk^2 )$ & ${\mathcal O}(k\alpha(n) + nk^2)$ \\ \hline
\hline
\end{tabular}
\end{center}
\caption{Computational complexities of the reduced Jacobians. The results for tensorial method correspond to a $p^{\textrm{th}}$ polinomial nonlinearity.
By  $\alpha(p)$ we denote the cost for evaluating $p$ entries of the high-fidelity Jacobian at ${\bf \bar x} + U{\bf \tilde x}_{t_i}$ or function $\F$ for direct 
projection method or directional derivative approach, respectively.}\label{table_complexityII}
\end{table} %
\endgroup%

The most expensive on-line stage is proposed by the direct projection technique which is an exact method. By transferring some of the calculations to the off-line stage, the other exact approach, tensorial method becomes competive against the DEIM based techniques but only for quadratic nonlinearities \citep{Stefanescu_etal_forwardPOD_2014}. As an approximation method, one should expect that directional derivative approximation would be faster than the exact methods. This is not the case since its complexity depends on the number of space points $n$. This technique is preferred for situation when the partial derivatives of the function are difficult to compute analitically and only function evaluations are needed. The choice of $h$ in \eqref{eq::direct_deriv} must be careful considered. Among all the proposed techniques, MDEIM and SMDEIM own the fastest on-line stage. DEIM method is much faster in the off-line stage but it is a price paid at the expense of the Jacobian accuracy. DEIM method for the Jacobian approximation guarantees accurate entrees only along the rows indicated by the DEIM points. This is not the case for the MDEIM and SMDEIM formulations which preserve the Jacobian accuracy globally as DEIM method does for the function approximation. The newly introduced sparse version of the MDEIM technique now poses the properties required for large-scale simulations with an off-line cost depending only on the number of Jacobian nonzero elements, POD basis dimension and number of DEIM points.

\section{Numerical Experiments}\label{sec:Numerical_resuls}

We consider two nonlinear test problems, the 1D Burgers and the 2D Shallow Water Equations, and first compare the accuracy of various greedy based Jacobian
approximations described in Sections \ref{Sec::Greedy_alg_Jacobians} and \ref{sec::SMDEIM}. Next, we analyze the performance of the novel reduced
order models obtained by integrating MDEIM and SMDEIM into POD/ROM framework against the available techniques already existing in the literature and discussed in
Section \ref{Sub::Red_Jacobian_comp}. Both the high-fidelity and the corresponding reduced order models make
use of the same time discretization schemes thus avoiding additional errors inside the surrogate models solutions. In all of the experiments the nonlinear
terms are computed using tensorial calculus.

\subsection{One-dimensional Burgers' equation}\label{subsec:Burgers}

\subsubsection{Numerical scheme}\label{subsec:Burgers}

Burgers' equation is a fundamental partial differential equation from fluid mechanics. It occurs in various areas of applied mathematics. For a given
velocity $u$ and viscosity coefficient $\mu$, the model considered here has the following form

\begin{equation}
\frac{\partial u}{\partial t} + u\frac{\partial u}{\partial x} = \mu \frac{\partial^2 u}{\partial x^2}, \quad x \in [0,L],~t \in (0,t_f]\label{eqn:Burgers-pde}.
\end{equation}

We assume Dirichlet homogeneous boundary conditions $u(0,t) = u(L,t) = 0,~t \in (0,t_f]$ and as initial conditions we use a seventh degree polynomial depicted in Figure \ref{Fig::1D-Burgers-IC}

\begin{figure}[h]
  \centering
\includegraphics[scale=0.37]{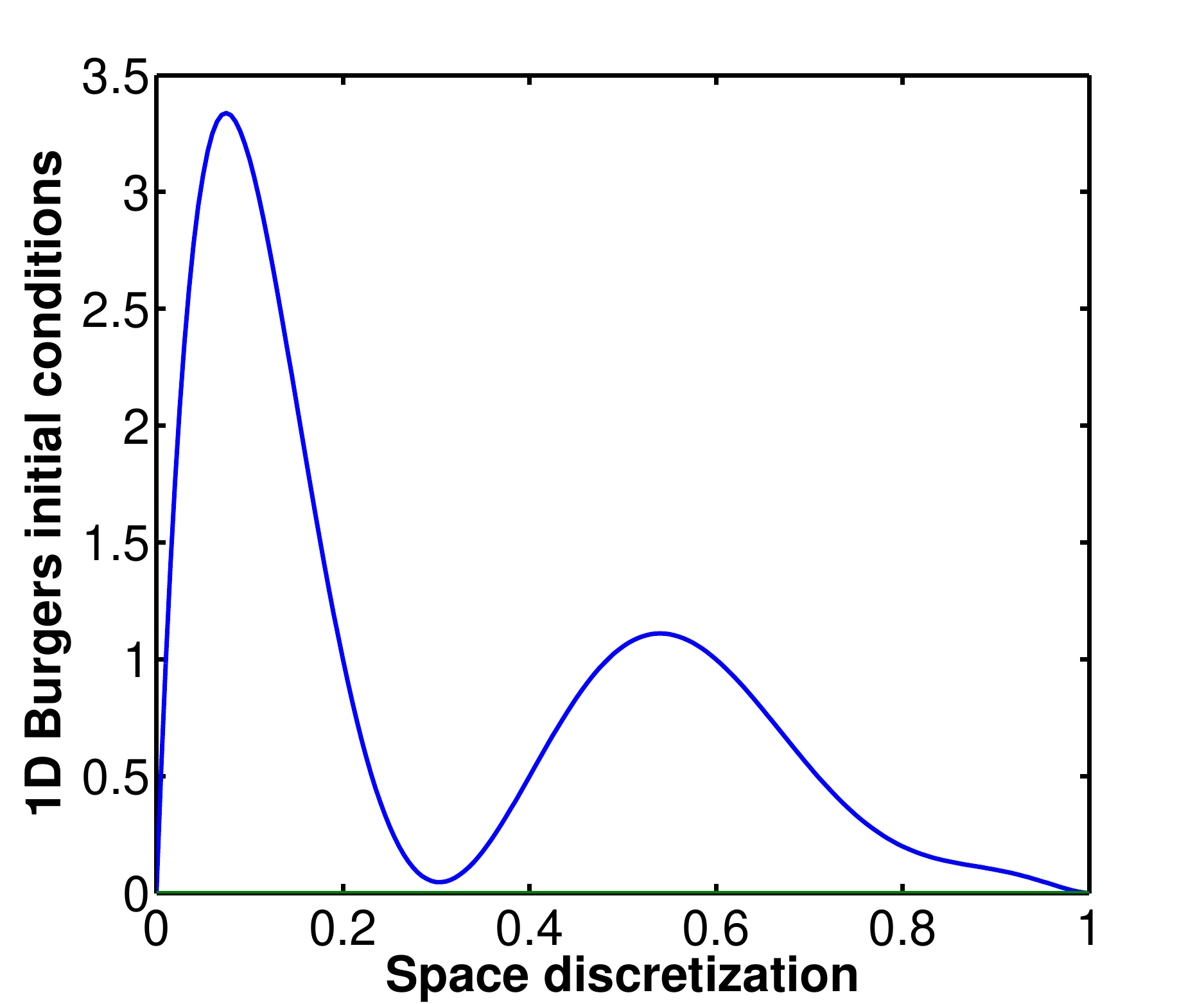}
\caption{Seventh order polynomial used as initial conditions for 1D Burgers model.
\label{Fig::1D-Burgers-IC}}
\end{figure}
Let us introduce a mesh of $n$ equidistant space points on $[0,L]$, with $\Delta x=L/(n-1)$. For the time interval $[0,t_{\rm f}]$ we employ $N_t$ equally distributed points with $\Delta t=t_{\rm f}/(N_t-1)$. By defining the vector of unknown variables of dimension $n-2$ (we eliminate the known boundaries) 
with ${\boldsymbol u}(t_N)\approx [u(x_i,t_N)]_{i=1,2,..,n-2} \in \mathbb{R}^{n-2},~N=1,2,..N_t,$ the semi-discrete version of 1D Burgers model \eqref{eqn:Burgers-pde} is:
\begin{equation}\label{eqn:Burgers-sd}
 {\bf u}'  =  -{\bf u}\odot A_x{\boldsymbol u} + \mu A_{xx}{\boldsymbol u},
\end{equation}
where ${\bf u}'$ denotes the semi-discrete time derivative of ${\bf u}$. $A_x,~A_{xx}\in \mathbb{R}^{(n-2)\times (n-2)}$ are the central difference first-order and second-order space derivatives operators which include also the boundary conditions.

The viscosity parameter is set to $\mu = 0.01$, the final time $t_{\rm f} = 2$ and $L=1$. The backward Euler method is employed for time discretization and it is implemented in Matlab. The nonlinear algebraic systems are solved using Newton-Raphson method and the maximum number of Newton iterations allowed per each time step is set to $50$. The solution is considered accurate enough when the euclidian norm of the residual is less than $10^{-10}$.

\subsubsection{Greedy based Jacobian approximation techniques}

Here we discuss different aspects characterizing the newly introduced SMDEIM method and compare its properties including spectrum of snapshots matrix, locations of
DEIM indexes and approximation accuracy against the ones proposed by MDEIM and DEIM methods using the high-fidelity framework.

In order to generate the greedy based Jacobian approximations we use $401$ time snapshots, i.e. $N_t =401$. Thus we have $401$ model Jacobian (including the linear
and non-linear terms derivatives) snapshots rearranged in vector format and $401$ snapshots of the advection term of the 1D-Burgers model \eqref{eqn:Burgers-sd}. For the SMDEIM method, the matrix of snapshots belongs to $\mathbb{R}^{[3(n-4)+4]\times 401}$ and 
counts only the nonzero entries of the MDEIM snapshots matrix of dimensions $\mathbb{R}^{(n-2)^2\times 401}$. The nonlinear term snapshots matrix contains $(n-2)\times 401$ dense elements. The numerical experiments for this subsection are performed using a mesh of $n=201$ space points.

Figure \ref{Fig::SVD_MDEIM_versions} illustrates the singular values of the SMDEIM and MDEIM snapshots matrices which are very similar. As expected, the computational time for obtaining the SVD decompositions of the SMDEIM matrix is $4$ times smaller than in the case of MDEIM matrix factorization.

\begin{figure}[h]
  \centering
\includegraphics[scale=0.43]{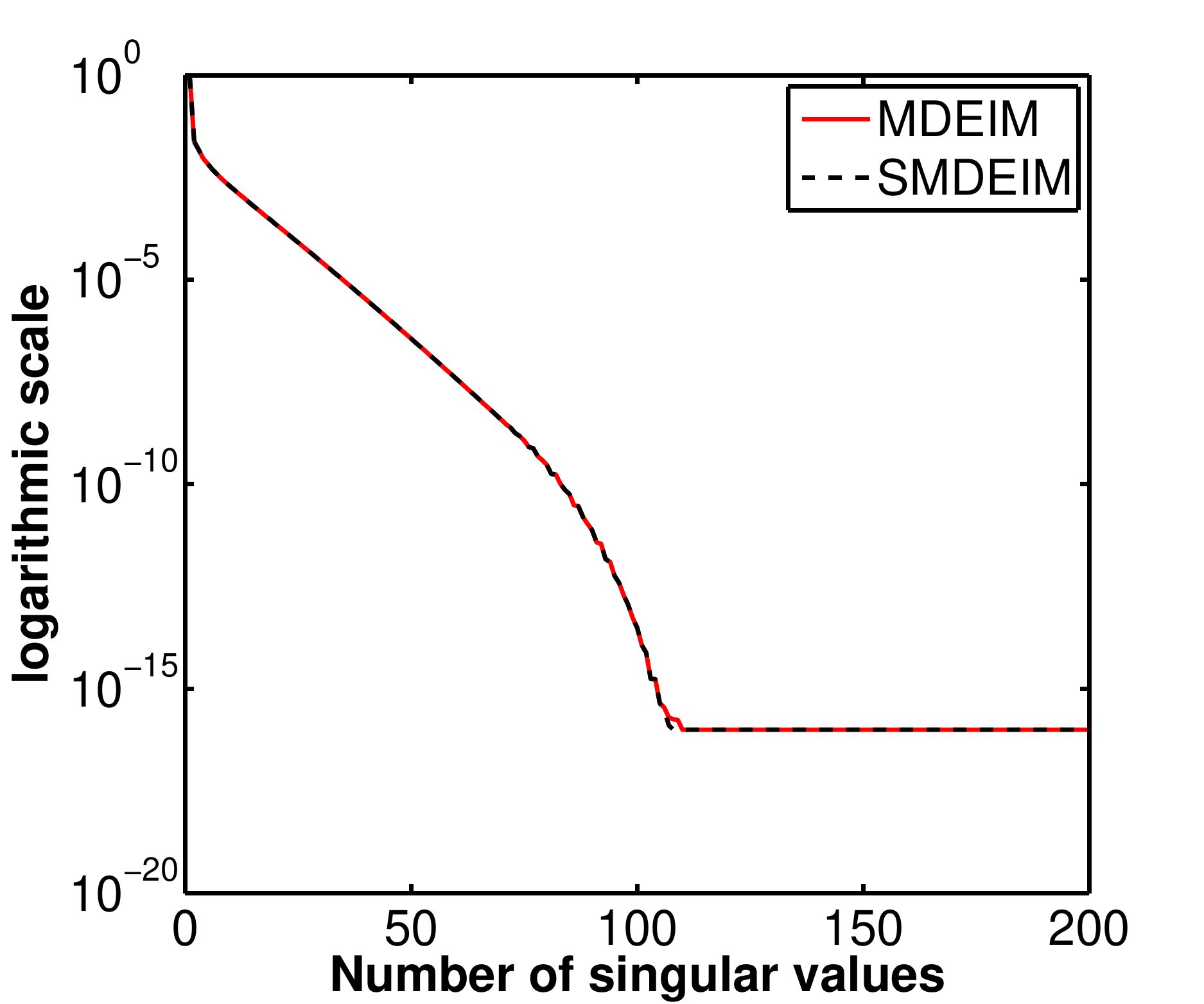}
\caption{Singular values of MDEIM and SMDEIM snapshots matrices 
\label{Fig::SVD_MDEIM_versions}}
\end{figure}
   
Moreover, each singular vector of the MDEIM snapshot matrix has $989$ nonzero entries while a Jacobian snapshot contains only $595$ nonzero elements. 
This is a very well know behaviour, since the singular vectors of a sparse matrix are usually denser. However, in our case the structure of the MDEIM snapshots is regular having entire null rows owing to the Jacobian pattern and the additional nonzero singular vectors artifacts arise from the round-off errors introduced by the matrix factorization. This can be noticed in Figure \ref{Fig::MDEIM_location_points} where the interpolation indexes generated by algorithm \ref{alg::DEIM} using MDEIM and SMDEIM singular vectors are depicted. Figure \ref{Fig::MDEIM_location_points}(a) shows a perfect match of the first $20$ DEIM indexes that correspond to singular values ranging from $467.43$ to $0.107$. Figure \ref{Fig::MDEIM_location_points}(b) presents the DEIM indexes for the $100^{\textrm{th}}-120^{\textrm{th}}$ singular values with ranges between $1.41e^{-11}$ to $4.65e^{-14}$. At this low magnitude DEIM indexes mismatches can be noticed. We remark the multiple indexes of the MDEIM singular vectors outside the diagonal band that point to zero entries of the Jacobian snapshots. The first mismatch occurs at the $84^{th}$ singular value where the level of energ
y is $2.66e^{-8}$. However for most of the applications there is no need to select so many DEIM indexes including those corresponding to such small singular 
vectors, perhaps except simulating turbulence. For finding $20$ pair of DEIM indexes, algorithm \ref{alg::DEIM} was $24$ times faster when using the SMDEIM singular vectors.

\begin{figure}[h]
  \centering
  \subfigure[1D-Burgers - $1^{\textrm{st}}-20^{\textrm{th}}$ interpolation indexes]{\includegraphics[scale=0.37]{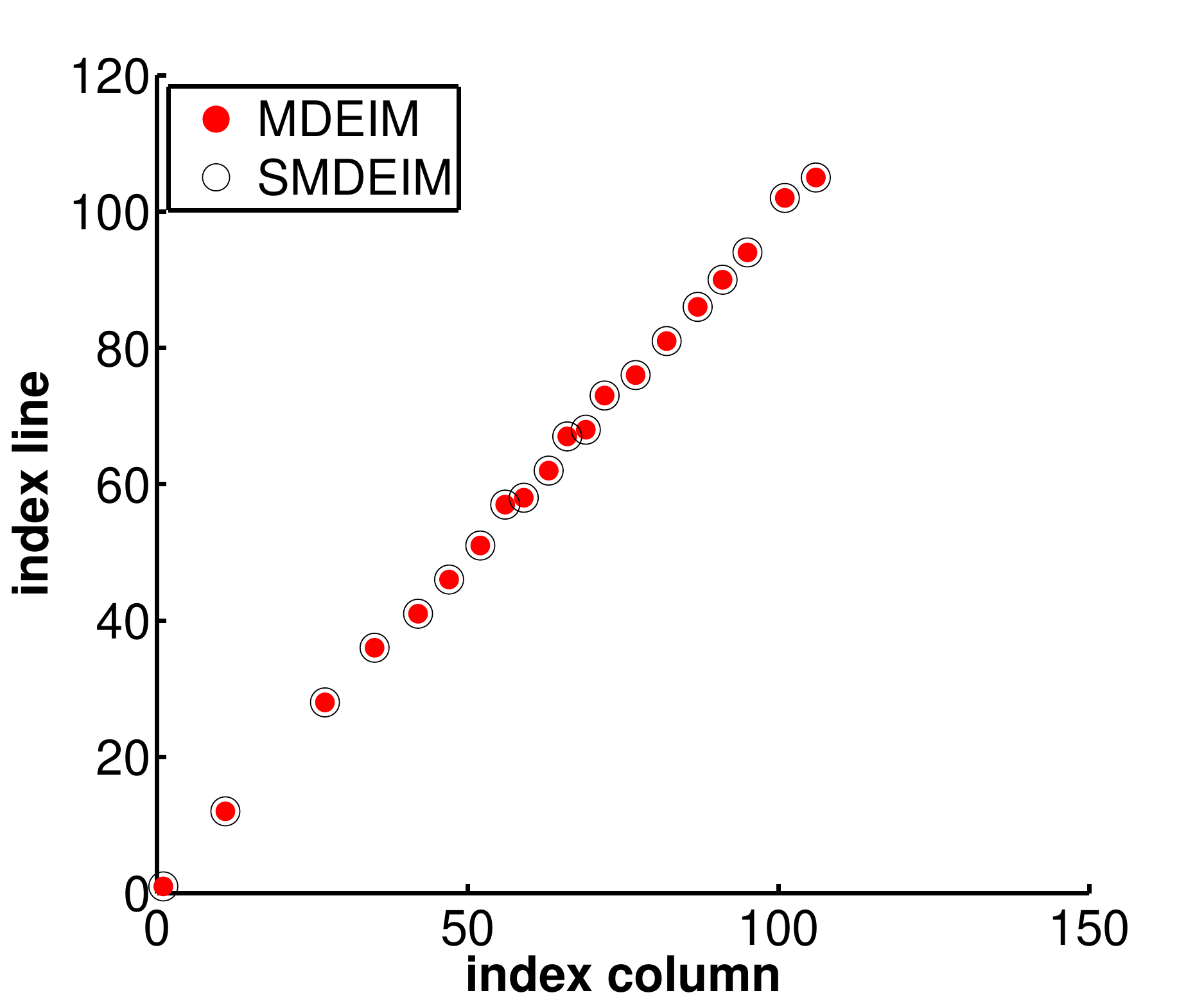}}
  \subfigure[1D-Burgers - $100^{\textrm{th}}-120^{\textrm{th}}$ interpolation indexes]{\includegraphics[scale=0.37]{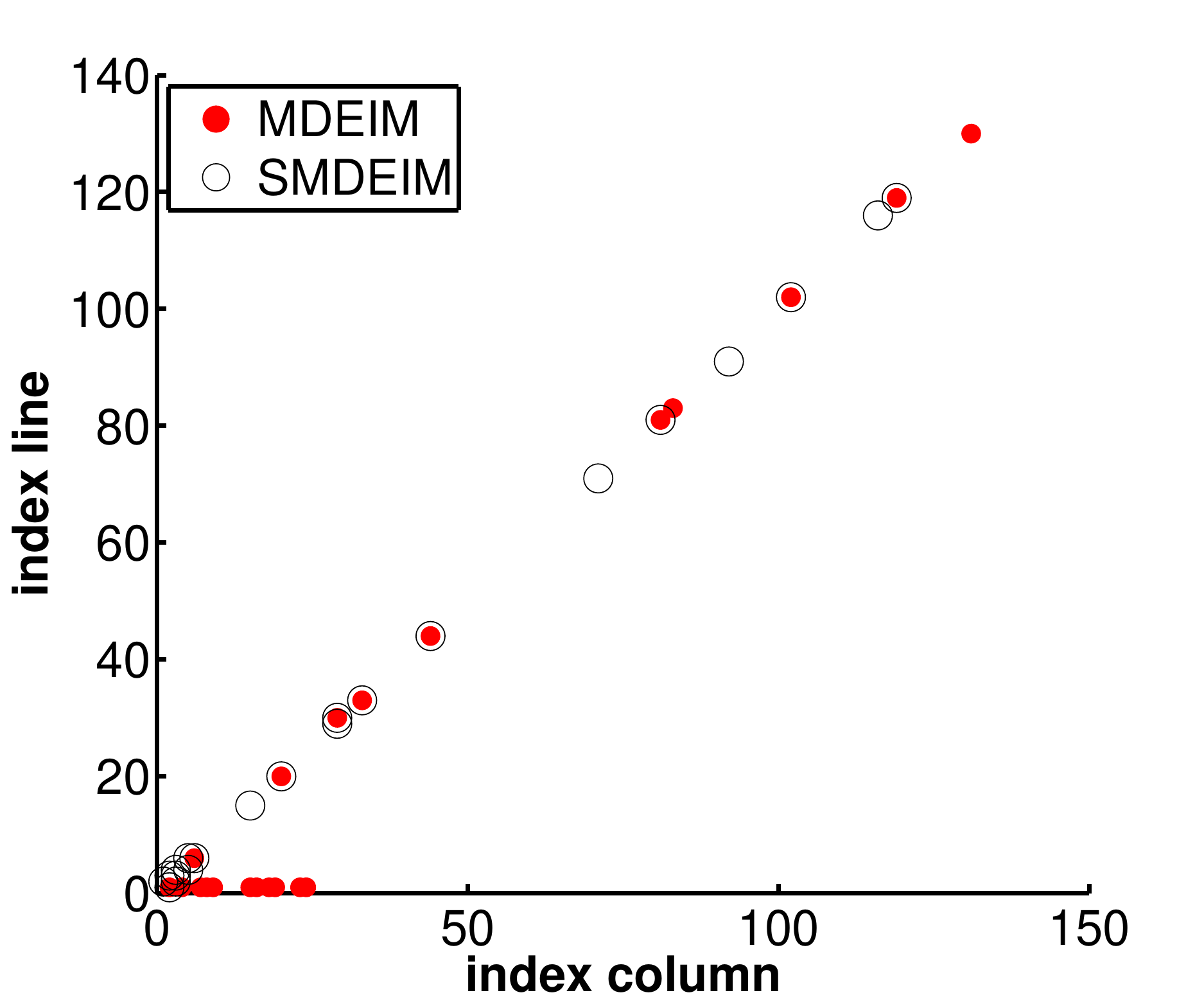}}
\caption{Localization of DEIM indexes using MDEIM and SMDEIM singular vectors 
\label{Fig::MDEIM_location_points}}
\end{figure}
   
Finally we compare the accuracy levels of DEIM \eqref{eq::Jacobian_DEIM_approx}, MDEIM \eqref{eq::MDEIM}  and SMDEIM \eqref{eq::sparseMDEIM} approximations of 
the 1D Burgers model Jacobian at initial time. The DEIM based Jacobian approximation \eqref{eq::Jacobian_DEIM_approx} requires adding the derivatives of the linear terms while for the matrix DEIM approximations the linear and nonlinear partial derivatives are both included into the snapshots. 
Figure \ref{Fig::Full_Jacobian_error} depicts the error of the Jacobian approximations using the Frobenius norm (left panel) and the absolute value of the discrepancies in the largest singular value of the matrices approximations and its true representation (right panel). The MDEIM and SMDEIM Jacobian approximations accuracy is improved with the increase of the DEIM indexes. This is not the case for the DEIM approximations which preserves the accuracy only for the rows $P^T{\bf J}_{\bf F}(\tau)$ selected by the interpolation matrix in \eqref{eq::Jacobian_DEIM_approx}. The SMDEIM Jacobian approximation quality is similar as the MDEIM proposal but it is more advantageous since it is obtained at much lower computational cost.

\begin{figure}[h]
  \centering
  \subfigure[1D-Burgers - $201$ space points indexes]{\includegraphics[scale=0.37]{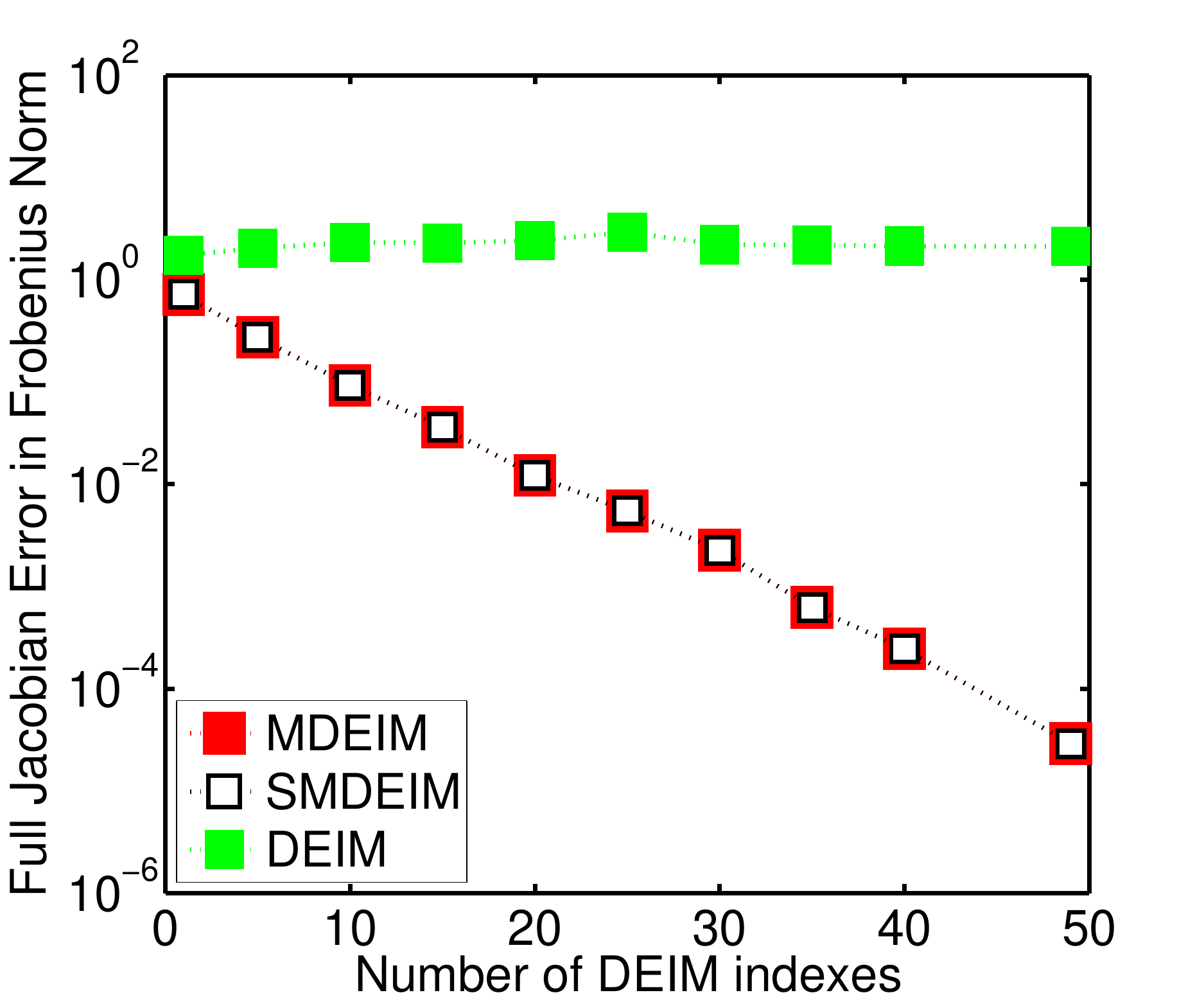}}
  \subfigure[1D-Burgers - $201$  space points ]{\includegraphics[scale=0.37]{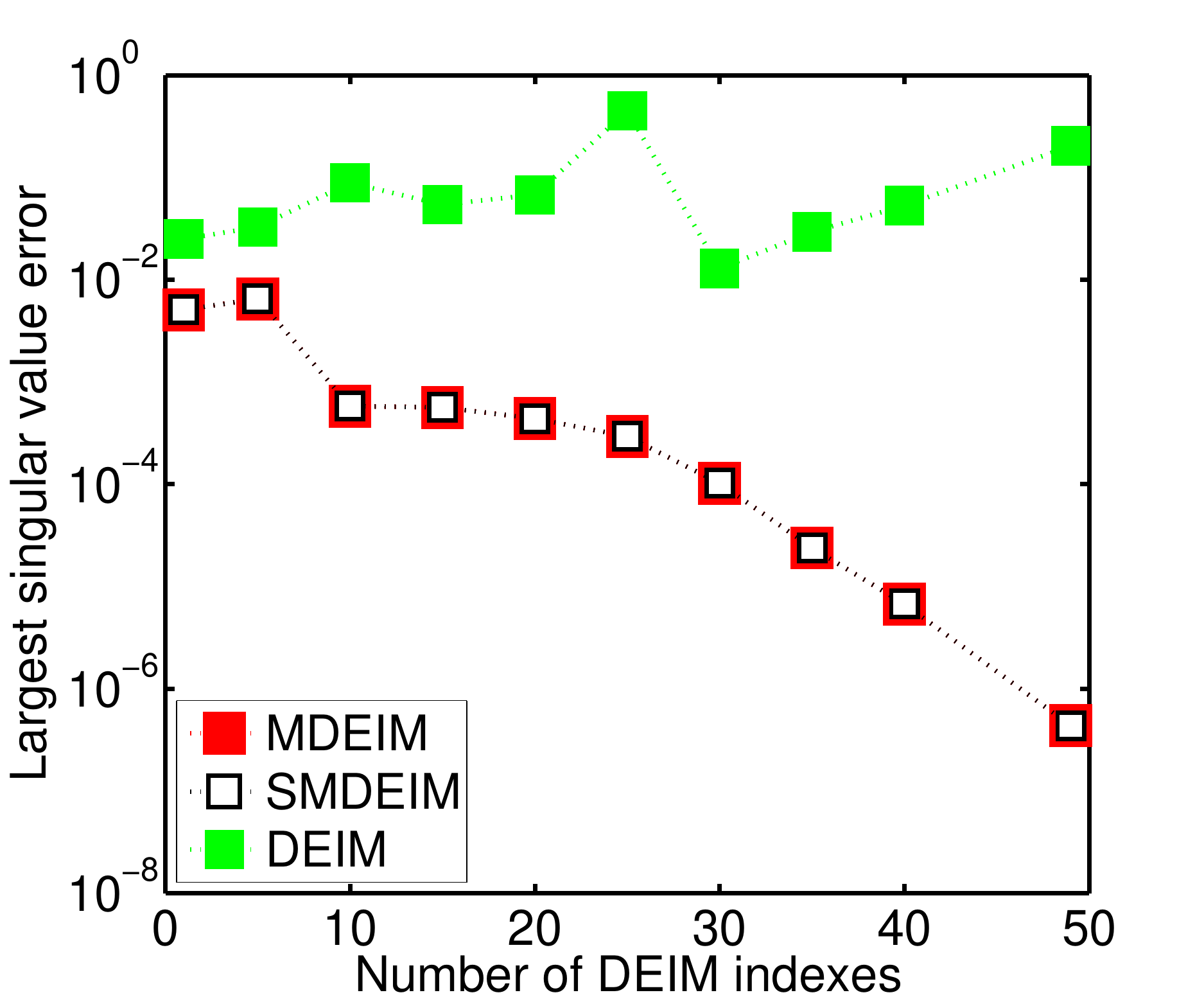}}
\caption{Full Jacobians errors at initial time - Frobenius norm - Largest SVD. 
\label{Fig::Full_Jacobian_error}}
\end{figure}

\subsubsection{Performance of implicit reduced order models}

In general, the implicit discrete problems obtained from discretization of nonlinear partial differential equations and ordinary differential
equations are solved by employing some sort of a Newton based technique. It requires some residuals computations and their space derivatives evaluations. 
This is also the case for reduced implicit discrete problems. While for reduced residual calculations we will apply the tensorial POD approach,
for reduced Jacobians computations we will make use of six different techniques described in subsection \ref{Sub::Red_Jacobian_comp} including the newly introduced
SMDEIM method. For simplicity we decide to employ a reduced order expansion \eqref{eq::POD_expansion} that does not account for the mean.

We will compare the computational off-line/on-line costs as well as the accuracy of the proposed methods. As measures we propose the Frobenius norm of the 
errors between the reduced Jacobian and the solutions of the reduced order models. Details about the number of Newton iterations are presented for each method.

We derive the reduced order 1D Burgers model by employing a Galerkin projection. The constructed POD basis functions are the singular vectors
of the state variable and nonlinear term snapshots matrices obtained from the numerical solution of the full - order implicit Euler 1D Burgers discrete model. 
Figure \ref{Fig::Eigenvalues_1D_Burgers+MDEIM_vs_SMDEIM} shows the decay around the singular values of the solution $u$ and the advection term 
$u\frac{\partial u}{\partial x}$ for $401$ snapshots equally distributed in the interval $[0,2]$ and $201$ number of space points. This configuration is used for 
the majority of the experiments in this subsection.

\begin{figure}[h]
  \centering
   \includegraphics[scale=0.37]{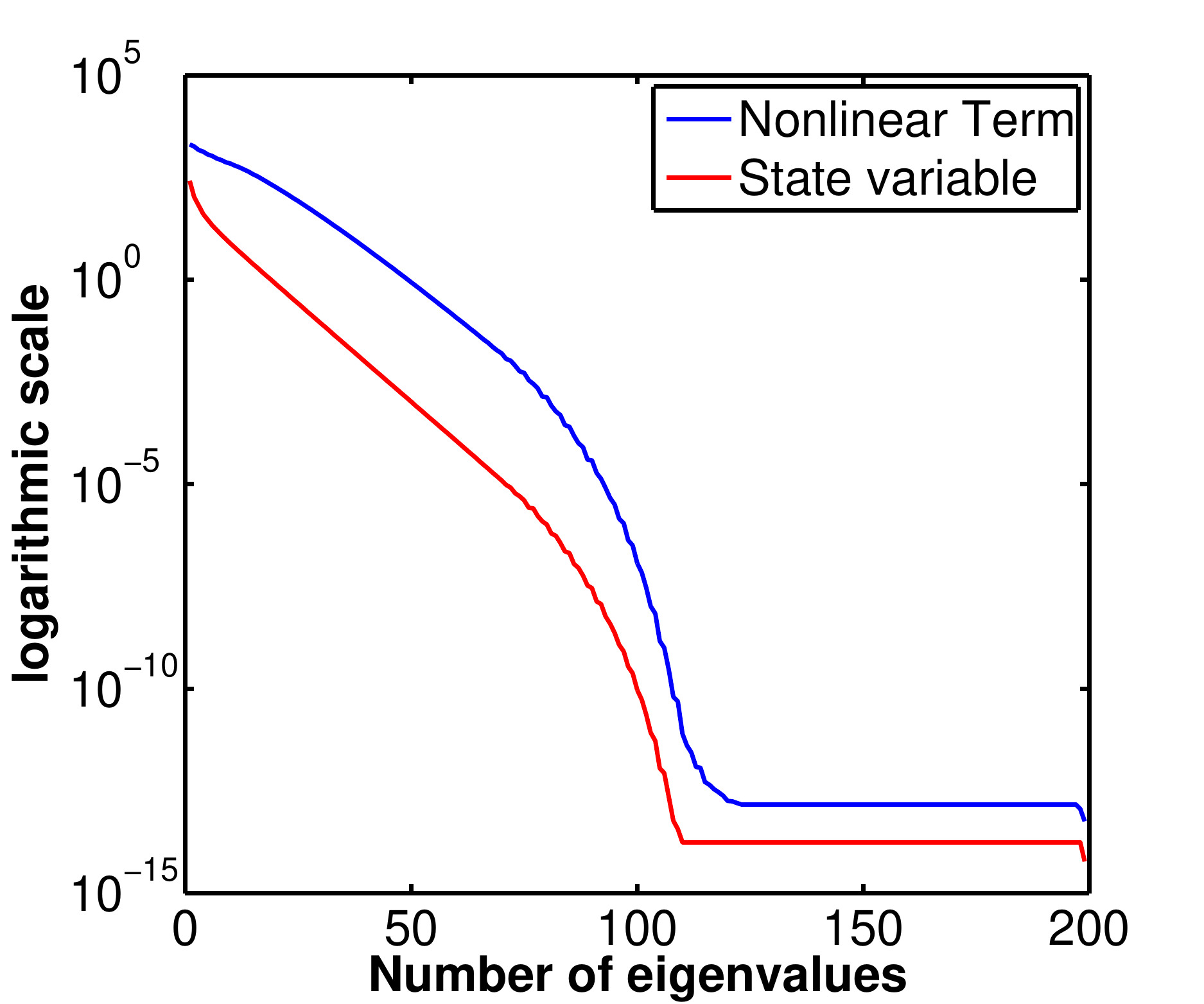}
  \caption{Spectrum properties of state variable and nonlinear term snapshots matrices. 
\label{Fig::Eigenvalues_1D_Burgers+MDEIM_vs_SMDEIM}}
\end{figure}

\paragraph{Off-line computational performances} We begin our comparison study by focusing on the off-line CPU costs of the reduced Jacobian methods. Only DEIM based techniques and tensorial method have off-line 
stages thus only their performances will be discussed here. Various POD basis dimensions $k$, number of interpolation indexes $m$, space points $n$ 
and time steps $N_t$ are considered. 

First the SVD factorizations of the MDEIM, SMDEIM ad DEIM snapshots matrices are derived. Next the DEIM algorithm \ref{alg::DEIM} computes the interpolation
indexes for each set of singular vectors. Then matrices $U^TV(P^TV)^{-1}$, ${C} \cdot V_m\, \left(P_m^TV_m\right)^{-1}$ and 
${\tilde C} \cdot V_{m_{nz}}\Big(P_{m_{nz}}^TV_{m_{nz}}\Big)^{-1}$ in  \eqref{eq::Jacobian_DEIM_approx}, \eqref{eq::MDEIM}  and \eqref{eq::sparseMDEIM}
are assembled along with tensor G \eqref{eq::tensor_g} and the total computational costs of the above off-line operations are counted separately for each
method and depicted in Figure \ref{Fig::Offline_CPU_comparisons}. In panel (a), we set the number of selected interpolation indexes to $30$ and notice that 
for a reduced basis with $50$ modes the SMDEIM is $45.7$ times faster than MDEIM. The DEIM method has the smallest off-line computational cost while tensorial
method CPU time tends to increase exponentially with the growth of POD basis dimension. For the numerical experiment depicted in Figure \ref{Fig::Offline_CPU_comparisons}(b)
we choose $25$ POD basis functions. The off-line CPU times of the greedy based techniques are slightly increased with the growth of the number of the interpolation
indexes. For $m=50$, we remark that SMDEIM is $54.2$ times more rapid than MDEIM. Once we increase the number of mesh points we observe in
Figure \ref{Fig::Offline_CPU_comparisons}(c) that the sparse 
version becomes much faster. For $25$ POD basis functions and $30$ interpolation indexes the sparse
MDEIM is approximately $200$ times faster than the MDEIM method for $501$ space points and $1001$ time steps.

\begin{figure}[h]
  \centering
  \subfigure[$m~=~30$;] {\includegraphics[scale=0.24]{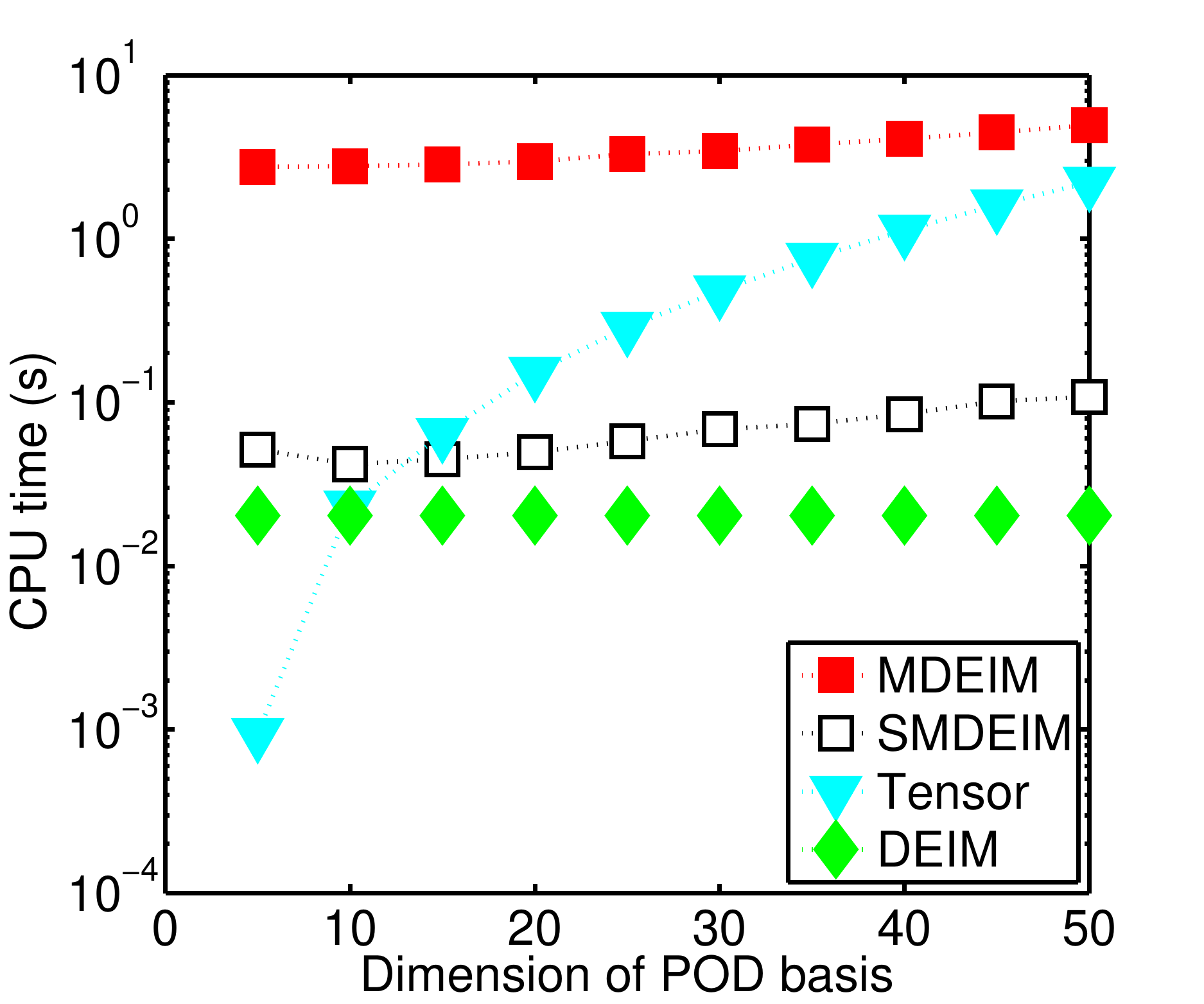}}
  \subfigure[$k~=~25$;]{\includegraphics[scale=0.24]{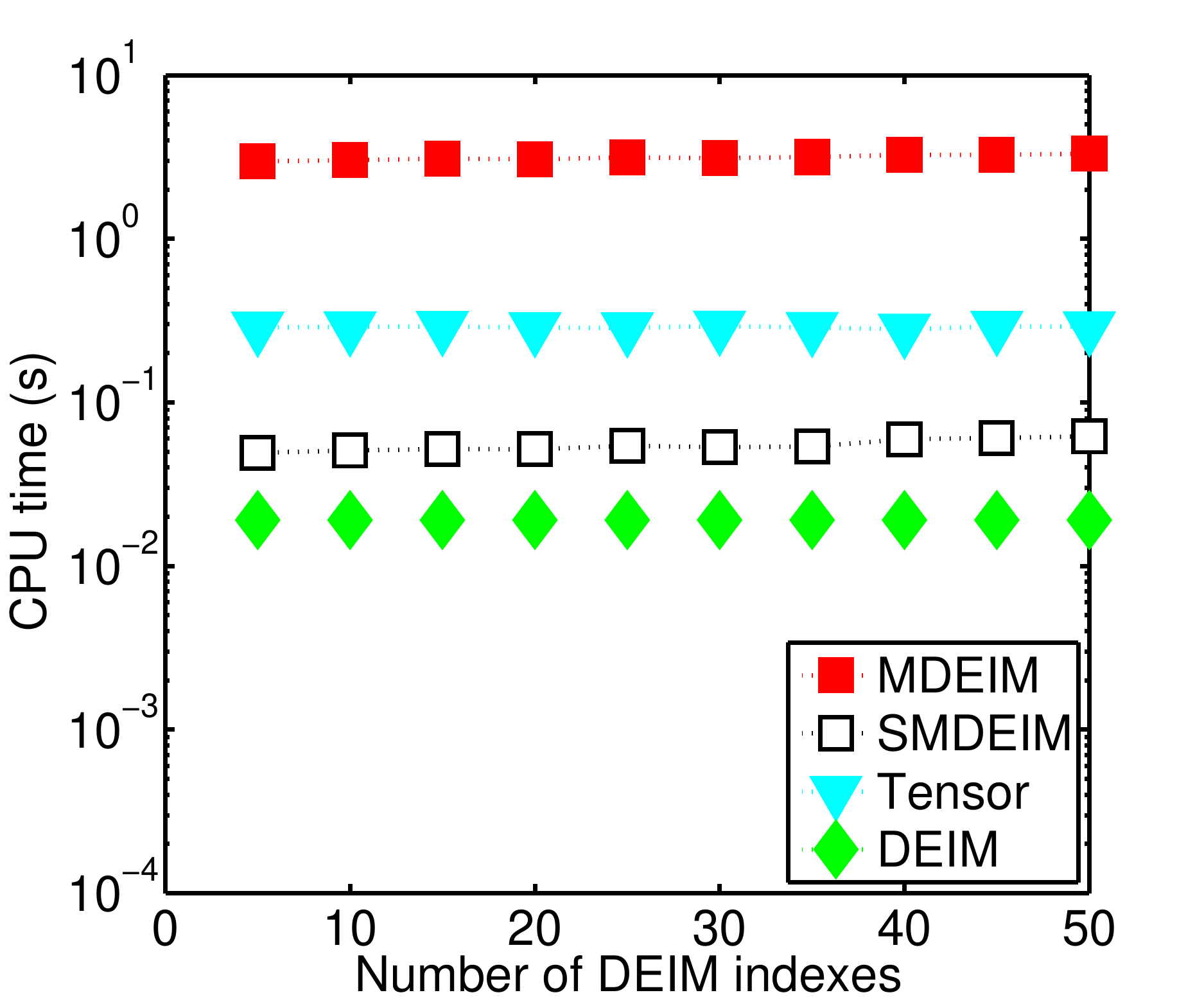}}
  \subfigure[$m~=~30$,~$k~=~25$;]{\includegraphics[scale=0.24]{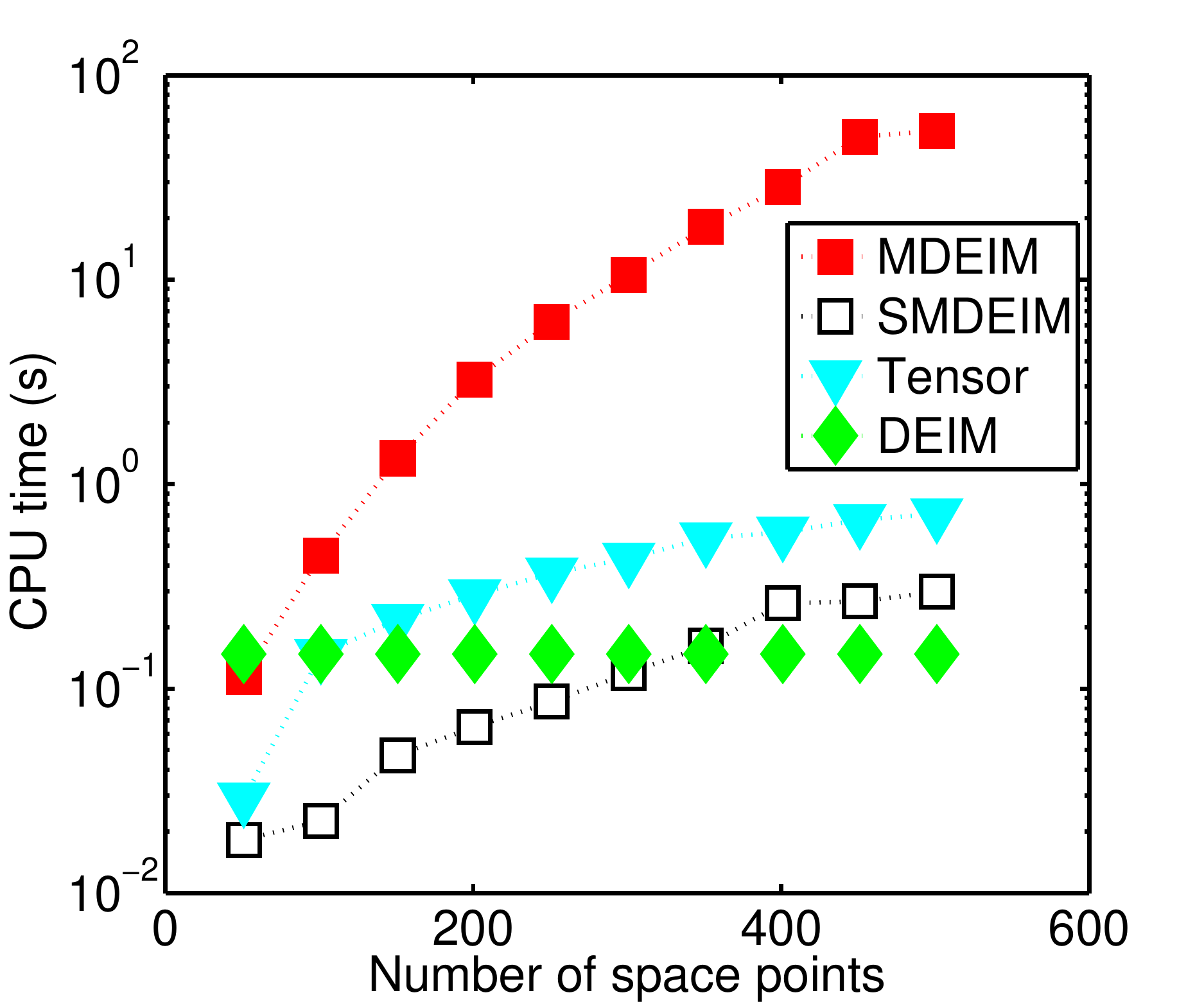}}
\caption{Off-line computational time performances of MDEIM and SMDEIM 
\label{Fig::Offline_CPU_comparisons}}
\end{figure}

\paragraph{Reduced Jacobian errors} The accuracy of the reduced Jacobian approximations is measured in the reduced space against the standard $U^T{\bf J}_{\bf r^i}(\x_{t_i})U$,
where the residual full Jacobian is computed from the high fidelity model solutions. Figure \ref{Fig::Red_Jac_errors} describes the reduced derivatives
errors at the initial time step calculated using Frobenius norm for various number of POD basis functions and DEIM indexes. 

For the experiment designed in Figure \ref{Fig::Red_Jac_errors}(a), the number of interpolation indexes is set to $30$ while the POD basis dimension is increased.
We notice that the DEIM based approximations quality is highly dependent on the size of the reduced manifold. However SMDEIM and MDEIM are less affected since their
proposed Jacobians are $2$ orders of magnitude more accurate than DEIM approximation. Figure \ref{Fig::Red_Jac_errors}(b) illustrates the impact of an increased
number of interpolation indexes onto the quality of Jacobian approximations while the dimension of the POD basis is maintained steady with $k=20$. SMDEIM and 
MDEIM requires $25$ interpolation indexes to reach the same precision as the tensorial, directional derivatives and direct projection methods while DEIM needs more 
than $50$ points for the same accuracy level.

\begin{figure}[h]
  \centering
  \subfigure[$m~=~30$] {\includegraphics[scale=0.35]{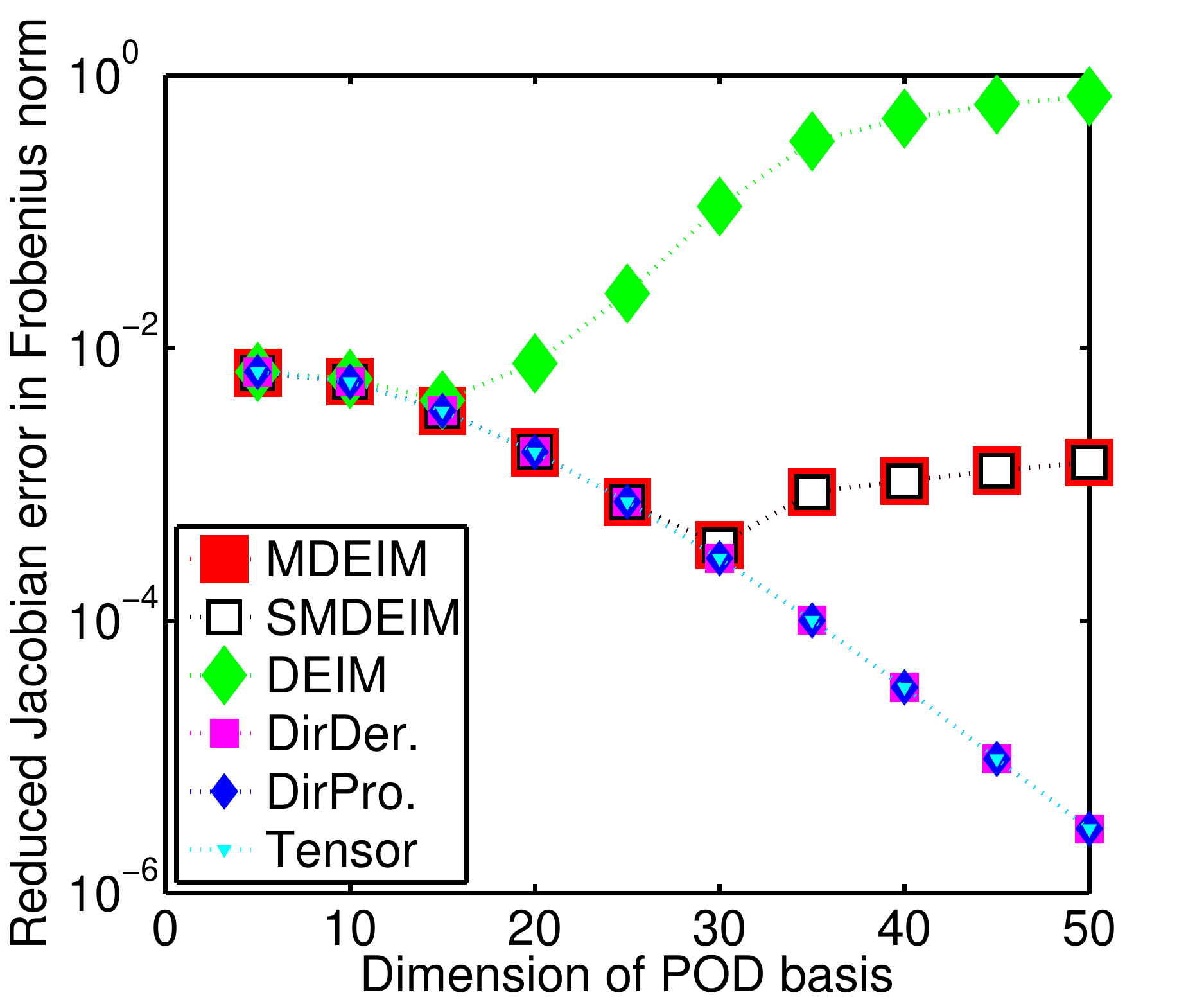}}
  \subfigure[$k~=~20$ ]{\includegraphics[scale=0.35]{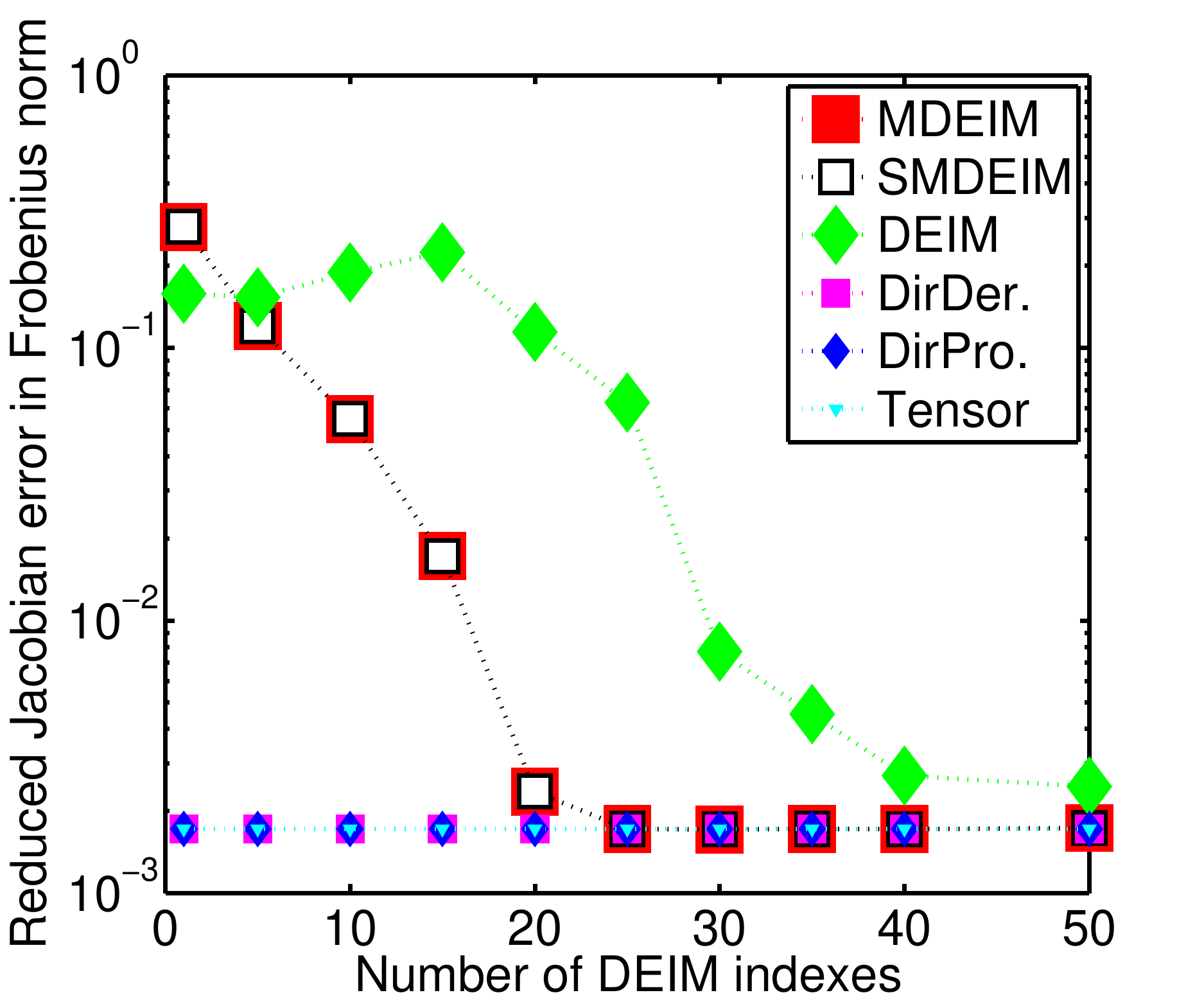}}
\caption{Reduced Jacobians errors
\label{Fig::Red_Jac_errors}}
\end{figure}

\paragraph{On-line computational performances} Here we compare the on-line characteristics of the proposed implicit reduced order models. Two features are of interest,
i.e. CPU time and solution accuracy. Figure \ref{Fig::Red_online_CPU_time1}(a) shows the amount of integration time required by the 
reduced order models to obtain their solutions with respect to the dimension of POD basis $U$. We set the number of DEIM interpolation indexes to $30$. 
SMDEIM and MDEIM perform similarly thus confirming the teoretical on-line computational complexities derived in subsection 
\ref{Sub::Red_Jacobian_comp}. For $50$ POD basis modes the MDEIM and SMDEIM are $2\times$, $3.75\times$, $4\times$, $7.43\times$ $16\times$ and $16.2\times$ faster 
than directional derivative, direct projection, DEIM, high-fidelity and tensorial models, respectively. The efficiency loss 
in the case of tensorial method is in accordance with the results in table \ref{table_complexityII}. The DEIM reduced order model performance is affected by the 
quality of the Jacobian approximation as we can notice in Figure \ref{Fig::Red_Jac_errors}(a) which doubles the averaged number of Newton-Raphson iterations per each time
step (see table \ref{table_NR_dim_POD}). While directional derivative and direct projection reduced order models use Matlab implementations based on vector operations,
the other models don't due to their core algorithms nature. It is well known that Matlab is tuned to enhance efficient vector operations calculations thus the former
models are advantaged. Similar Frobenius norms measuring the discrepancies between the projected reduced and high-fidelity solutions are obtained by all reduced 
order models (not shown due to data redundancy), however DEIM requires more Newton iterations as seen in table \ref{table_NR_dim_POD}.

\begingroup
\begin{table}  
\begin{center} 
\begin{tabular}[h]{|c|c|c|c|c|c|c|c|c|c|c|} \hline 
 POD basis dimension & 5 & 10 & 15 & 20 & 25 & 30 & 35 & 40 & 45 & 50 \cr \hline \hline 
MDEIM/SMDEIM & 3.98 & 4.11 & 4.37 & 4.50 & 4.61 & 4.66 & 4.71 & 4.75 & 4.80 & 4.82 \cr \hline 
DEIM & 4.00 & 4.11 & 4.45 & 4.95 & 5.73 & 7.03 & 10.05 & 10.89 & 11.39 & 11.59 \cr \hline 
DirDer/DirPro/Tensor & 3.85 & 4.09 & 4.37 & 4.50 & 4.61 & 4.66 & 4.70 & 4.78 & 4.81 & 4.83 \cr \hline 
Full & 4.83 & 4.83 & 4.83 & 4.83 & 4.83 & 4.83 & 4.83 & 4.83 & 4.83 & 4.83 \cr \hline 
\end{tabular} 
\end{center} 
\caption{The mean variation of Newton-Raphson iterates per time step along the change in the POD basis dimension}\label{table_NR_dim_POD} 
\end{table}  
\endgroup%

For the next experiment we measure the effect of an increased number of space points and the results are depicted in Figure \ref{Fig::Red_online_CPU_time1}(b). 
The dimension of POD basis is set to $25$ and $30$ DEIM interpolation indexes are selected. While only directional derivative and direct projection 
have on-line computational complexities depending on the dimension of the full space, all of the proposed reduced order models suffer as the number 
of meshed points becomes larger. As the space dimension increases the number of time discretization points is raised too. For example, for
$501$ mesh points we use $1001$ time steps and MDEIM and SMDEIM are $2.4\times$, $2.55 \times$, $4.2 \times$, $11\times$ and $102\times$ times faster than 
the directional derivative, DEIM, tensorial, direct projection, and full models. The quality of the Jacobian approximations is reflected in the number of Newton-Raphson
iterations, and again DEIM requires more loops (see table \ref{table_NR_space_points}) to achieve the same level of accuracy as the other reduced order models.

\begingroup
\begin{table}  
\begin{center} 
\begin{tabular}[h]{|c|c|c|c|c|c|c|c|c|c|c|} \hline 
 Number of space points & 51 & 101 & 151 & 201 & 251 & 301 & 351 & 401 & 451 & 501 \cr \hline \hline 
MDEIM/SMDEIM & 6.14 & 5.38 & 4.95 & 4.61 & 4.44 & 4.33 & 4.23 & 4.12 & 3.97 & 3.92 \cr \hline 
DEIM & 6.22 & 6.35 & 5.82 & 5.73 & 5.74 & 5.41 & 5.84 & 5.54 & 5.09 & 5.62 \cr \hline 
DirDer/DirPro/Tensor & 6.14 & 5.38 & 4.95 & 4.61 & 4.43 & 4.33 & 4.23 & 4.12 & 3.95 & 3.92 \cr \hline 
Full & 4.14 & 4.14 & 4.14 & 4.14 & 4.14 & 4.14 & 4.14 & 4.14 & 4.14 & 4.14 \cr \hline 
\end{tabular} 
\end{center} 
\caption{The mean variation of Newton-Raphson iterates per time step along the change in the number of space points}\label{table_NR_space_points} 
\end{table}  
\endgroup%

\begin{figure}[h]
  \centering
  \subfigure[$m~=~30$] {\includegraphics[scale=0.24]{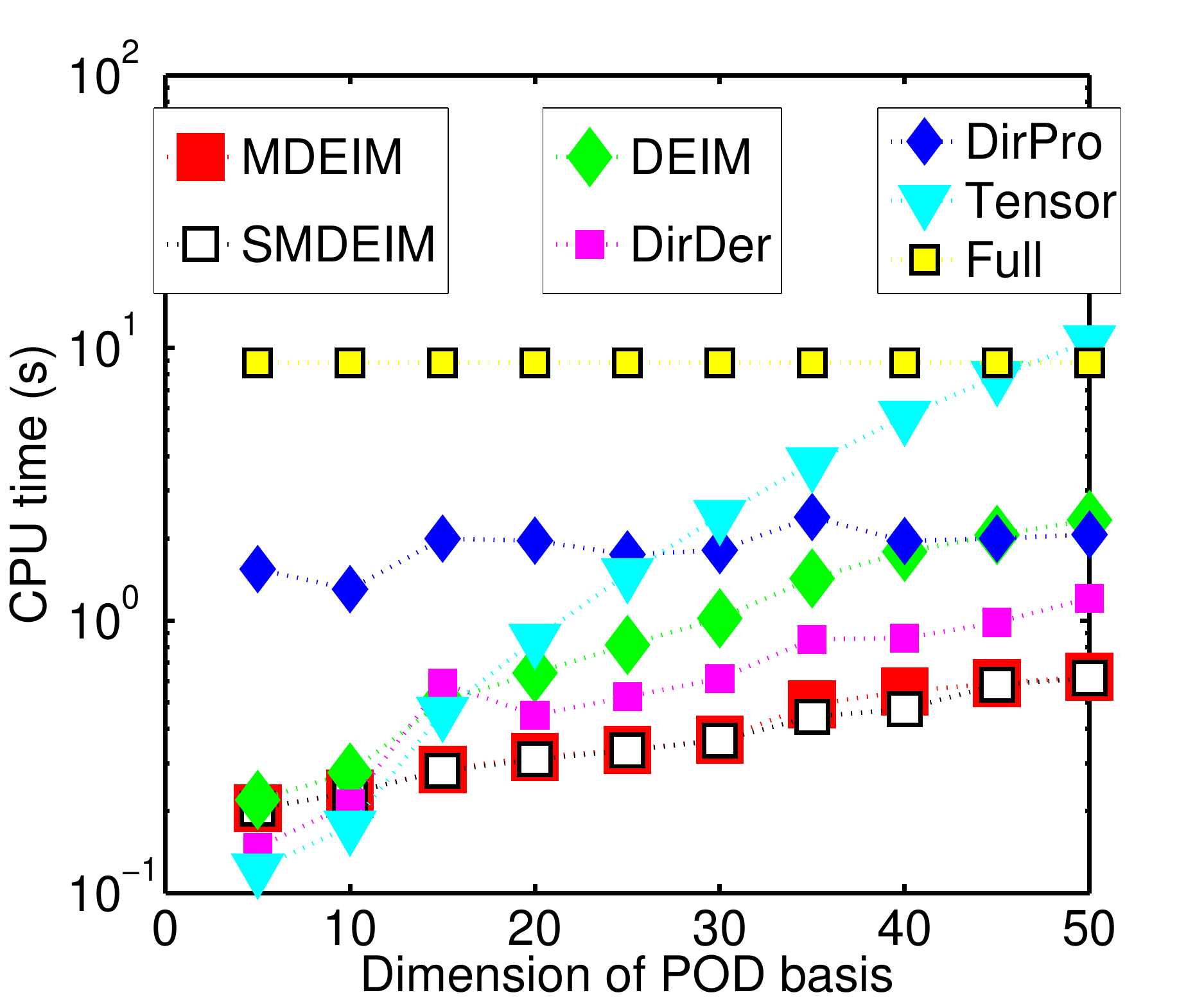}}
  \subfigure[$m~=~30,~k~=~25$]{\includegraphics[scale=0.24]{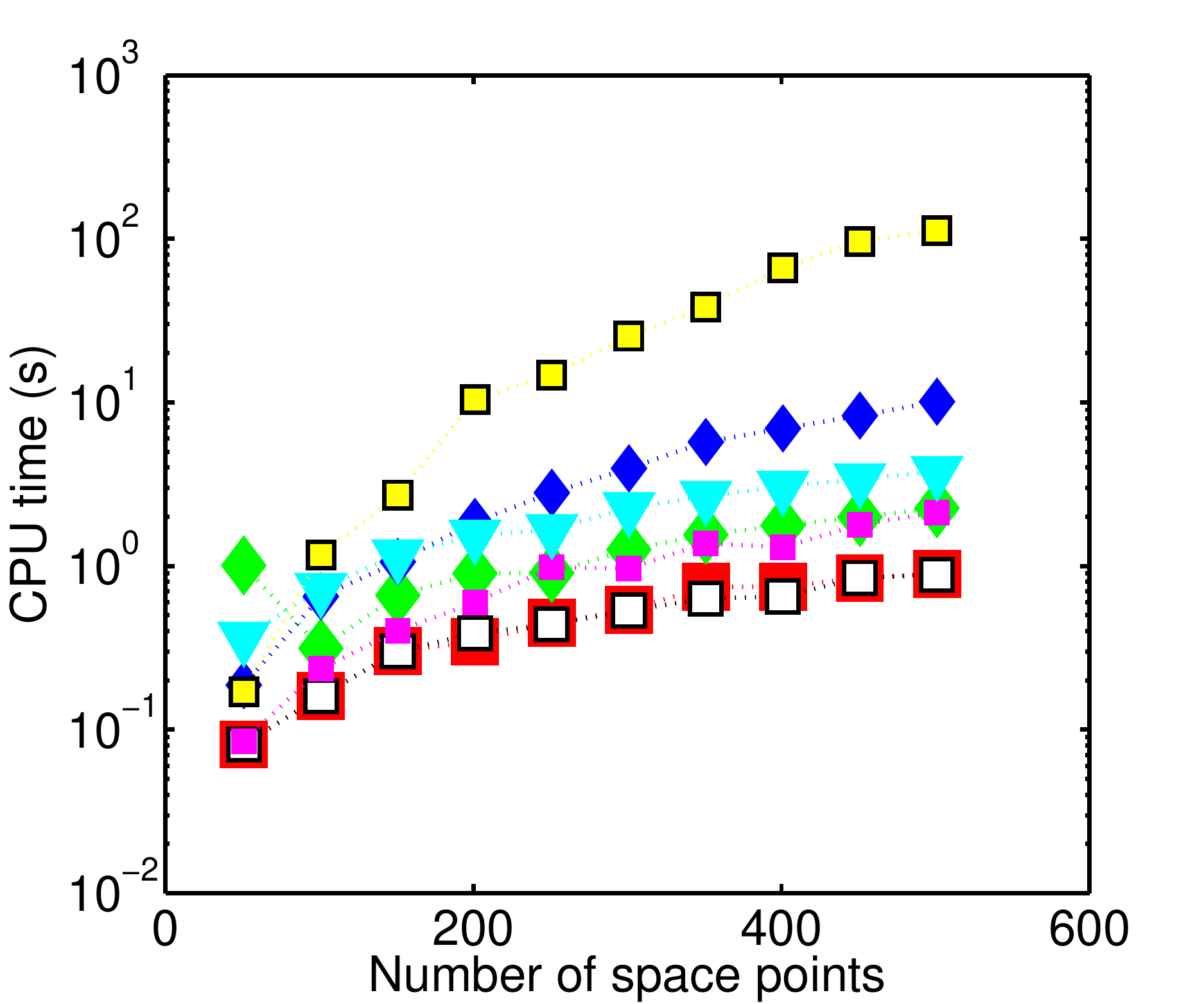}}
  \subfigure[$k~=~25$]{\includegraphics[scale=0.24]{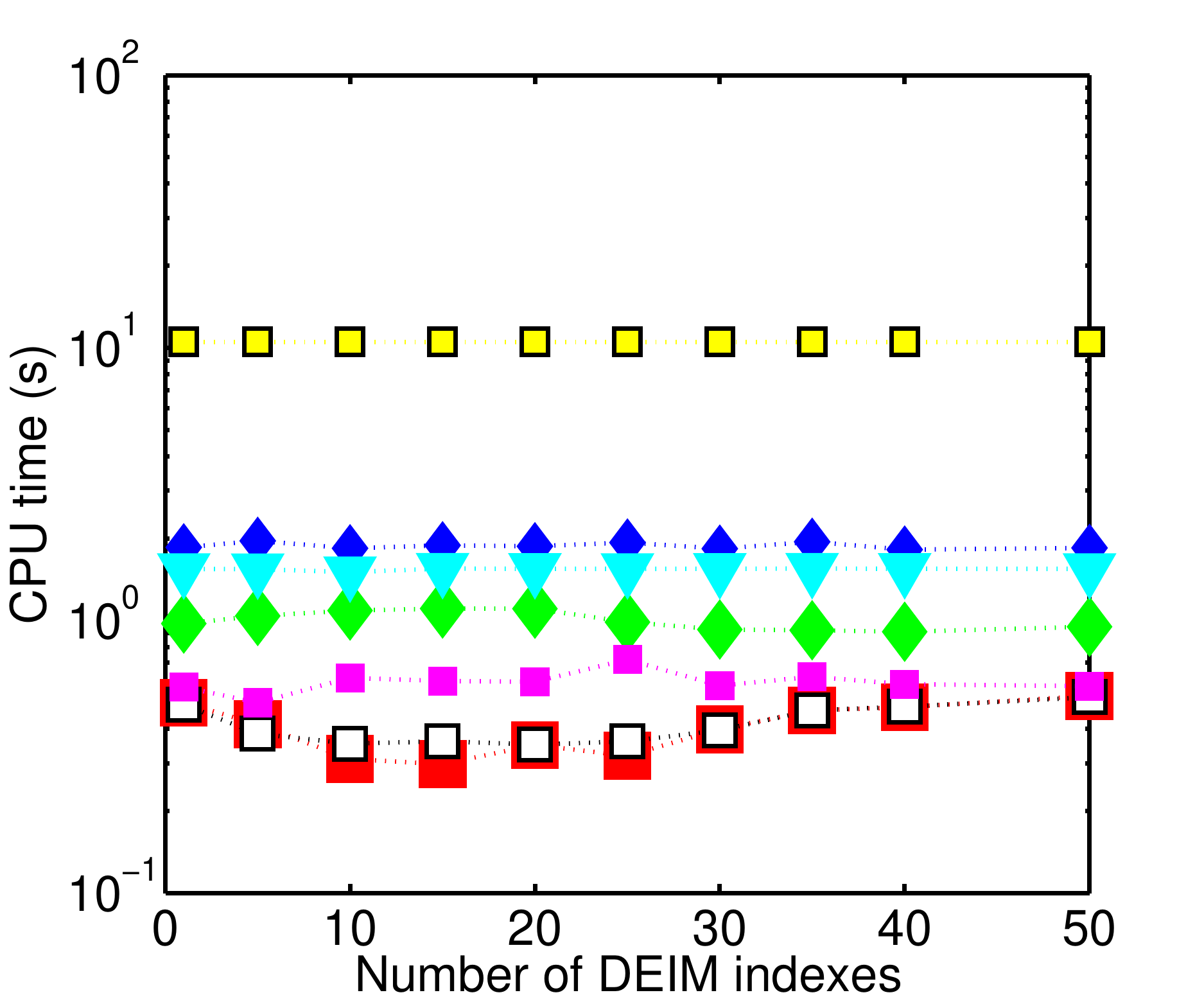}}
\caption{On-line computational time performances of reduced order models
\label{Fig::Red_online_CPU_time1}}
\end{figure}

By increasing the number of interpolation indexes to $50$ the CPU times for MDEIM and SMDEIM reduced order models are slightly increased as seen in 
Figure \ref{Fig::Red_online_CPU_time1}(c). The number of Newton iterations for matrix DEIM techniques becomes similar with those
of exact Jacobian techniques for $m$ larger than $20$ indexes, while in the case of DEIM more than $40$ DEIM indexes are required (\ref{table_NR_DEIM_indexes}). 
This is also noticed in Figure \ref{Fig::Red_online_CPU_time1}(c) where the DEIM reduced order model CPU time is decreasing even if the number of interpolation
indexes is raised.

\begingroup
\begin{table}  
\begin{center} 
\begin{tabular}[h]{|c|c|c|c|c|c|c|c|c|c|c|} \hline 
 Number of DEIM indexes & 1 & 5 & 10 & 15 & 20 & 25 & 30 & 35 & 40 & 50 \cr \hline \hline 
MDEIM/SMDEIM & 8.58 & 6.88 & 5.64 & 4.98 & 4.62 & 4.61 & 4.61 & 4.61 & 4.61 & 4.61 \cr \hline 
DEIM & 11.38 & 11.41 & 11.48 & 11.12 & 10.29 & 7.62 & 5.73 & 4.95 & 4.65 & 4.61 \cr \hline 
DirDer/DirPro/Tensor & 4.61 & 4.61 & 4.61 & 4.61 & 4.61 & 4.61 & 4.61 & 4.61 & 4.61 & 4.61 \cr \hline 
Full & 4.83 & 4.83 & 4.83 & 4.83 & 4.83 & 4.83 & 4.83 & 4.83 & 4.83 & 4.83 \cr \hline 
\end{tabular} 
\end{center} 
\caption{The mean variation of Newton-Raphson iterates per time step along the change in the number of DEIM indexes}\label{table_NR_DEIM_indexes} 
\end{table}  
\endgroup%

\subsection{Two-dimensional Shallow Water Equations}\label{subsec:SWE}

\subsubsection{Numerical Scheme}

In meteorological and oceanographic problems, one is often not interested in small time steps because the discretization error in time is small compared to the discretization error in space. The alternating direction fully implicit (ADI) scheme \citep{Gus1971} considered in this paper is first order in both time and space and it is stable for large CFL condition numbers. It was also proved that the method is unconditionally stable for the linearized version of the SWE model. Other research work on this topic include efforts of \citet{FN1980} and \citet{NVG1986}.

We are solving the SWE model using the $\beta$-plane approximation on a rectangular domain \citep{Gus1971}
\begin{equation}\label{eqn:swe-pde}
\frac{\partial w}{\partial t}=A(w)\frac{\partial w}{\partial x}+B(w)\frac{\partial w}{\partial y}+C(y)w,
\quad (x,y) \in [0,L] \times [0,D], \quad t\in(0,t_{\rm f}],
\end{equation}
where $w=(u,v,\phi)^T$ is a vector function and $u,v$ are the velocity components in the $x$ and $y$ directions, respectively. Geopotential is computed using
$\phi = 2\sqrt{gh}$, $h$ being the depth of the fluid and $g$ the acceleration due to gravity.

The matrices $A$, $B$ and $C$ are
\[
A=-\left(\begin{array}{ccc}
           u&0&\phi/2\\
           0&u&0\\
           \phi/2&0&u \end{array}\right), \quad
B=-\left(\begin{array}{ccc}
           v&0&0\\
           0&v&\phi/2\\
           0&\phi/2&v \end{array}\right), \quad
C=\left(\begin{array}{rrr}
           0&f&0\\
           -f&0&0\\
           0&0&0 \end{array}\right),
\]
and $f$ is the Coriolis term
\[
f=\hat f + \beta(y-D/2),~\beta=\frac{\partial f}{\partial y}, \quad y\in[0,D],
\]
with $\hat f$ and $\beta$ constants.

We assume periodic solutions in the $x$ direction for all three state variables
while in the $y$ direction
$$v(x,0,t)=v(x,D,t)=0,~x\in[0,L],~t\in(0,t_{\rm f}]$$
and Neumann boundary condition are considered for $u$ and $\phi$.

We derive the initial conditions from the initial height condition No. 1 of \citet{Gram1969} i.e.
\begin{equation*}
\hspace{-10mm}h(x,y,0)=H_0+H_1+\tanh\biggl(9\frac{D/2-y}{2D}\biggr)+H_2\textrm{sech}^2\biggl(9\frac{D/2-y}{2D}\biggr)\sin\biggl(\frac{2\pi x }{L}\biggr).
\end{equation*}
and the initial velocity fields are calculated from the initial height field using the geostrophic relationship.

Now we introduce a mesh of $n = N_x\cdot N_y$ equidistant points on $[0,L]\times[0,D]$, with $\Delta x=L/(N_x-1),~\Delta y=D/(N_y-1)$. We also discretize 
the time interval $[0,t_{\rm f}]$ using $N_t$ equally distributed points and $\Delta t=t_{\rm f}/(N_t-1)$. Next we define vectors of unknown variables 
containing approximate solutions such as
$${\boldsymbol w}(t_N)\approx [w(x_i,y_j,t_N)]_{i=1,2,..,N_x-2,~j=1,2,..,N_y-2} \in \mathbb{R}^{3(N_x-2)\times (N_y-2)},~N=1,2,..N_t,~(\textrm{no boundaries included}) $$
The semi-discrete equations of SWE \eqref{eqn:swe-pde} contain six nonlinear functions $F_{11},~F_{12},$ $F_{21},~F_{22}$, $F_{31},~F_{32}$ as described in
\citep[Section 4]{Stefanescu_etal_forwardPOD_2014}. The ADI scheme splits the finite difference equations into two, taking implicitly the x derivatives terms first while
the y derivatives components are treated implicitly in the second stage. The discrete model was implemented in Fortran and uses a sparse matrix environment . For operations with sparse matrices we employed SPARSEKIT library \citep{Saad1994}
and the sparse linear systems obtained during the quasi-Newton iterations were solved using MGMRES library \citep{Barrett94,Kelley95,Saad2003}. The LU decomposition is performed at every time step.
All numerical experiments use the following constants: $L=6000\, km$, $D=4400\, km$, $t_{\rm f} = 3\, hours$, $\hat f=10^{-4}\, sec^{-1}$, $\beta=1.5\cdot10^{-11}\, sec^{-1}m^{-1}$, $g=10\, m\, sec^{-2}$, $H_0=2000\, m$, $H_1=220\, m$, $H_2=133\, m$.
The nonlinear algebraic systems are solved using a Newton-Raphson method and the allowed number of Newton iterations per each time step is set to 50. The solution is
considered accurate enough when the euclidian norm of the residual is less than $1e^{-10}$.

\subsubsection{Greedy based Jacobian approximation techniques}

Initially we discuss the spectrum characteristics of the SMDEIM and MDEIM snapshots matrices of the high-fidelity 2D SWE model and then we compare the accuracy of their 
output Jacobian approximations against the one proposed by DEIM where the building blocks are the nonlinear functions approximations. The numerical experiments in this subsection are obtained for a space mesh of $21\times15 $ points, with $\Delta x = 300$km and $\Delta y \approx 315$km. 
The integration time windows is set to $6$h and we use $91$ time steps ($N_t=91$) with $\Delta t = 240$s. 

The Jacobian approximations are constructed using $180$ snapshots (ADI scheme has an intermediary step, thus the $91$ time steps provide $180$ state variables, Jacobians 
and nonlinear terms snapshots) obtained from the numerical solution of the full - order ADI finite difference SWE model. For the MDEIM method 
the matrices of snapshots belong to $\mathbb{R}^{\big(9(N_x-2)^2(N_y-2)^2\big)\times 180}$ for
both directions. In the case of SMDEIM technique we extract the nonzero elements of the Jacobians and form matrices of snapshots of the sizes $32(N_y-2)+16(N_x-4)(N_y-2) \times 180$
 (x direction) and $22(N_x-2)+16(N_y-4)(N_x-2) \times 180$ (y direction). The DEIM based Jacobians of the ADI SWE finite difference model are constructed using 
 the Jacobians of the $6$ nonlinear functions and each of the corresponding matrices of snapshots has $(N_x-2)(N_y-2) \times 180$ dense elements.

Figure \ref{Fig::Eigenvalues_2D_SWE+MDEIM_vs_SMDEIM} illustrates the singular values of the MDEIM and SMDEIM snapshots matrices in the x direction. The spectra are 
similar however the SMDEIM method is approximately $144 \times$ times faster. Moreover, the number of nonzero entries of the MDEIM singular vectors 
is $4228$, where $276$ of these values are artificially created by the matrix factorization as a consequence of the induced round-off errors.   
 
\begin{figure}[h]
  \centering
   \includegraphics[scale=0.43]{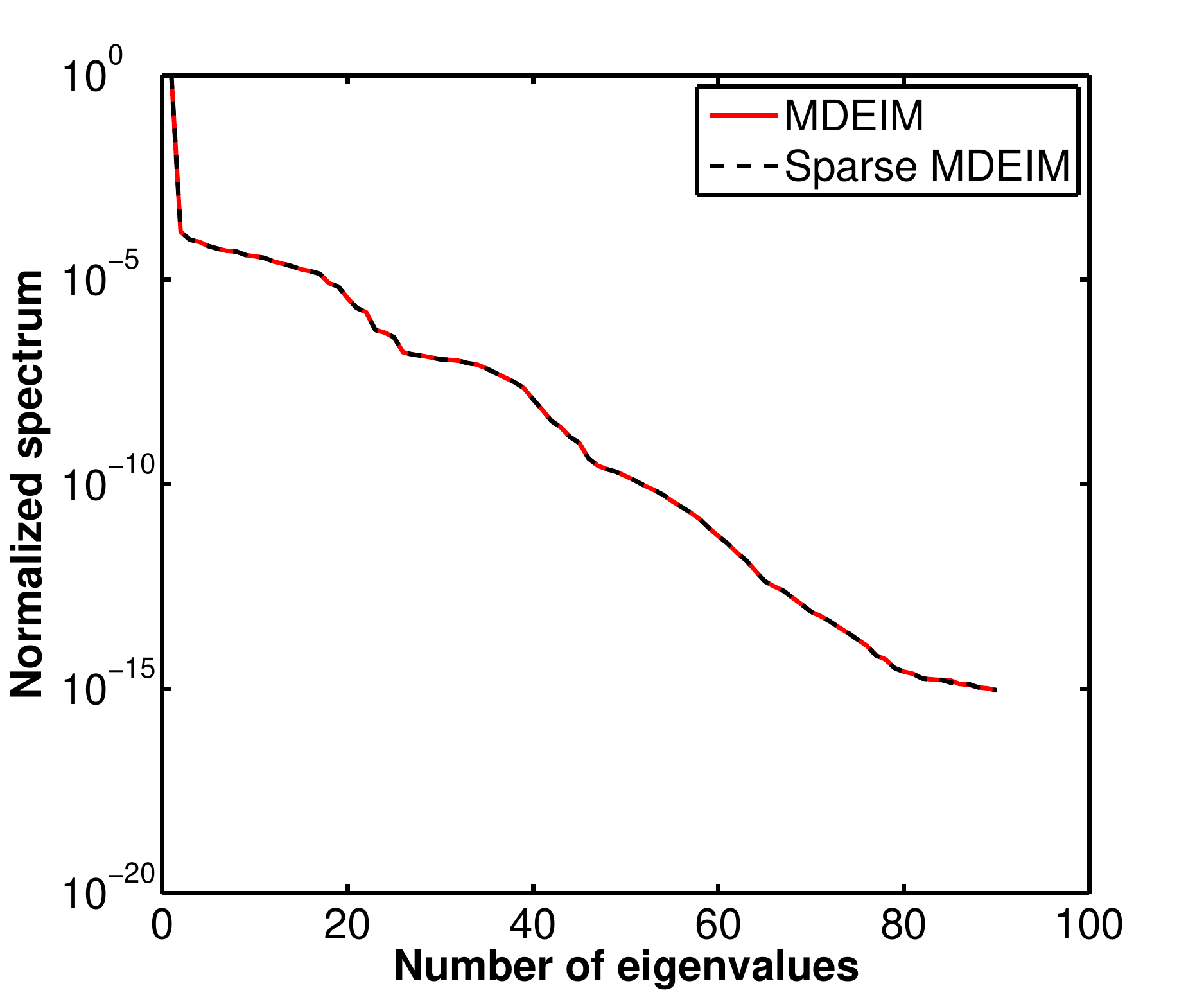}
\caption{Spectrum properties of MDEIM and SMDEIM snapshots matrices. 
\label{Fig::Eigenvalues_2D_SWE+MDEIM_vs_SMDEIM}}
\end{figure} 

Figure \ref{Fig::MDEIM_location_points_2D_SWE} depicts the location of the DEIM indexes obtained by the algorithm \ref{alg::DEIM} 
using MDEIM and SMDEIM singular vectors of the x-derivative implicit 2D SWE model. The first $20$ indices shown in panel (a) correspond to singular values
between $366$ and $1e-3$ and match perfectly. For extremely small singular values ranging from $1e-12$ to $1e-13$ the output DEIM indexes differ significantly
as noticed in panel (b). The MDEIM singular vectors propose also indexes outside of the Jacobian nonzeros bands not seen in the case of SMDEIM approach.
However, for most of the applications, selecting such large number of DEIM indexes does not necessary enhance the quality of the approximation obtained with a smaller 
number of indices. The computational cost of finding the first $20$ DEIM indexes using SMDEIM singular vectors is approximately $220 \times$ times faster than in the case of employing MDEIM singular vectors.

\begin{figure}[h]
  \centering
  \subfigure[2D-SWE - $1^{\textrm{th}}-20^{\textrm{th}}$ interpolation indexes]{\includegraphics[scale=0.37]{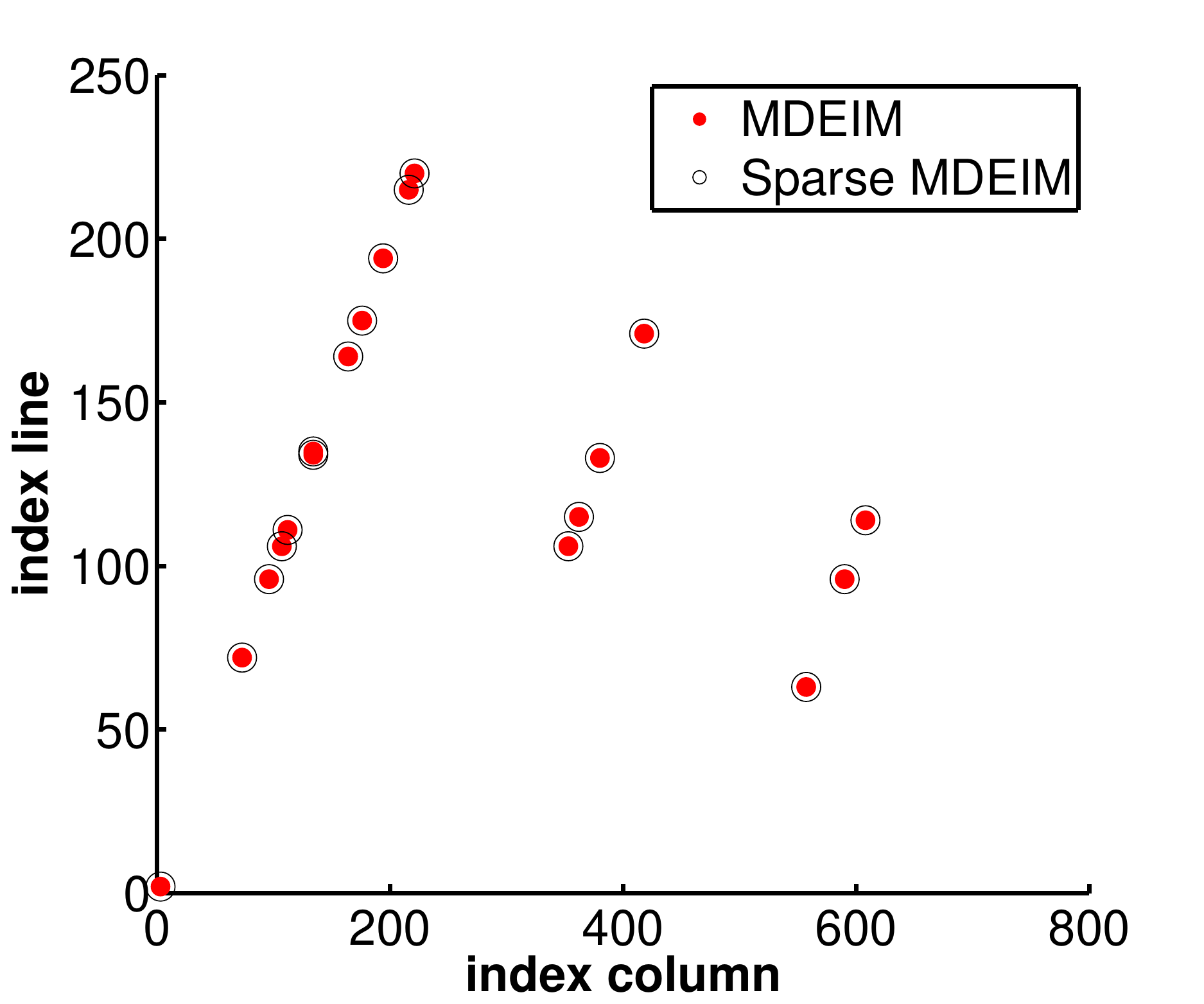}}
  \subfigure[2D-SWE - $80^{\textrm{th}}-100^{\textrm{th}}$ interpolation indexes]{\includegraphics[scale=0.37]{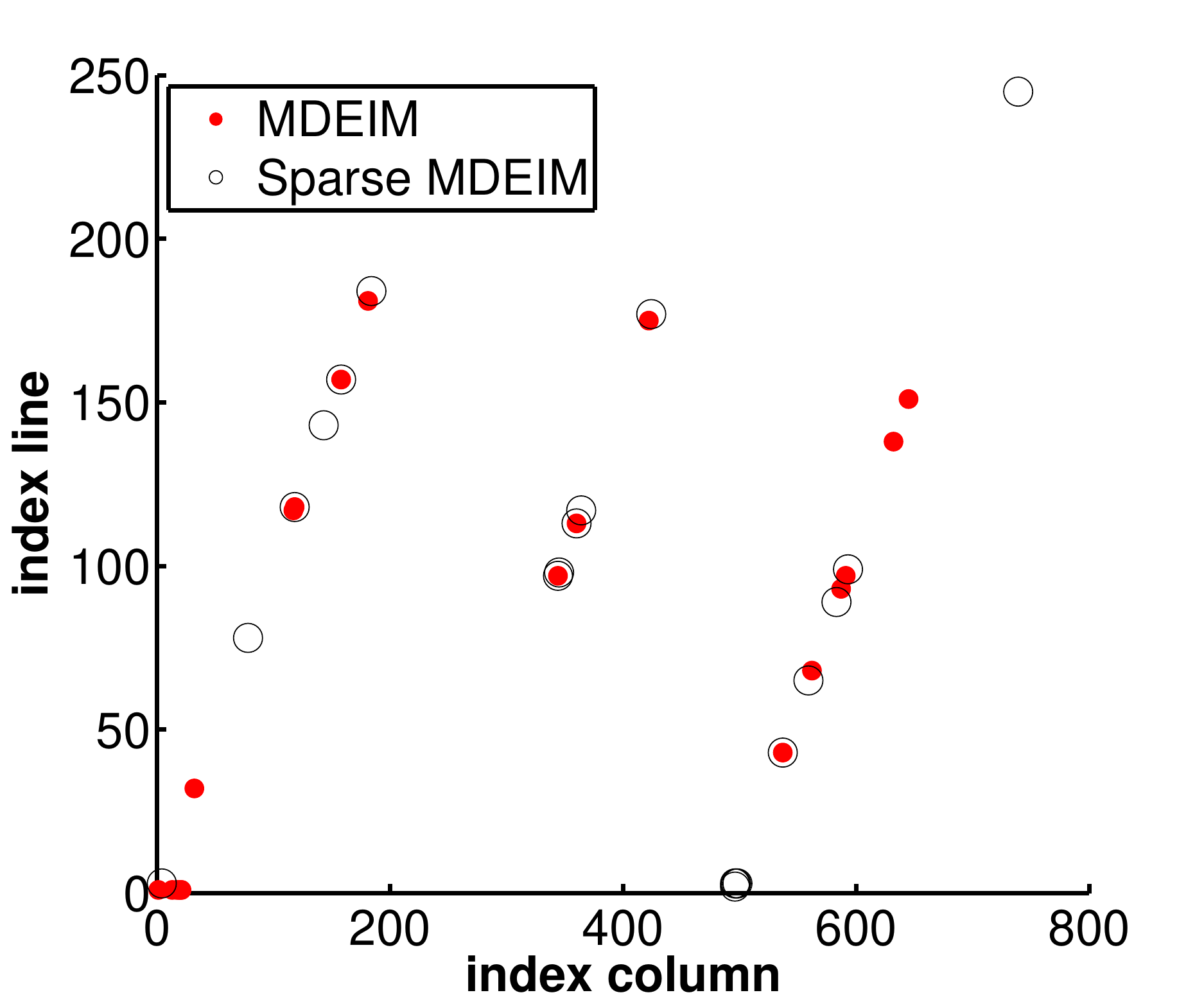}}
\caption{Localization of DEIM indexes using MDEIM and SMDEIM singular vectors 
\label{Fig::MDEIM_location_points_2D_SWE}}
\end{figure}

Next we assemble the approximations \eqref{eq::MDEIM}, \eqref{eq::sparseMDEIM} and \eqref{eq::Jacobian_DEIM_approx} of the x-derivative implicit 2D SWE model Jacobian
at the initial time step.  The Frobenius norm of the errors between the matrix approximations and its true representation is shown in Figure \ref{Fig::Full_Jacobian_error_2D_SWE} (a).  MDEIM and SMDEIM outputs present similar accuracy levels and the  mismatches Frobenius norms are inverse proportionally with the number of DEIM indexes. DEIM approximation is accurate only for the Jacobian rows selected by the interpolation indexes explaining the constant green trajectory even if the number of indexes is increased. The discrepancies between the largest singular value of the exact Jacobian and the greedy based approximations present  similar pattern as in the case of the Frobenius norms. This confirms that matrix DEIM approximations are more accurate than the DEIM proposed Jacobian.

\begin{figure}[h]
  \centering
  \subfigure[2D-SWE - $21\times 15$ space points ]{\includegraphics[scale=0.37]{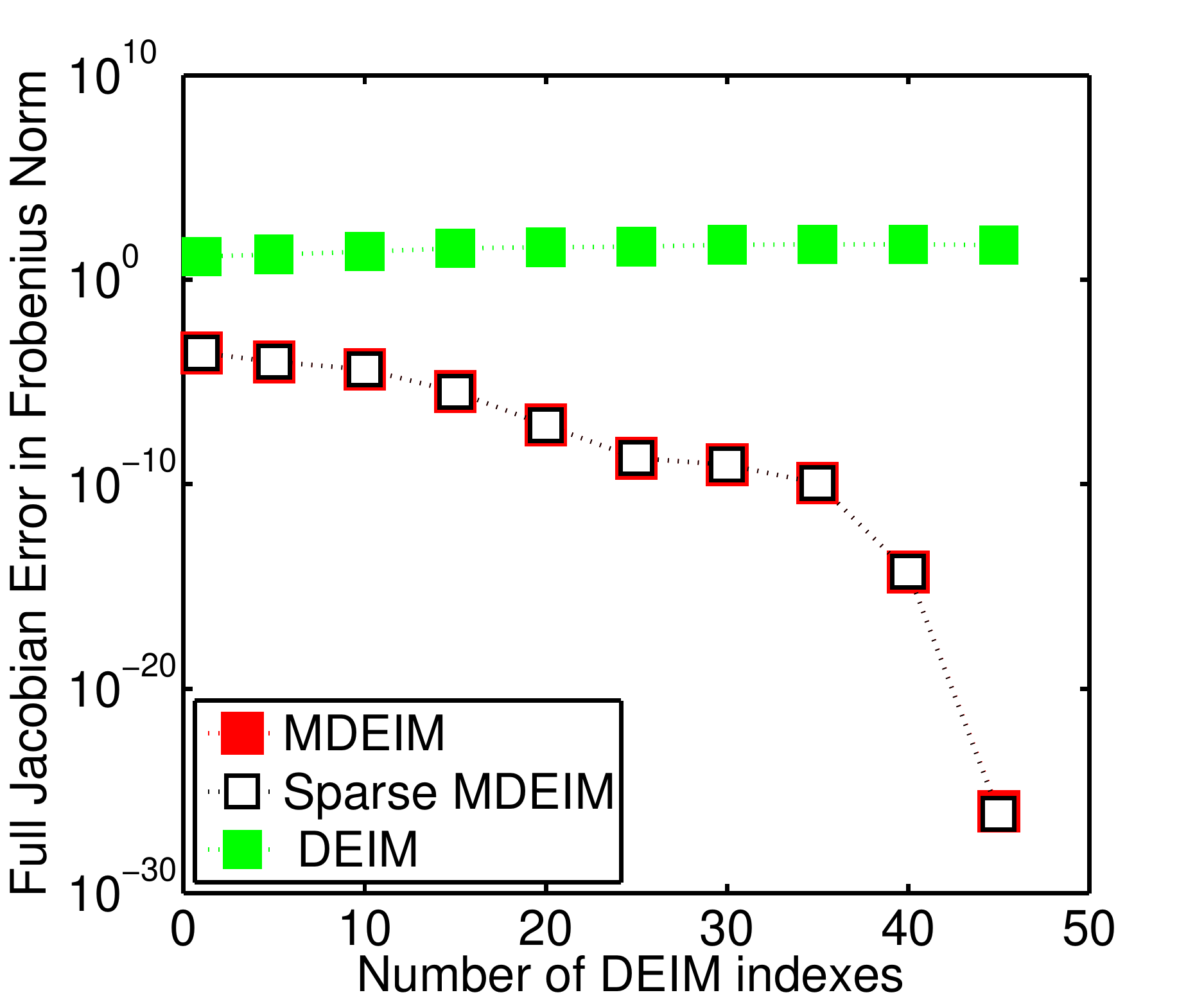}}
  \subfigure[2D-SWE - $21 \times 15$  space points ]{\includegraphics[scale=0.37]{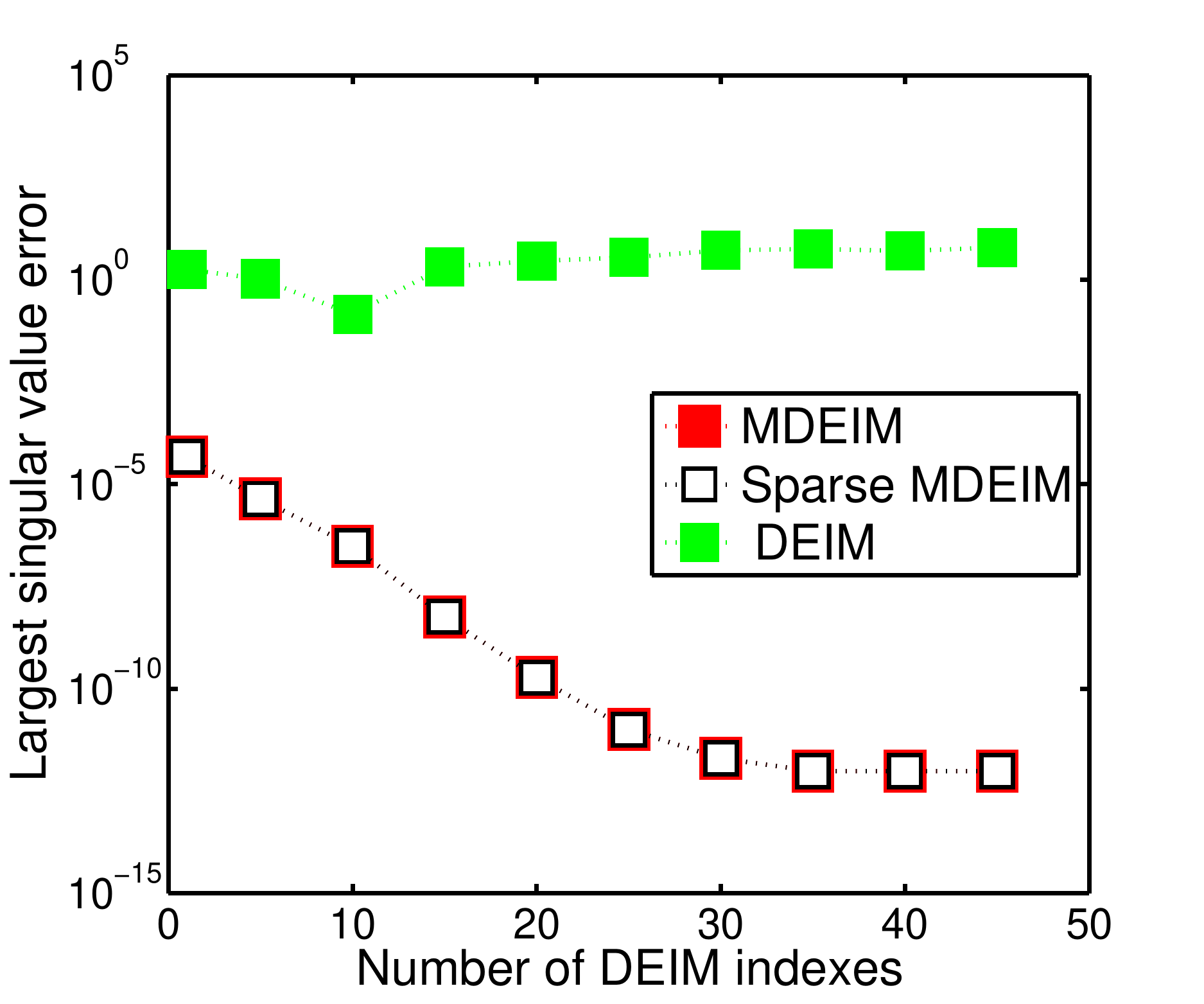}}
\caption{Full Jacobian errors at initial time - Frobenius norm - Largest SVD. 
\label{Fig::Full_Jacobian_error_2D_SWE}}
\end{figure}

\subsubsection{Performance of implicit reduced order models}

The proposed Jacobian approximations are embedded into POD reduced order framework using a Galerkin projection and the resulting reduced order models are compared in terms of computational cost and solution accuracy.  Most of the discussed results are obtained for a space resolution of $61\times 45$ points, with $\Delta x = \Delta y =100 km$. The models are integrated $6h$ in time and the number of time steps is set to $N_t=91$. The POD bases functions of the state variables, Jacobians matrices and nonlinear terms are constructed using 
$180$ snapshots obtained from the numerical solution of the full-order ADI finite difference SWE model. Figures \ref{Fig::Eigenvalues_state_2D_SWE}(a--b)  show the decay around the eigenvalues of the snapshot solutions for $u,~v,~\phi$ and the nonlinear snapshots $F_{11},~F_{12},$ $F_{21},~F_{22}$, $F_{31},~F_{32}$. The state variables spectra decrease faster than those of the nonlinear terms.

\begin{figure}[h]
  \centering
  \subfigure[Spectrum rate of decay] {\includegraphics[scale=0.4]{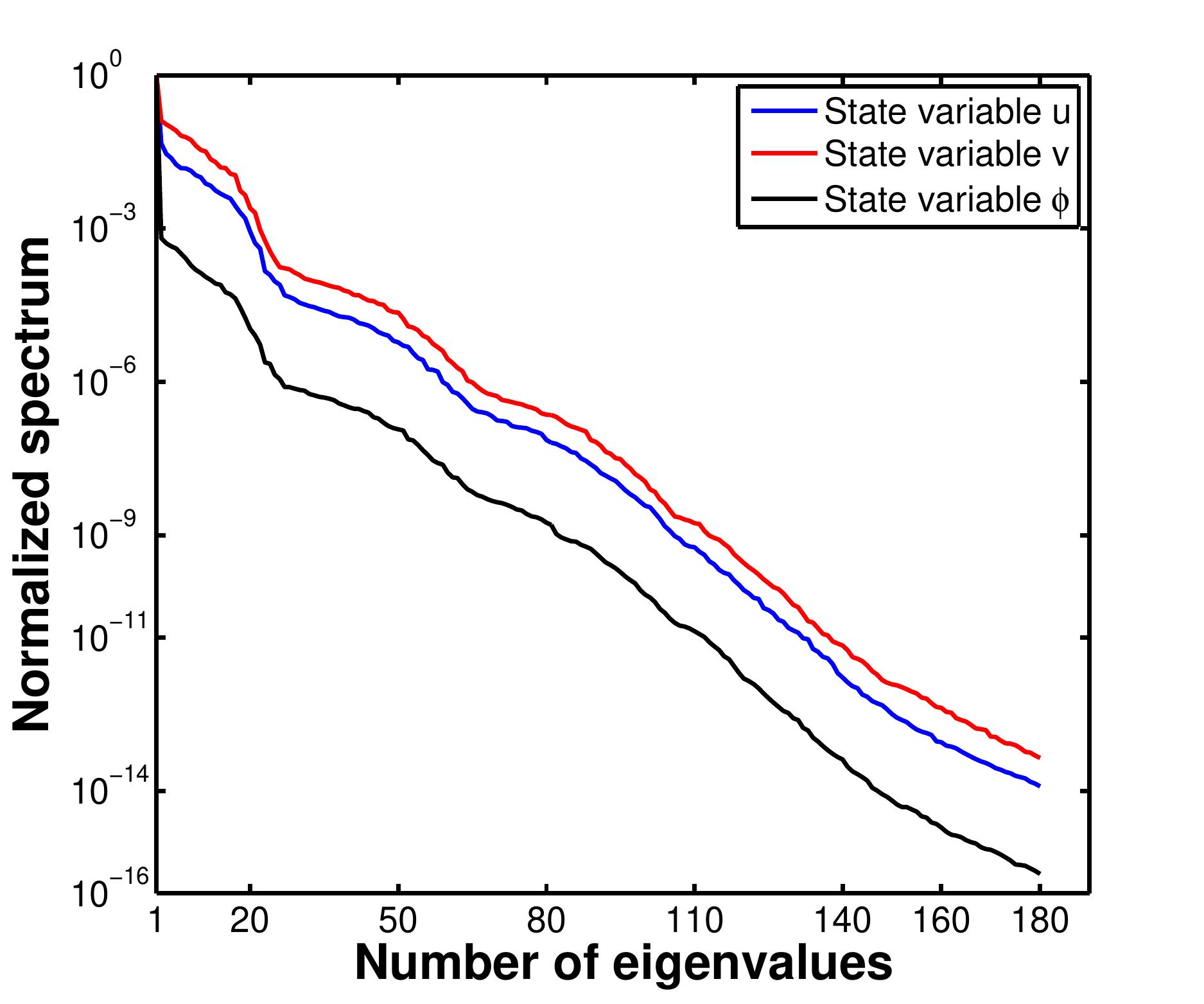}}
  \subfigure[Eigenvalues rate of decay]{\includegraphics[scale=0.4]{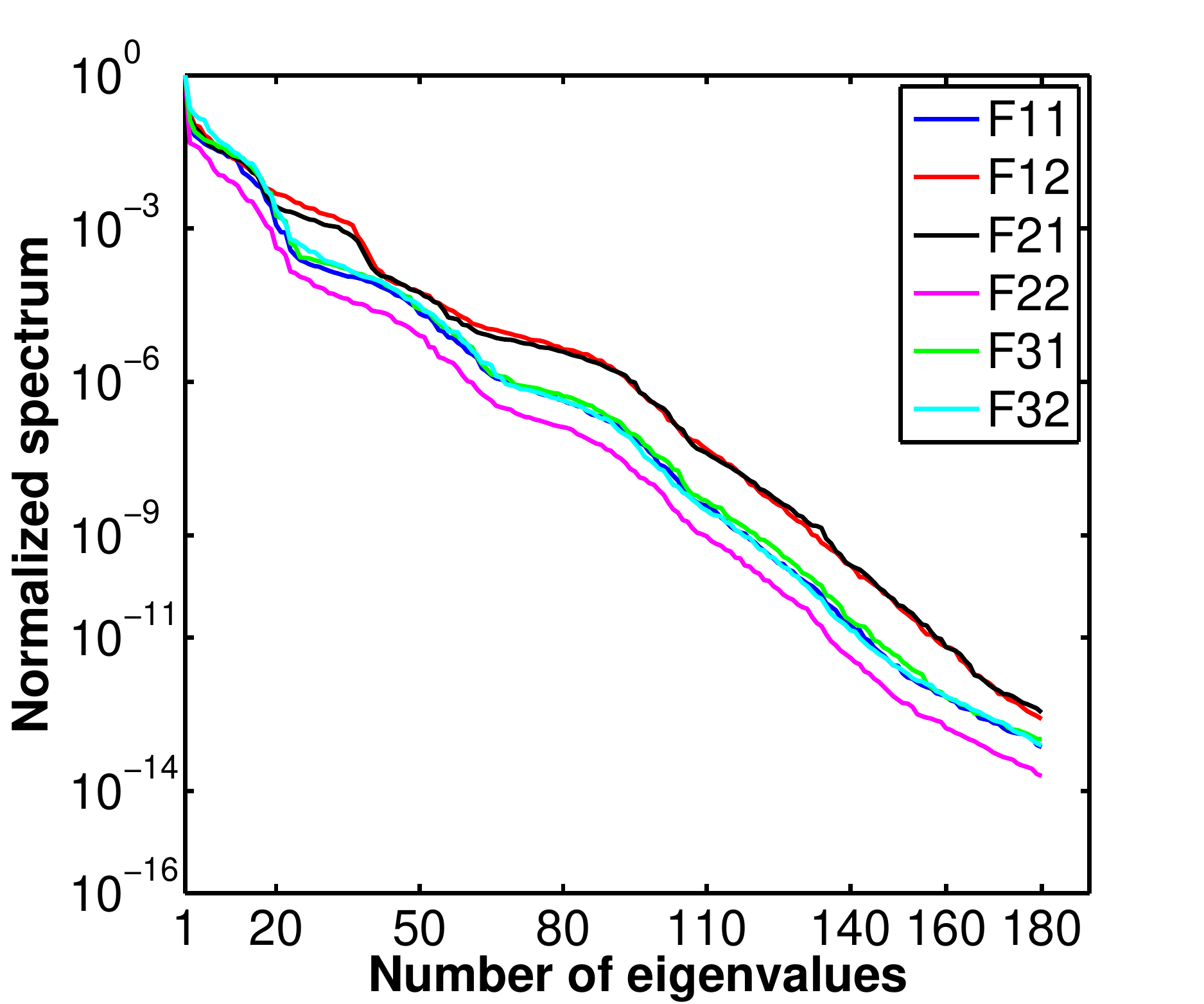}}
\caption{Spectrum properties of snapshots matrices
\label{Fig::Eigenvalues_state_2D_SWE}}
\end{figure}

\paragraph{Off-line computational performances}

During the off-line stages operations such as the singular value decomposition of the Jacobian snapshots matrices, calculation of interpolation indexes and computations
of matrices $U^TV(P^TV)^{-1}$, ${C} \cdot V_m\, \left(P_m^TV_m\right)^{-1}$ and 
${\tilde C} \cdot V_{m_{nz}}\Big(P_{m_{nz}}^TV_{m_{nz}}\Big)^{-1}$ in  \eqref{eq::Jacobian_DEIM_approx}, \eqref{eq::MDEIM}  and \eqref{eq::sparseMDEIM}
and tensor G \eqref{eq::tensor_g} are required only by $4$ of the discussed reduced order models described in subsection \ref{Sub::Red_Jacobian_comp}. These are the greedy based and tensorial surrogate models.  Figure \ref{Fig::Offline_CPU_comparisons_2DSWE} describes the total off-line computational time required for the on-line Jacobians evaluations as a function of the POD state dimension and number of interpolation indexes. In Figure \ref{Fig::Offline_CPU_comparisons_2DSWE}(a) we set the number of selected
interpolation indexes to $20$ and vary the dimensions of POD bases of state variables ${\bf u},~{\bf v}$ and ${\boldsymbol \phi} $ between $5$ and $50$. Tensorial approach has the smallest off-line cost however it becomes slower with the increase of the POD basis size. Among the greedy based techniques, SMDEIM computational effort is $400 \times$ times smaller than in the case of MDEIM method and only $9 \times$ times larger than the DEIM CPU time. It is worth mentioning the compromise proposed by SMDEIM method which manages to preserve the same level of accuracy as the MDEIM method at reasonable costs.  

Similar results are obtained if the number of DEIM interpolation indexes are varied as seen in Figure \ref{Fig::Offline_CPU_comparisons_2DSWE}(b). The number of POD bases functions is choose $20$ and the SMDEIM approach shows its efficiency being $1600$ times faster than the MDEIM method.  

\begin{figure}[h]
  \centering
  \subfigure[$m~=~20$;]{\includegraphics[scale=0.4]{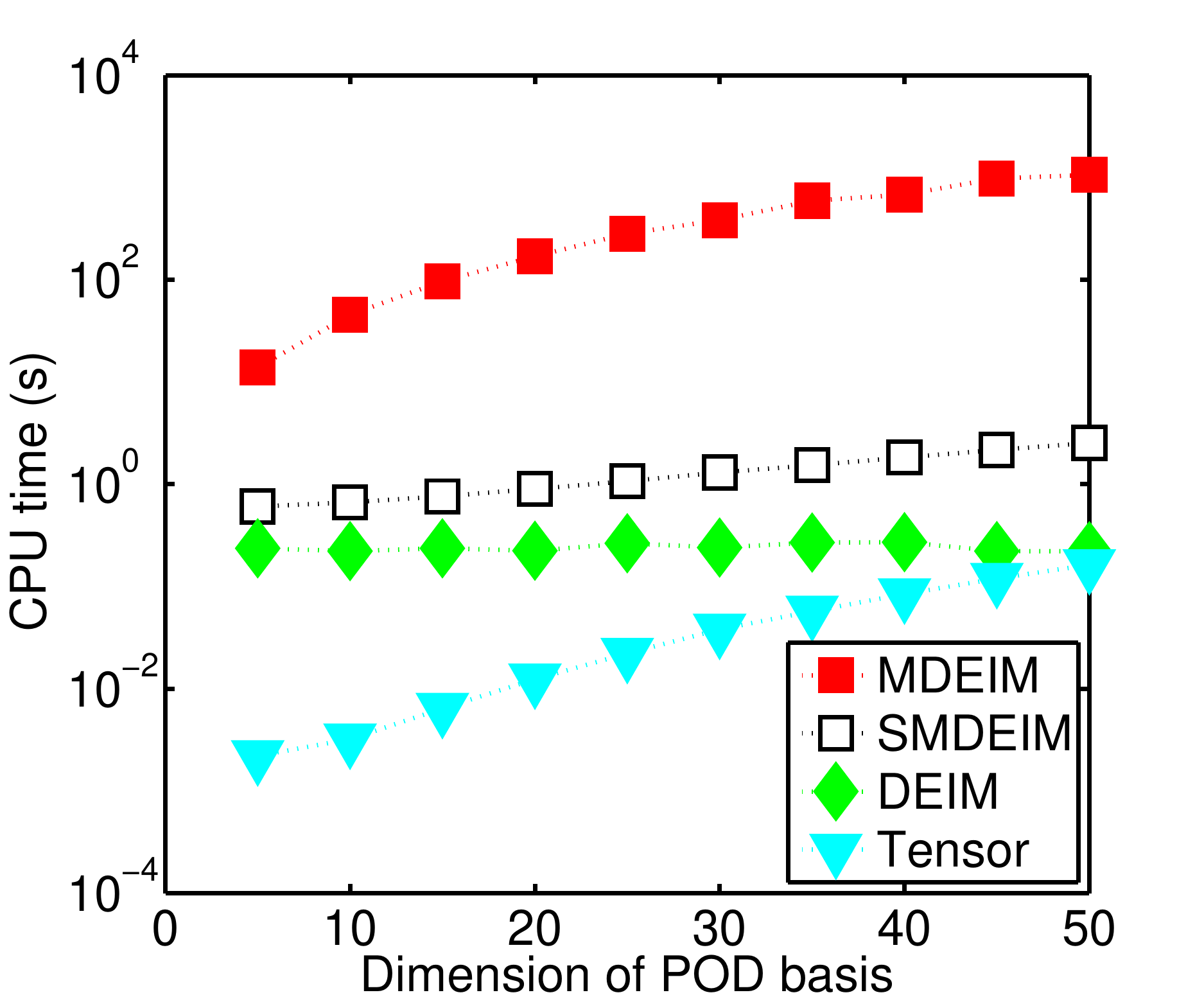}}
\subfigure[$k~=~20$;] {\includegraphics[scale=0.4]{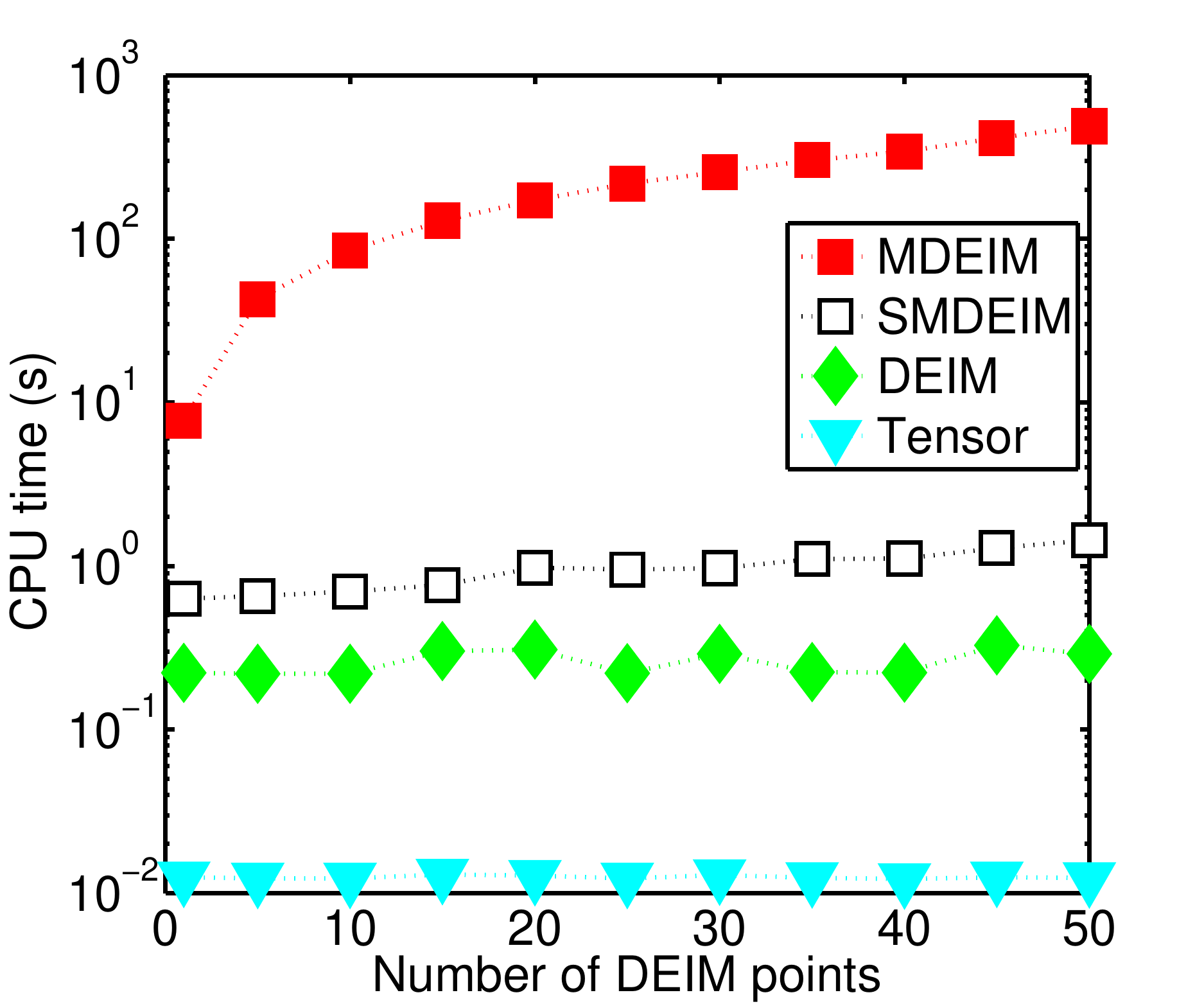}}
\caption{Off-line computational time performances of DEIM based and tensorial ROMs
\label{Fig::Offline_CPU_comparisons_2DSWE}}
\end{figure}

\paragraph{Reduced Jacobian errors}

The reduced residual Jacobian errors obtained using DEIM, MDEIM, SMDEIM,  directional derivative, direct projection and tensorial methods are
illustrated in Figure \ref{Fig::Red_Jac_errors_2D_SWE}, where Frobenius norm is employed. Initially we measure the reduced Jacobian errors with respect to the number of POD bases functions and
the results are presented in Figure \ref{Fig::Red_Jac_errors_2D_SWE}(a). The number of DEIM indexes is set to $30$. SMDEIM reduced
derivatives present similar levels of accuracy as the ones computed using the MDEIM method for $50$ POD basis functions. The direct projection and tensorial approaches present the most accurate reduced Jacobian. The directional derivative reduced derivatives
are calculated for $h=0.01$ in \eqref{eq::direct_deriv} and the precision is more than $3$ order of magnitude lower than in the case of SMDEIM, direct projection and tensorial methods. As expected the DEIM approximation is less accurate explained by the level of precision of its full Jacobian approximation depicted in Figure \ref{Fig::Full_Jacobian_error_2D_SWE}.

Figure \ref{Fig::Red_Jac_errors_2D_SWE}(b) describes the impact of the number of interpolation indexes onto the quality of reduced Jacobians. The dimension of
POD bases is set to $25$ and once the number of DEIM indexes is larger than $20$ the MDEIM and SMDEIM reduced Jacobians are almost as accurate as
the direct projection, directional derivative and tensorial Jacobians. In Figure \ref{Fig::Red_Jac_errors_2D_SWE}(b) by increasing the number of DEIM points to $50$ the quality of DEIM reduced Jacobian approximation is enhanced by one order of magnitude. 

\begin{figure}[h]
  \centering
  \subfigure[$m~=~30$;] {\includegraphics[scale=0.4]{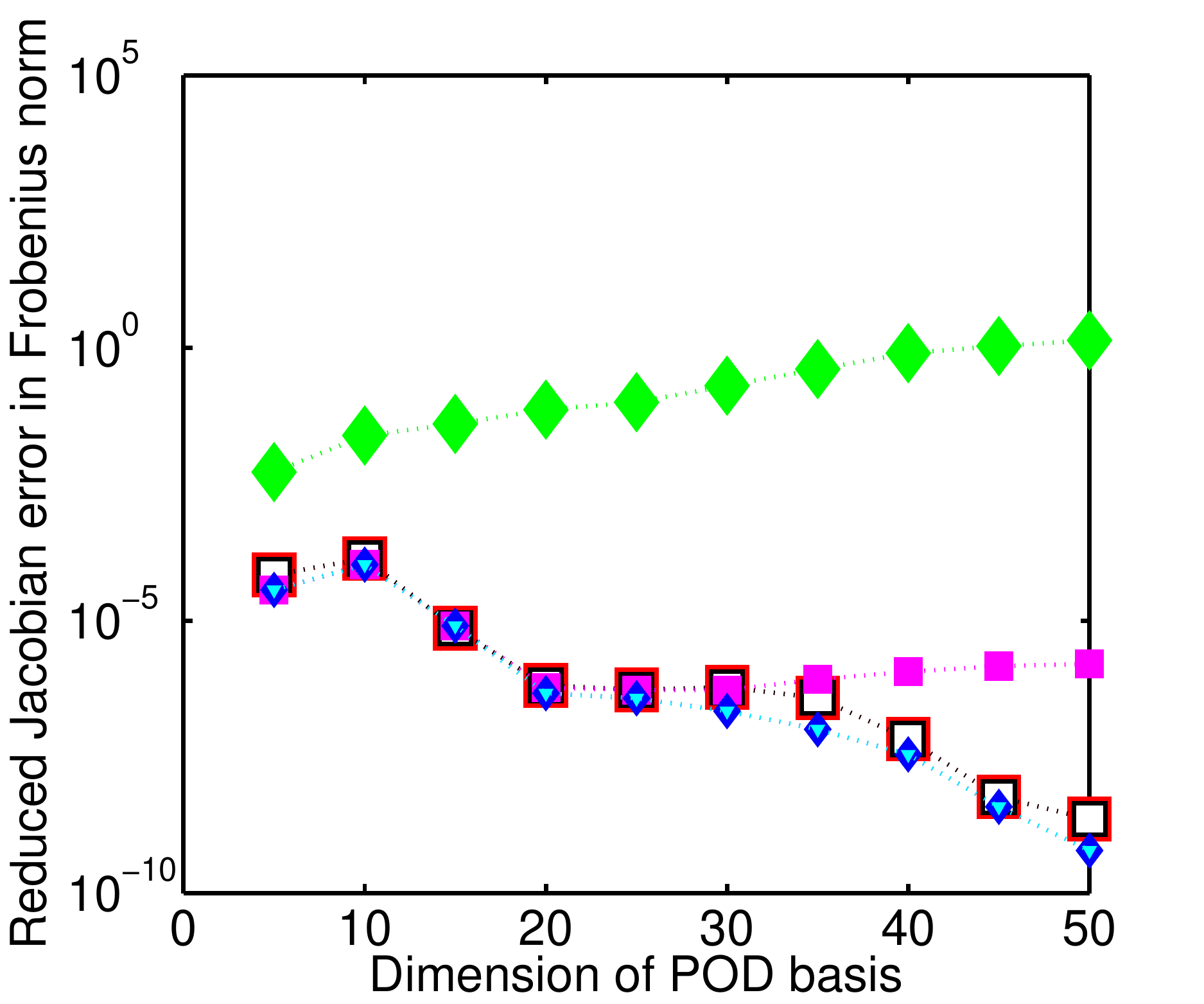}}
  \subfigure[$k~=~25$;]{\includegraphics[scale=0.4]{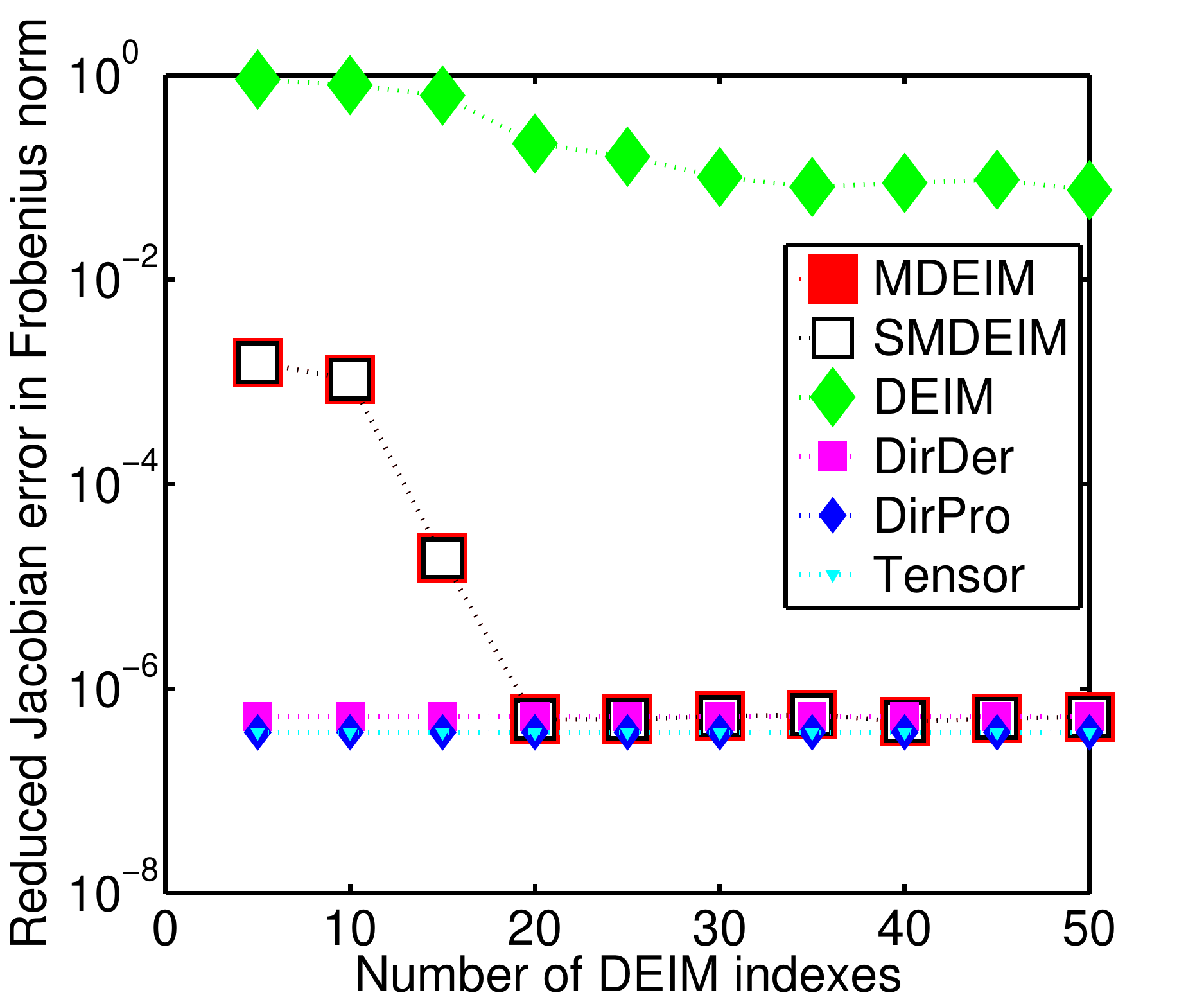}}
\caption{Reduced Jacobians errors
\label{Fig::Red_Jac_errors_2D_SWE}}
\end{figure}

\paragraph{On-line computational performances}

At this space resolution, MDEIM already requires storing snapshots vectors of size $57e+6$ and the cost of assemblying the Jacobian approximation fir $50$ DEIM points
is of order of hundred of seconds (see Figure \ref{Fig::Offline_CPU_comparisons_2DSWE}(a)). For $101 \times 71$ space points the computational complexity for building
the current version of MDEIM reduced order model becomes prohibitive. Since the planned experiments for this subsection include testing the performances of reduced 
order models as functions of space dimension it is computationally infeasible to run the MDEIM model.

Here we analyze the on-line computational CPU times obtained by the studied reduced order models with respect to the number of POD bases functions,
space points and DEIM indexes (see Figures \ref{Fig::Online_CPU_comparisons_2DSWE}). Since both direct projection and directional derivative reduced Jacobian computations depend on the full space dimension the corresponding reduced order
models do not gain much efficiency sometimes being slower even than the high-fidelity model. Among them, the directional derivative approach is faster. Initially we fix the number of DEIM points to $30$ and start increasing the
size of POD basis. For smaller values of $k$, tensorial approach leads to the fastest surrogate model as noticed in Figure \ref{Fig::Online_CPU_comparisons_2DSWE}(a).
Tensorial POD has a theoretical computational complexity depending on $k$ and for POD size of $50$ we notice that
SMDEIM is $1.25\times $ times faster than the tensorial calculus based reduced order model thus confirming the theoretical results. 
The DEIM based surrogate model is $2.75\times$ slower than the SMDEIM model explained by their corresponding computational complexities in table \ref{table_complexityI}
and the lack of Jacobian accuracy noticed in Figure \ref{Fig::Red_Jac_errors_2D_SWE}. The latter forces the corresponding reduced order model to increase the number 
of Newton iterations (see table \ref{table_NR_dim_POD_2D_SWE}) in order to achieve a similar level of solution accuracy as the other surrogate models. 

\begingroup
\begin{table}  
\begin{center} 
\begin{tabular}[h]{|c|c|c|c|c|c|c|c|c|c|c|} \hline 
 POD basis dimension & 5 & 10 & 15 & 20 & 25 & 30 & 35 & 40 & 45 & 50 \cr \hline \hline 
SMDEIM & 3.00 & 3.00 & 3.00 & 3.00 & 3.00 & 3.00 & 3.00 & 3.00 & 3.00 & 3.00 \cr \hline 
DEIM & 4.00 & 5.00 & 5.00 & 5.00 & 5.00 & 6.00 & 7.00 & 8.00 & 9.00 & 9.00 \cr \hline 
DirDer/DirPro/Tensor & 3.00 & 3.00 & 3.00 & 3.00 & 3.00 & 3.00 & 3.00 & 3.00 & 3.00 & 3.00 \cr \hline 
Full & 4.00 & 4.00 & 4.00 & 4.00 & 4.00 & 4.00 & 4.00 & 4.00 & 4.00 & 4.00 \cr \hline 
\end{tabular} 
\end{center} 
\caption{The mean variation of Newton-Raphson iterates per time step along the change in the POD basis dimension}\label{table_NR_dim_POD_2D_SWE} 
\end{table}  
\endgroup%

Next we set the number of DEIM indexes to $30$ and the dimension of POD bases to $50$. Figure \ref{Fig::Online_CPU_comparisons_2DSWE}(b) 
describes the computational costs of the discussed reduced order models as a function of number of space points. For more than $10^5$ mesh points we notice 
that reduced order SMDEIM model is $1.24,~2.12,~309,~344$ and $400$ times
faster than tensorial, DEIM, directional derivative, direct projection and high-fidelity models. Clearly there is no advantage of using the direct projection and directional derivative approaches
from computational complexity point of view. However they are useful due to their non-intrusive nature making them easy to implement even for very complex models. 

The number of averaged Newton iterations per time step for DEIM reduced order model increases with the space dimensions while for SMDEIM, tensorial, directional derivative
and direct projection it remains constant. The SMDEIM and tensorial reduced order models efficiency is also a consequence of the reduced number of Newton iterations
in addition to the reduced number of degrees of freedom as shown in table \ref{table_NR_dim_space_2D_SWE}.

\begin{figure}[h]
  \centering
  \subfigure[$m~=~ 30$;]{\includegraphics[scale=0.27]{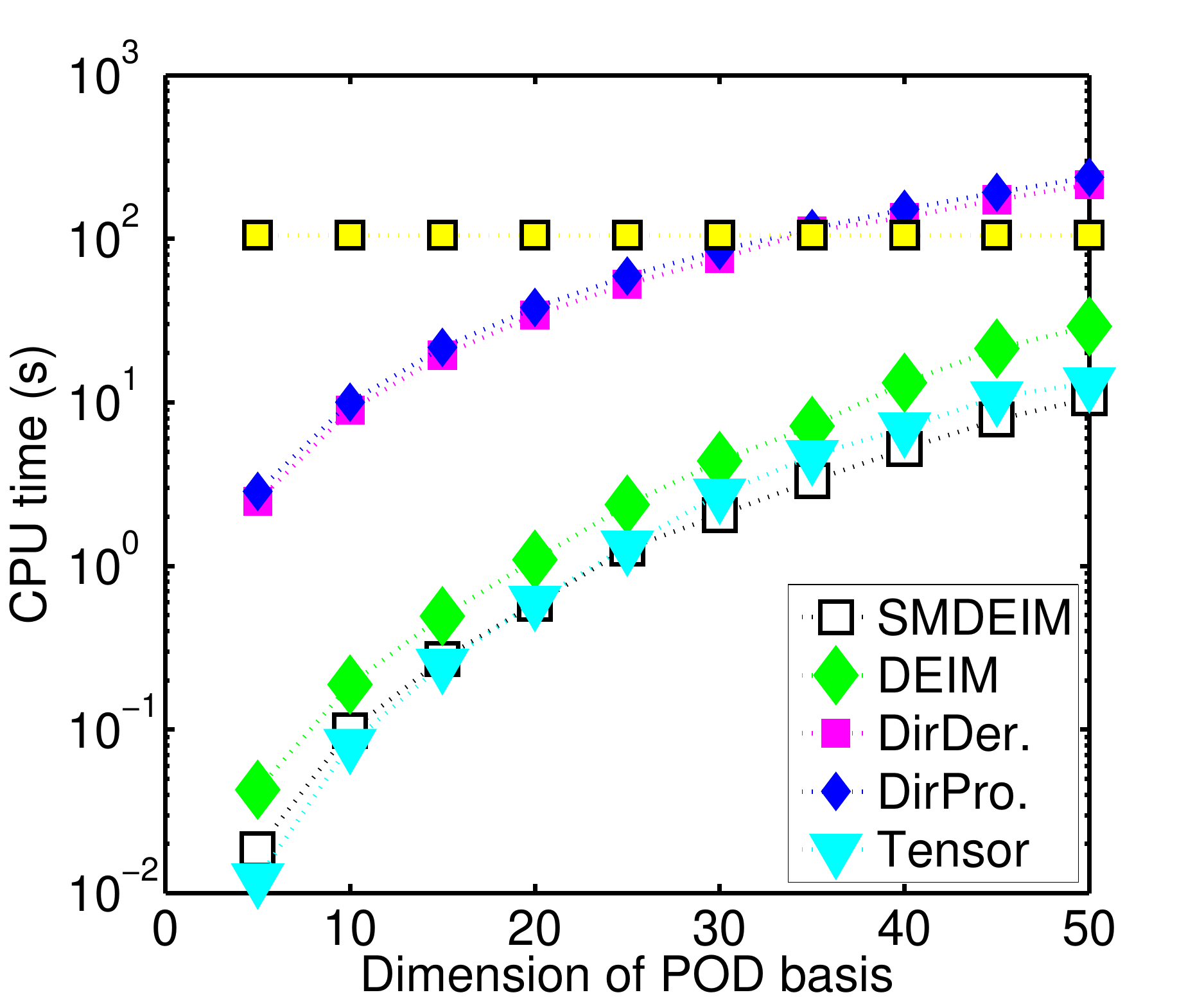}}
  \subfigure[$m~ =~30,~k =~ 50$;]{\includegraphics[scale=0.27]{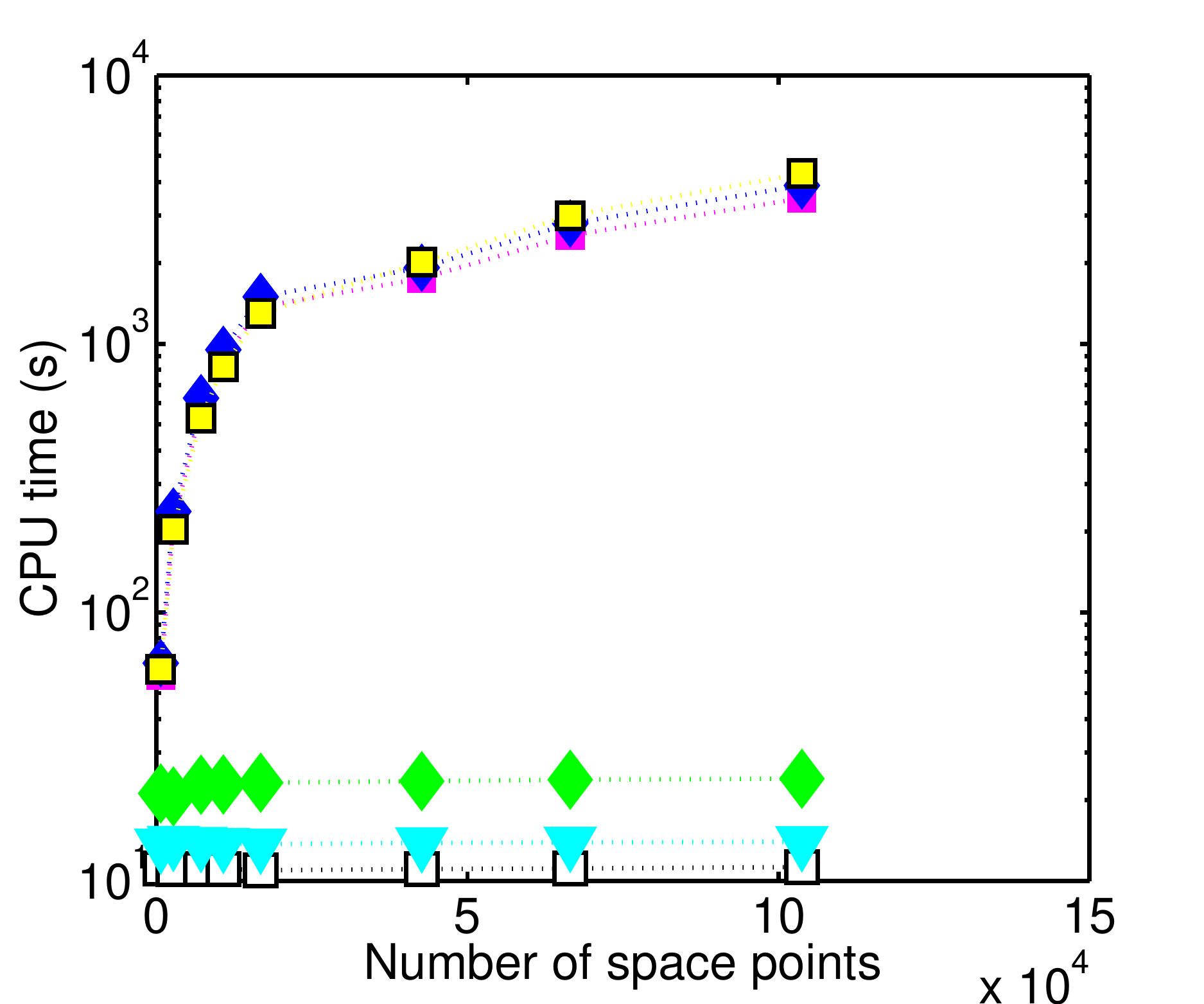}}
\subfigure[$,k~ = ~25$;] {\includegraphics[scale=0.27]{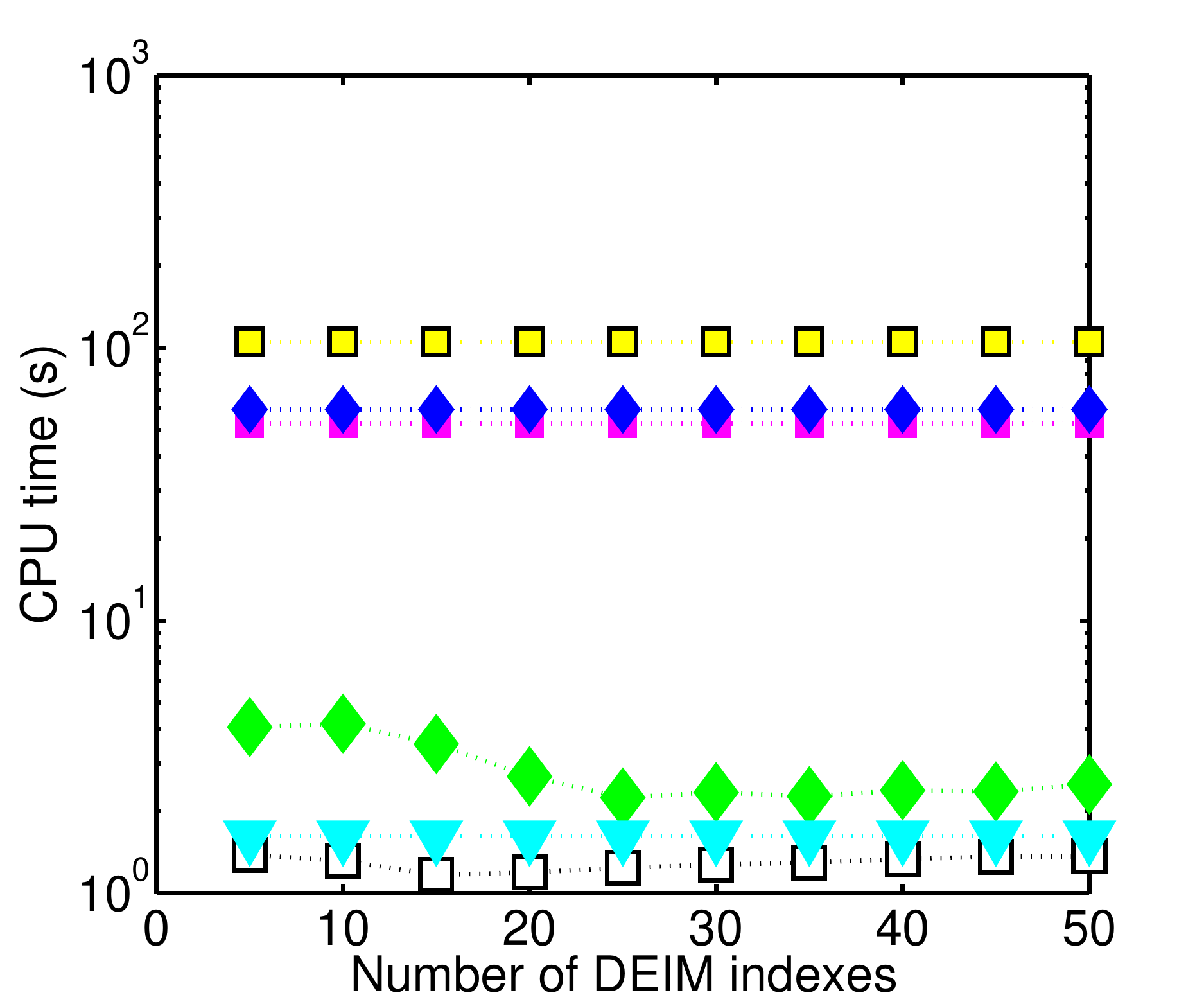}}
\caption{On-line computational time performances of 2D-SWE reduced order models
\label{Fig::Online_CPU_comparisons_2DSWE}}
\end{figure}

\begingroup
\begin{table}  
\begin{center} 
\begin{tabular}[h]{|c|c|c|c|c|c|c|c|c|} \hline 
 No of space points & 713 & 2745 & 7171 & 10769 & 16761 & 42657 & 66521 & 103776  \cr \hline \hline 
SMDEIM & 3.00 & 3.00 & 3.00 & 3.00 & 3.00 & 3.00 & 3.00 & 3.00   \cr \hline 
DEIM & 6.00 & 6.00 & 7.00 & 7.00 & 7.00 & 8.00 & 8.00 & 8.00  \cr \hline 
DirDer/DirPro/Tensor & 3.00 & 3.00 & 3.00 & 3.00 & 3.00 & 3.00 & 3.00 & 3.00  \cr \hline 
Full & 5.00 & 4.00 & 5.00 & 6.00 & 6.00 & 5.00 & 5.00 & 6.00  \cr \hline 
\end{tabular} 
\end{center} 
\caption{The mean variation of Newton-Raphson iterates per time step along the change in the space dimension}\label{table_NR_dim_space_2D_SWE} 
\end{table}  
\endgroup%

Figure \ref{Fig::Online_CPU_comparisons_2DSWE}(c) shows the on-line CPU times of the reduced order models for various numbers of 
DEIM indexes. We notice that the computational cost of DEIM reduced order model is decreasing and this behaviour is explained by the reduction in 
the average Newton iterations number per time steps (see table \ref{table_NR_DEIM_points_2D_SWE}). This is a consequence of the Jacobians accuracy gain once the number
of DEIM points is increased. 
\begingroup
\begin{table}  
\begin{center} 
\begin{tabular}[h]{|c|c|c|c|c|c|c|c|c|c|c|} \hline 
 No of DEIM indexes & 5 & 10 & 15 & 20 & 25 & 30 & 35 & 40 & 45 & 50 \cr \hline \hline 
SMDEIM & 3.00 & 3.00 & 3.00 & 3.00 & 3.00 & 3.00 & 3.00 & 3.00 & 3.00 & 3.00 \cr \hline 
DEIM & 13.00 & 12.00 & 10.00 & 6.00 & 6.00 & 5.00 & 5.00 & 5.00 & 5.00 & 5.00 \cr \hline 
DirDer/DirPro/Tensor & 3.00 & 3.00 & 3.00 & 3.00 & 3.00 & 3.00 & 3.00 & 3.00 & 3.00 & 3.00 \cr \hline 
Full & 4.00 & 4.00 & 4.00 & 4.00 & 4.00 & 4.00 & 4.00 & 4.00 & 4.00 & 4.00 \cr \hline 
\end{tabular} 
\end{center} 
\caption{The mean variation of Newton-Raphson iterates per time step along the change in the number of DEIM indexes}\label{table_NR_DEIM_points_2D_SWE} 
\end{table}  
\endgroup%

For all the experiments, the proposed reduced order solutions present similar accuracy levels obtained using various numbers of Newton iterations as seen in tables
across this subsection.  For different models, where higher nonlinearities are presented, we anticipate a different scenario where the reduced models errors 
are more sensitive to the quality of reduced order Jacobian. These scenarios advocate the use of a SMDEIM reduced order model since it will be much faster than 
tensorial model and more accurate than DEIM based reduced order model. 

The computational savings and accuracy levels obtained by the SMDEIM reduced order model depend on the number of
POD modes and number of DEIM interpolation indexes. These numbers may be large in practice in order to capture well the full model dynamics.
Local POD and local DEIM versions were proposed by \citet{Rapun_2010} and \citet{Peherstorfer_2013} to alleviate this problem. A similar
strategy can be applied for Jacobian approximation thus improving the performances of SMDEIM models. The idea of a local approach for nonlinear model reduction
with local POD and local GNAT was first proposed by \citet{Amsallem_2012}. Machine learning techniques such as $K$-means \citep{Lloyd_1957,MacQueen_1967,Steinhaus_1956}  can be used for both time and space partitioning.
A recent study investigating cluster-based reduced order modeling was proposed by \citet{Kaiseretal2013}. First order and second order adjoint methodologies can be
employed to compute useful error indicators \citep{Rao20141256}  in the context of building statistically modeling errors 
\citep{Roderick_et_al_2013,Drohmann_Carlberg_2013} to enhance the SMDEIM reduced order model solutions.

\section{Conclusions}\label{sec:Conclusions}

In the POD Galerkin approach to reduced order modeling the cost of evaluating nonlinear terms and their derivatives during
the on-line stage scales with the full space dimension, and this constitutes a major efficiency bottleneck. This work introduces the sparse matrix discrete empirical interpolation method to compute accurate approximations of parametric matrices with constant sparsity structure. The approach is employed in a reduced order modeling  framework to efficiently obtain  accurate reduced order Jacobians. The sparse algorithm utilizes samples of of only the nonzero entries of the matrix series. The economy SVD factorization of the nonzero elements of the snapshots matrix, when appropriately padded with zeros, is equivalent to a valid economy SVD for the full snapshots matrix. The sparse SVD is applied to much smaller vectors and is therefore considerably more efficient.

In contrast with the traditional matrix DEIM method, the sparse version can be applied to much larger problems since its off-line computational complexity  depends on the number of nonzero Jacobian elements, dimension of state POD basis, and the number of DEIM indexes. The on-line cost for computing reduced  derivatives using SMDEIM is similar with the cost required by MDEIM.

An important application of SMDEIM is the construction of reduced order implicit time integration schemes since many problems arising in practice are stiff and so are the corresponding reduced order models.

Several strategies for Jacobian computations are implemented and the performance of the corresponding reduced order models is analyzed.  These strategies are based on discrete empirical interpolation method applied to function, tensorial calculus, the full Jacobian projection onto the reduced  basis subspace, and directional derivatives. Numerical experiments are carried out using the one dimensional Burgers equation and two dimensional SWE models. The construction of reduced order models is based on proper orthogonal decompositions and Galerkin projection, and the reduced nonlinearities are constructed using tensorial calculus. In the majority of experiments SMDEIM provides the fastest on-line reduced order model. For $10^5$ mesh points the 
reduced order SMDEIM SWE model is $1.24,~2.12,~309,~344$, and $400$ faster than the tensorial, DEIM, direct projection, directional derivative, and high-
fidelity models, respectively. For this space configuration the memory burden of the off-line stage of the traditional MDEIM reduced order model exceeded 
our computational resources - it was not possible to store $n^2$ dimensional vectors as required, with $n$ representing the number of discrete variables.
While being slower, the direct projection and directional derivative methods propose non-intrusive reduced order models. The numerical results showed that
DEIM Jacobians approximations have a lower accuracy when using samples from the underlying functions
instead of the Jacobian matrices. Therefore a larger number of Newton iterations are equired by the DEIM reduced order model to produce solutions as accurate as those of the other discussed surrogate models. 

Future research will focus on decreasing the temporal complexity of the DEIM implicit scheme by exploiting the knowledge of the model's temporal  behavior as proposed in \citet{Carberg_etal_2012}. Forecasting the unknown variables of the reduced-order system 
of nonlinear equations at future time steps provides an initial guess for the Newton-like solvers and can significantly decrease the number of linear systems solved at each step. 

For the current test problems the reduced order models solutions are not very sensitive to the quality of the reduced order Jacobian approximations.  However, for higher nonlinearities we anticipate a different behavior where the model errors are affected more by the quality of reduced order derivatives. For these cases, the SMDEIM reduced model is expected to be much faster than the tensorial model and more accurate than the DEIM surrogate model.

As future application we propose the use of SMDEIM reduced Jacobian approximation for strongly coupled fluid-structure problems (see \citet{Vierendeels2007}) to approximate the Jacobian of the fluid and/or structural problem during the coupling iterations. On-going work by the authors focuses on reduced order constrained optimization. The current research represents an important step toward implementing the SMDEIM method for solving a reduced order optimal control problem such as the one discussed by \citet{Negri_et_al_2013}.

{\centering
\section*{Acknowledgments}
}
The work of Dr. R\u azvan Stefanescu and Prof. Adrian Sandu was supported by awards NSF CCF--1218454, NSF DMS--1419003, AFOSR FA9550--12--1--0293--DEF, AFOSR 12-2640-06, and by the Computational Science Laboratory at Virginia Tech.

\newpage
\section*{References}
\bibliographystyle{plainnat}
\bibliography{Bib/Software,Bib/ROM_state_of_the_art,Bib/POD_bib,Bib/CDS_E_proposal,Bib/sandu,Bib/comprehensive_bibliography1,Bib/Razvan_bib,Bib/Razvan_bib_ROM_IP,Bib/NSF_KB,Bib/data_assim_weak-fdvar,Bib/reduced_models}

\end{document}